\documentclass[pdftex]{amsart} 

\usepackage{array}
\usepackage{enumerate}
\usepackage[margin=2.0cm]{geometry}
\usepackage[active]{srcltx}
\usepackage{amsfonts,amssymb,amsmath,latexsym,upref}
\usepackage{verbatim} 
\usepackage[english]{babel}
\usepackage[latin1]{inputenc}  
\usepackage{color}

\usepackage{tikz}
\usepackage{pgf}
\usetikzlibrary{shapes, intersections, calc, patterns, positioning,
  matrix, arrows, fit, backgrounds, 
  decorations.pathmorphing, decorations.markings}

\newtheorem{lemma}[equation]{Lemma}
\newtheorem{prop}[equation]{Proposition}
\newtheorem{thm}[equation]{Theorem}
\newtheorem{cor}[equation]{Corollary}

\newtheorem{defn}[equation]{Definition}
\newtheorem{conj}[equation]{Conjecture}

\theoremstyle{definition}
\newtheorem{exmp}[equation]{Example}

\newtheorem{rmk}[equation]{Remark}

\newtheorem{notation}[equation]{Notation}

\numberwithin{equation}{section}

\newcommand{\Z}{\mathbf{Z}}

\newcommand{\Q}{\mathbf{Q}}
\newcommand{\C}{\mathbf{C}}
\newcommand{\F}{\mathbf{F}}
\newcommand{\R}{\mathbf{R}}

\newcommand{\grp}[1]{\varphi_2(#1)} 

\newcommand{\SOodd}[2]{\operatorname{SO}_{#1}(#2)}
\newcommand{\SOeven}[3]{\operatorname{SO}^{#1}_{#2}(#3)}

\newcommand{\GL}[3][{}]{\operatorname{GL}^{#1}_{#2}(#3)}
\newcommand{\SL}[3][{}]{\operatorname{SL}^{#1}_{#2}(#3)}

\newcommand{\Sp}[2]{\operatorname{Sp}_{#1}(#2)}

\newcommand{\Ob}[1]{\mathrm{Ob}(#1)}

\newcommand{\cat}[3]{\mathcal{#1}_{#2}^{#3}}  
\newcommand{\catp}[4][p+]{\mathcal{#2}_{#3}^{#1#4}}  

\newcommand{\m}{morphism}

\newcommand{\gen}[1]{{\langle}#1{\rangle}}

\newcommand{\syl}[1]{Sylow $#1$-subgroup}
\newcommand{\we}{weighting}
\newcommand{\Mb}{M\"obius}
\newcommand{\Euc}{Euler characteristic}
\newcommand{\rchi}{\widetilde{\chi}}
\newcommand{\ralpha}{\widetilde{\alpha}}
\newcommand{\wa}{\widetilde{a}}
\newcommand{\Syl}[2]{\operatorname{Syl}_{#1}(#2)}

\newcommand{\sd}[1]{\operatorname{sd}(#1)}

\newcommand{\B}{\operatorname{B}}

\newcommand{\mynote}[1]{\noindent{\textcolor{red}{\textbf{[#1]}}}}

\newcommand{\Irr}[1]{\operatorname{Irr}(#1)}
\newcommand{\KB}{\operatorname{KB}}
\newcommand{\OP}{\operatorname{OP}}

\title{Euler characteristics of centralizer subcategories}

\author{Jesper M.~M\o ller}
\address{Institut for Matematiske Fag\\
  Universitetsparken 5\\
  DK--2100 K\o benhavn}
\email{moller@math.ku.dk}
\urladdr{htpp://www.math.ku.dk/~moller}

\thanks{Supported by the Danish National Research Foundation through
  the Centre for Symmetry and Deformation (DNRF92) and by Villum
  Fonden through the project Experimental Mathematics in Number
  Theory, Operator Algebras, and Topology} \thanks{\tiny
  /home/moller/projects/euler/orbit/alperin03.tex}

\usepackage[bookmarks=true,bookmarksopen=false]{hyperref}
\hypersetup{
   pdftitle = {},
   pdfauthor = {Jesper Michael Møller},
   pdfpagemode = {UseOutlines},
   pdfstartview = {FitH},
   pdfborder = {0 0 0},
   backref = {true},
   colorlinks = {true},
   urlcolor = {blue},
   citecolor = {blue},
   linkcolor = {blue},
   pdftoolbar = {true},
}

\subjclass[2000]{05E15, 20J15} 
\keywords{\Euc, Brown poset, Quillen poset, Bouc
  poset, Alperin Weight Conjecture, Kn\"orr--Robinson Conjecture} 

\begin{document}
\date{\today}
\maketitle
\tableofcontents

\begin{abstract}
  Let $G$ be a finite group and $p$ a prime number.  By a theorem of
  K.S.\@ Brown, the reduced \Euc\ of $\cat SG{p+*}$, the poset of
  nonidentity $p$-subgroups of $G$, is a multiplum of $|G|_p$, the
  $p$-part of the group order.  We prove here an equivariant version
  of Brown's theorem for the case where $G$ is equipped with an
  action of a finite group $A$. The equivariant version asserts that
  the reduced \Euc\ of $C_{\cat SG{p+*}}(A)$, the poset of nonidentity
  $A$-normalized $p$-subgroups of $G$, is a multiplum of $|C_G(A)|_p$,
  the $p$-part of the centralizer order.  

  We also examine the higher equivariant \Euc s $\chi_r(\cat
  SG{p+*},G)$, $r \geq 0$, introduced by Atiyah and Segal, for the
  conjugation action of $G$ on $\cat SG{p+*}$. We observe that the
  $r$th equivariant \Euc s can be computed in the class function space
  from a virtual rational representation of $G$, the $r$th Euler class
  function. The second equivariant \Euc\ is especially interesting in
  connection with the Kn\"orr--Robinson conjecture. 
\end{abstract}

\section{Introduction}
\label{sec:intro}


Let $A$ be a finite group and $\cat C{}{}$ a finite $A$-category.  For
any subset $X$ of the action group $A$, the $X$-centralized
subcategory, $C_{\cat C{}{}}(X)$, is the subcategory of $\cat C{}{}$
consisting of all $\cat C{}{}$-\m s $c \xrightarrow{\varphi} d$ such
that $c^x=c$, $d^x=d$ and $\varphi^x= \varphi$
for all endofunctors $x$ in $X$.  Atiyah and Segal
\cite{atiyah&segal89} define the $r$th, $r \geq 0$, equivariant \Euc\
to be the normalized sum
\begin{equation*}
  \chi_r(\cat C{}{},A) = \frac{1}{|A|} \sum_{X \in C_r(A)}
  \chi(C_{\cat C{}{}}(X))
\end{equation*}
of the \Euc s of the $X$-centralized subcategories as $X$ ranges over
the set $C_r(A)$ of commuting $r$-tuples in $A$.  We are implicitly
assuming for instance that $\cat C{}{}$ is an $\mathrm{EI}$-category
so that $X$-centralized subcategories have \Euc s in the sense of
Leinster \cite{leinster08}.

Alternatively, the equivariant \Euc s may  be computed in
the virtual rational representation ring $\Q \otimes R_\Q(A)$ of $A$:
The $r$th equivariant \Euc\ is the character inner product
\begin{equation*}
   \chi_r(\cat C{}{},A) = \gen{\alpha_r(\cat C{}{},A),|C_A|}_A
\end{equation*}
of the $r$th Euler class function of $(\cat C{}{},A)$,
\begin{equation*}
  \alpha_r(\cat C{}{},A) \colon A \to \Q, x \to \chi_{r-1}(C_{\cat
    C{}{}}(x),C_A(x)), \qquad r \geq 1, 
\end{equation*}
and the permutation character $|C_A|$ for the action of $A$ on itself
by conjugation (Definition~\ref{defn:classfct}).

An especially interesting case of this general set-up arises when $G$
a finite group, $p$ is a prime, and $\cat C{}{} = \cat SG{p+*}$ is the
Brown poset of nontrivial $p$-subgroups of $G$, and $A=G$ acts by
conjugation.  The $G$-poset $\cat SG{p+*}$ has equivariant \Euc s
$\chi_r(\cat SG{p+*},G)$ defined for every $r \geq 0$.  We shall
mainly focus on the cases $r=0,1,2$ as the significance of the higher
equivariant \Euc s for $r \geq 3$ remains unclear
(Remark~\ref{exmp:3conjs}).  It seems that these three first
equivariant \Euc s carry some interesting information.  For $r=0$ it
is immediate from the definition that the equivariant \Euc\ of the
Brown poset
\begin{equation*}
  \chi_0(\cat SG{p+*},G) = \chi(\cat SG{p+*})/|G|
\end{equation*} 
is simply the usual \Euc\ divided by the order of $G$. The strong form
of a conjecture by Quillen asserts that
\begin{equation*}
    \chi_0(\cat SG{p+*},G)=0 \iff O_p(G) \neq 1
\end{equation*}
For $r=1$, Webb \cite{webb87} proved that
\begin{equation*}
  \chi_1(\cat SG{p+*},G) = 1
\end{equation*}
when $\cat SG{p+*}$ is nonempty.
For $r=2$, Kn\"orr and Robinson
\cite{knorr_robinson:89,thevenaz93Alperin} conjecture that
\begin{equation*}
  \chi_2(\cat SG{p+*},G) = k(G)-z_p(G) 
\end{equation*}
The Kn\"orr--Robinson Conjecture is known to be equivalent to (the
non-blockwise form of) Alperin's Weight Conjecture. (See
Notation~\ref{notation} for explanations of the symbols $k(G)$ and
$z_p(G)$.)

A little more generally, suppose $G$ is an $A$-group. The poset $\cat
SG{p}$ of $p$-subgroups is then an $A$-poset. The objects of the
centralized subposet $C_{\cat SG{p}}(A)$ are the $A$-normalized
$p$-subgroups of $G$. For any such $A$-normalized $p$-subgroup $H$ of
$G$, the sub-quotient $N_G(H)/H$ is an $A$-group with associated
centralized Brown poset, $C_{\cat S{N_G(H)/H}{p+*}}(A)$.
%
The main result of this paper explores the local \Euc s given by the
function $H \to -\rchi(C_{\cat S{N_G(H)/H}{p+*}}(A))$ defined for $H \in
C_{\cat SG{p}}(A)$.  It is easy to see that this function vanishes off
the poset $C_{\cat SG{p+\mathrm{rad}}}(A)$ of $A$-normalized
$G$-radical $p$-subgroups of $G$ (Lemma~\ref{lemma:OpHA}).


\begin{thm}\label{thm:global}
  Let $G$ be a finite $A$-group and $p$ a prime.
  \begin{enumerate}

  \item \label{thm:global1} If $|C_G(A)_p|$ denotes the number of
    $p$-singular elements of $G$ centralized by $A$ then
  \begin{equation*}
     \sum_{H \in C_{\cat SG{p+\mathrm{rad}}}(A)}
  -\rchi(C_{\cat S{N_G(H)/H}{p+*}}(A)) |C_H(A)| = 
  |C_G(A)_p|
  \end{equation*}
 

  \item \label{thm:global0} For any $A$-normalized
   $G$-radical $p$-subgroup $H$ of $G$
    \begin{equation*}
      \sum_{\substack{H \leq K \in C_{\cat SG{p+\mathrm{rad}}}(A)}} 
        -\rchi(C_{\cat S{N_G(K)/K}{p+*}}(A)) =1 
    \end{equation*}
   where the sum runs over the set of $A$-normalized $G$-radical
   $p$-subgroups $K$ containing $H$.



\item \label{thm:global2} If $|C_G(A)|_p$ denotes the $p$-part of the
  order of the centralizer subgroup then
  \begin{equation*}
  |C_G(A)|_p \mid \rchi(C_{\catp SG*}(A))  
  \end{equation*}
  
  \end{enumerate}
\end{thm}

The proof of Theorem~\ref{thm:global}, to be found in
\S\ref{sec:chiSGpA} and \S\ref{sec:chiCOGA}, is obtained by analyzing
the \Euc s of the orbit category $\cat OGp$ of $p$-subgroups of $G$
and (a category related to) its $A$-centralized subcategory $C_{\cat
  OGp}(A)$ (Propositions \ref{prop:coweOGp}, \ref{prop:coweOGp},
\ref{prop:coweCOGpA}, \ref{prop:weCOGpA}).
Theorem~\ref{thm:global}.\eqref{thm:global1} expresses the number of
$A$-centralized $p$-singular elements in $G$ as an affine combination
of the number of $A$-centralized elements in the $G$-radical
$p$-subgroups of $G$. The divisibility statement of
Theorem~\ref{thm:global}.\eqref{thm:global2} follows from this affine
relation by induction over group order.

By the theorem of K.S.\@ Brown \cite{brown75}, reproved by Quillen
\cite{quillen78} and Webb \cite{webb87} among others, the $p$-part of
the group order divides the reduced \Euc\ of the Brown poset.
Theorem~\ref{thm:global}.\eqref{thm:global2} is an equivariant
generalization of Brown's theorem.

In the special case where $A$ is trivial, item \eqref{thm:global2} of
Theorem~\ref{thm:global} reduces to Brown's theorem while item
\eqref{thm:global1},
\begin{equation}\label{eq:global1}
   \sum_{H \in \cat SG{p+\mathrm{rad}}}
  -\rchi(\cat S{N_G(H)/H}{p+*})|H| = |G_p|
\end{equation}
expresses the number of $p$-singular elements in $G$ as an affine
combination of the orders of the $G$-radical $p$-subgroups.  Frobenius
proved in $1907$ (or even earlier) that $|G|_p$ divides $|G_p|$
\cite{frobenius:1907,isaacsrobinson}.  Brown's theorem, $|G|_p$
divides $\rchi(\cat SG{p+*})$, came much later in $1975$.  The affine
relation \eqref{eq:global1} provides a link between these two theorems
showing that they actually are equivalent
(Theorem~\ref{thm:frobeniusbrown}).  Thus the theorems of Sylow,
Frobenius, and Brown are equivalent.

We note that already in \cite[Theorem 6.3]{HIO89} Hawkes, Isaacs, and
\"Ozaydin prove a more general version of Equation~\eqref{eq:global1}
and also observe the connection in one direction between the theorems
of Brown and Frobenius.

In Section~\ref{sec:KRC} we compute the second equivariant \Euc\
$\chi_2(\cat SG{p+*},G)$ for the Mathieu group $G= M_{11}$ and check
the validity of the Kn\"orr--Robinson conjecture for this group
(Example~\ref{exmp:M11}). We also compute (Table~\ref{tab:artinM11})
the uniquely determined Artin decomposition of the second equivariant
Euler class function
\begin{equation*}
  \ralpha_2(\cat SG{p+*},G) = 
  \sum_{[C] \in [\cat SG{p'+\mathrm{cyc}}]} 
   \wa_2(\cat SG{p+*},G)(C) \frac{1_C^G}{|N_G(C):C|} 
\end{equation*}
into $\Z$-linear combinations of the class functions
$|N_G(C):C|^{-1} 1_C^G$ where $C$ runs through the set of
$p$-regular cyclic subgroups of $G$ (Corollary~\ref{cor:OpH}).

\subsection{Notation}
\label{sec:notation}
The following definitions and notation will be used throughout this paper:
\begin{defn}\label{defn:basic}
Let $G$ be a finite group and $p$ a prime number.
  \begin{enumerate}
  \item An element $g \in G$ of order $|g|$ is $p$-regular if $ p \nmid |g|$,
$p$-irregular if $p \mid |g|$, and $p$-singular if $|g|$ is a power
of $p$ \cite[Definition 40.2, \S82.1]{cr} \label{defn:basic1}
\item A finite group $A$ is $p$-regular if $p \nmid |A|$  \label{defn:basic2}
\item An irreducible $\C$-character on $G$ has $p$-defect $0$ if $p$
  does not divide
  $|G|/\chi(1)$ \cite[p 134]{isaacs} \label{defn:basic3}
\item $O_p(G)$ is the largest normal $p$-subgroup of $G$
\item 
A $p$-subgroup $H$ of $G$ is $G$-radical if $H=O_p(N_G(H))$
\cite{AlpFong} \label{defn:basic4} 
  \end{enumerate}
\end{defn}


\begin{notation}\label{notation}
Let $G$ be a finite group and $p$ a prime number.

\begin{itemize}

\item $G_p = \bigcup \mathrm{Syl}_p(G)$ is the set of $p$-singular
  elements of $G$, the union of the \syl ps

\item $|G|_p$ is the $p$-part and $|G|_{p'}$ the $p'$-part of the
  group order $|G| = |G|_p |G|_{p'}$

\item $[G]$ is the set of conjugacy classes of elements of $G$

\item $k(G) =|[G]| = | \Irr G| $ is the number of irreducible
  $\C$-characters of $G$

\item $z_p(G) = | \{ \chi \in \Irr G \mid |G|_p \mid \chi(1) \}|$ is
  the number of irreducible $\C$-characters of $p$-defect $0$

\item $k_{p'}(G) = | \{ [g] \in [G] \mid p \nmid |g| \}|$ is the number
  of $p$-regular conjugacy classes in $G$

\item $\cat SG{}$ is the poset of subgroups of $G$
   and $[\cat SG{}]$ is the set of conjugacy
  classes of subgroups (Section~\ref{sec:centposet})

\item $N_G(H,K) = \{g \in G \mid H^g \leq K \}$ is the transporter set

\item $\cat OG{}$ is the orbit category of $G$
   with \m\ sets $\cat OG{}(H,K) = N_G(H,K)/K$ (Section~\ref{sec:orbit})

\item $\cat FG{}$ is the fusion category of $G$ with \m\ sets $\cat
  FG{}(H,K) = C_G(H) \backslash N_G(H,K)$

\item $R_{\C}(G)$ is the ring of virtual $\C$-characters and
  $R_{\Q}(G)$ the subring of virtual $\Q$-characters of $G$
  \cite[12.1]{serre77}

\end{itemize}
If $\cat C{}{}$ is a finite category of subgroups of $G$ then $\cat
C{}*$ $ \cat C{}p$, $\cat C{}{p'}$, $\cat C{}{\mathrm{abe}}$, $\cat
C{}{\mathrm{eab}}$, $\cat C{}{\mathrm{rad}}$ is the full subcategory
of $\cat C{}{}$ generated by all nonidentity subgroups, $p$-subgroups,
$p$-regular subgroups, abelian subgroups, elementary abelian
subgroups, $G$-radical $p$-subgroups, respectively.  We shall also use
various combinations of superscripts; $\cat C{}{p+*+\mathrm{rad}}$,
for instance, denotes the full subcategory of $\cat C{}{}$ generated
by all nonidentity $G$-radical $p$-subgroups of $G$. $\cat
C{}{}(H,K)$ denotes the set of $\cat C{}{}$-\m s from $H$ to $K$ and $\cat
C{}{}(H) = \cat C{}{}(H,H)$ for the $\cat C{}{}$-endo\m\ monoid of
$H$. $H/\cat C{}{}$ is the coslice of $\cat C{}{}$ under $H$ and 
$H//\cat C{}{}$ the strict coslice of nonisomorphisms under $H$
\cite[Definition 3.2]{gm:2012}. 
\end{notation}


\section{$A$-categories and $A$-posets}
\label{sec:Gcat}

Let $A$ be a finite group, $\cat C{}{}$ a small $A$-category, and $X$
a subset of $A$.

\begin{defn}\label{defn:centralizercat}
  The $X$-centralized subcategory of $\cat C{}{}$ is the subcategory,
  $C_{\cat C{}{}}(X)$, of $\cat C{}{}$ with objects $\Ob{C_{\cat
      C{}{}}(X)} = C_{\Ob{\cat C{}{}}}(X)$ and with \m\ sets $C_{\cat
    C{}{}}(X)(c,d) = C_{\cat C{}{}(c,d)}(X)$, $c,d \in \Ob{C_{\cat
      C{}{}}(X)}$. 
\end{defn}

In mathematical symbols, the $X$-centralized subcategory $C_{\cat
  C{}{}}(X)$ is the subcategory of $\cat C{}{}$ consisting of all \m s
$c \xrightarrow{\varphi} d$ in $\cat C{}{}$ such that $(c \to d) =
(c^x \xrightarrow{\varphi^x} d^x)$ for all $x \in X$. In words, the
$X$-centralized subcategory $C_{\cat C{}{}}(X)$ is the category of
$X$-stable $\cat C{}{}$-objects with $X$-stable $\cat C{}{}$-\m s
between them.  $C_{\cat C{}{}}(X)$ is an $N_A(X)$-category.


If $X_1$ and $X_2$ are conjugate subsets of $A$, then their
centralizers in $\cat C{}{}$ are isomorphic categories. Indeed,
$C_{\cat C{}{}}(X^g) = C_{\cat C{}{}}(X)^g$ for any element $g \in A$
and any subset $X \subseteq A$.

For later reference we define $A$-adjunctions between $A$-categories
and note that there are induced $N_A(X)$-adjunctions between
centralizer categories.





\begin{defn}\label{defn:Gadj}
  A $A$-adjunction between the $A$-categories $\mathcal{C}$ and
  $\mathcal{D}$ is a quadruple $(L,R,\eta,\varepsilon)$ consisting of
  $A$-functors, $R$ and $L$, and $A$-natural transformations, $\eta$
  (the unit)
  and $\varepsilon$ (the counit),
  \begin{center}
    \begin{tikzpicture}[>=stealth']
 
    \tikzset{dbl/.style={double,
                     double equal sign distance,
                     -implies}}
    
    \matrix (dia) [matrix of math nodes, column sep=20pt, row
    sep=2pt]{
    \cat C{}{}  & \cat D{}{} & & 1_{\cat C{}{}} & LR \\
    \cat C{}{}  & \cat D{}{} & & 1_{\cat D{}{}} & RL \\};
    \draw[->] (dia-1-1) -- (dia-1-2) node[pos=.5,above] {$R$};
    \draw[dbl] 
    (dia-1-4.10) -- (dia-1-4.10 -| dia-1-5.west) 
    node[pos=.5,above] {$\eta$};
    \draw[->] (dia-2-2) -- (dia-2-1) node[pos=.5,below] {$L$};
    \draw[dbl] 
    (dia-2-5.180) -- (dia-2-5.180 -| dia-2-4.east)
    node[pos=.5,below] {$\varepsilon$};
    \end{tikzpicture}
  \end{center}
  such that the $A$-natural transformations
  \begin{center}
    \begin{tikzpicture}[>=stealth']
    
    \tikzset{dbl/.style={double,
                     double equal sign distance,
                     -implies}}
      
    \matrix (dia) [matrix of math nodes, column sep=20pt, row
    sep=5pt]{
      L & LRL & L & & R & RLR & R \\};
    \draw[dbl]
    (dia-1-1)--(dia-1-2)
    node[pos=.5,above] {$\eta L$};
    \draw[dbl]
    (dia-1-2)--(dia-1-3)
    node[pos=.5,above] {$L \varepsilon$};
    \draw[dbl]
    (dia-1-5)--(dia-1-6)
    node[pos=.5,above] {$R \eta $};
    \draw[dbl]
    (dia-1-6)--(dia-1-7)
    node[pos=.5,above] {$\varepsilon R$};
    \end{tikzpicture}
  \end{center}
  are the identity transformations.
\end{defn}

\begin{prop}\label{prop:adjunction}
  Suppose that $(L,R,\eta,\varepsilon)$ is a $A$-adjunction between the
  $A$-categories $\mathcal{C}$ and $\mathcal{D}$.
  \begin{enumerate}
  \item The induced maps  
  \begin{tikzpicture}[>=stealth', baseline=(current bounding box.-2)]
  \matrix (dia) [matrix of math nodes, column sep=25pt, row sep=20pt]{
  B\mathcal{C} & B\mathcal{D} \\};
  \draw[->] ($(dia-1-1.east)+(0,0.1)$) -- ($(dia-1-2.west)+(0,0.1)$) 
  node[pos=.5,above] {$BR$}; 
  \draw[->] ($(dia-1-2.west)-(0,0.1)$) -- ($(dia-1-1.east)-(0,0.1)$)
   node[pos=.5,below] {$BL$} ; 
 \end{tikzpicture}
 are $A$-homotopy equivalences between $A$-spaces. 
  \item There is an induced $N_A(X)$-adjunction between the
    $N_A(X)$-categories $C_{\cat C{}{}}(X)$ and $C_{\cat D{}{}}(X)$
  and the induced maps  
  \begin{tikzpicture}[>=stealth', baseline=(current bounding box.-2)]
  \matrix (dia) [matrix of math nodes, column sep=25pt, row sep=20pt]{
  BC_{\mathcal{C}}(X) & BC_{\mathcal{D}}(X) \\};
  \draw[->] ($(dia-1-1.east)+(0,0.1)$) -- ($(dia-1-2.west)+(0,0.1)$) 
  node[pos=.5,above] {$BR$}; 
  \draw[->] ($(dia-1-2.west)-(0,0.1)$) -- ($(dia-1-1.east)-(0,0.1)$)
   node[pos=.5,below] {$BL$} ; 
 \end{tikzpicture}
 are $N_A(X)$-homotopy equivalences between $N_A(X)$-spaces. 
  \end{enumerate}
\end{prop}

\begin{exmp}
  Let $C_2 = \gen{\tau}$ be the group of order two.  The product, $G
  \times G$, of any group $G$ with itself is a $C_2$-group with
  action given by $(g_1,g_2)\tau = (g_2,g_1)$, and $\cat S{G \times
    G}{p}$ is then a $C_2$-poset. The product poset $\cat SG{p}
  \times \cat SG{p}$ is a $C_2$-poset with action $(H_1,H_2)\tau =
  (H_2,H_1)$.  The $C_2$-adjunction
  \begin{center}
     \begin{tikzpicture}[>=stealth']
  \matrix (dia) [matrix of math nodes, column sep=25pt, row sep=20pt]{
  \cat SG{p} \times \cat SG{p}   & \cat S{G \times G}{p}
  & (H_1,H_2)R=H_1 \times H_2,\; HL=(H\pi_1,H\pi_2) \\};
  \draw[->] ($(dia-1-1.east)+(0,0.1)$) -- ($(dia-1-2.west)+(0,0.1)$) 
  node[pos=.5,above] {$R$}; 
  \draw[->] ($(dia-1-2.west)-(0,0.1)$) -- ($(dia-1-1.east)-(0,0.1)$)
   node[pos=.5,below] {$L$} ; 
 \end{tikzpicture}
  \end{center}
  restricts to an adjunction between $C_{\cat SG{p} \times \cat
    SG{p}}(\tau)$ and $C_{\cat S{G \times G}{p}}(\tau)$ and restricts
  further to an adjunction between the posets $C_{\cat SG{p+*} \times
    \cat SG{p+*}}(\tau)$ and $C_{\cat S{G \times G}{p+*}}(\tau)$. But
  $\cat SGp$ is isomorphic to $C_{\cat SG{p} \times \cat SG{p}}(\tau)$
  and $\cat SG{p+*}$ to $C_{\cat S{G \times G}{p+*}}(\tau)$, so we
  conclude that $(\cat SGp,C_{\cat S{G \times G}{p}}(\tau))$ and
  $(\cat SG{p+*},C_{\cat S{G \times G}{p+*}}(\tau))$ are pairs of
  adjoint posets.
\end{exmp}

\subsection{Euler characteristics of $A$-posets}
\label{sec:eucAposet}
In this subsection we make a few easy observations about \we s, co\we
s \cite[Definition 1.10]{leinster08}, and \Euc s of finite $A$-posets.

\begin{lemma}\label{lemma:Ainvmu}
  Let $\cat S{}{}$ be a finite $A$-poset. Then the \Mb\ function $\mu
  \colon \cat S{}{} \times \cat S{}{} \to \Z$, the \we\ $k^\bullet
  \colon \cat S{}{} \to \Z$ and the co\we\ $k_\bullet \colon \cat
  S{}{} \to \Z$ are $A$-invariant.
\end{lemma}
\begin{proof}
  Note thet $\zeta(r,s)=\zeta(r^a,s^a)$ for all $r,s \in \cat S{}{}$ and
  $a \in A$.
  The function $(s,t) \to \mu(s^a,t^a)$ satisfies the
  definining relations
  \begin{equation*}
    \sum_s \zeta(r,s)\mu(s^a,t^a) = 
    \sum_s \zeta(r^a,s^a)\mu(s^a,t^a) =
    \delta_{r^a,t^a} = \delta_{r,t}
  \end{equation*}
  for $(s,t) \to \mu(s,t)$. It is now clear that the 
  \we\ $k^s = \sum_t \mu(s,t)$ and the co\we\ $k_t = \sum_s \mu(s,t)$
  are constant on the orbits for the $A$-action on $\cat S{}{}$.
\end{proof}

For any two $A$-orbits $x,y \in \cat S{}{}/A$ and any two elements
$s,t \in \cat S{}{}$, let 
 $\cat S{}{}(s,y) = | \{ t \in y | s \leq t \}|$ be the number
of successors of $s$ in $y$ and 
 $\cat S{}{}(x,t) = | \{ s \in x | s \leq t \}|$
 the number
of predecessors of $t$ in $x$.


\begin{defn}\label{defn:wecoweSA}
  A \we\ on $\cat S{}{}/A$ is a function $k^\bullet \colon \cat
  S{}{}/A \to \Z$ such that $\sum_{y \in \cat S{}{}/A} \cat
  S{}{}(s,y)k^y = 1$ for all $s \in \cat S{}{}$. A co\we\ on $\cat
  S{}{}/A$ is a function $k_\bullet \colon \cat S{}{}/A \to \Z$ such
  that $\sum_{x \in \cat S{}{}/A} k_x \cat S{}{}(x,t) = 1$ for all $t
  \in \cat S{}{}$.
\end{defn}

\begin{lemma}\label{lemma:eucSA}
  Let $k^\bullet \colon \cat S{}{}/A \to \Z$ be a \we\ and $k_\bullet
  \colon \cat S{}{}/A \to \Z$ a co\we\ on $\cat S{}{}/A$. Then
  \begin{enumerate}
  \item $\cat S{}{} \to \cat S{}{}/A \xrightarrow{k^\bullet}
  \Z$ is the \we\ and $\cat S{}{} \to \cat S{}{}/A
  \xrightarrow{k_\bullet} \Z$ the co\we\ for $\cat S{}{}$
\item The \Euc\ of $\cat S{}{}$ is $\sum_{x \in \cat S{}{}/A} k_x|x| =
  \chi(\cat S{}{}) = \sum_{y \in \cat S{}{}/A} |y|k^y$
  \end{enumerate}
\end{lemma}
\begin{proof}
  Let $s$ be a fixed element of $\cat S{}{}$ representing the orbit
  $x=sA$ of $\cat S{}{}/A$. Then
  \begin{equation*}
    \sum_{t \in \cat S{}{}} |\cat S{}{}(s,t)| k^{tA} =
     \sum_{y \in \cat S{}{}/A} \sum_{t \in y}  |\cat S{}{}(s,t)|
     k^{tA}  =
    \sum_{y \in \cat S{}{}/A} \cat S{}{}(s,y) k^y =  1
  \end{equation*}
  Thus $t \to k^{tA}$ is the \we\ for $\cat S{}{}$. The \Euc\ is
  $\chi(\cat S{}{}) = \sum_{t \in \cat S{}{}} k^{tA} = \sum_{y \in
    \cat S{}{}/A} |y|k^y$.
\end{proof}

\begin{cor}\label{cor:uweSA}
  $\cat S{}{}/A$ has a unique \we\ and a unique co\we\ induced from
  the unique \we\ and co\we\ on $\cat S{}{}$.
\end{cor}
\begin{proof}
  Let $k^\bullet \colon \cat S{}{} \to \Z$ be the unique \we\ on $\cat
  S{}{}$. Since $k^\bullet$ is $A$-invariant
  (Lemma~\ref{lemma:Ainvmu}) it induces a function $k^{y} \colon
  \cat S{}{}/A \to \Z$ on the $A$-orbits. Since
  \begin{equation*}
    1 = \sum_{t \in \cat S{}{}} \cat S{}{}(s,t)k^t 
      = \sum_{y \in \cat S{}{}/A} \sum_{t \in y}\cat S{}{}(s,t)k^y
      = \sum_{x \in \cat S{}{}/A} \cat S{}{}(s,y)k^y
  \end{equation*}
  the function $k^y$ is a \we\ for $\cat S{}{}/A$.  Conversely, if
  $k^y \colon \cat S{}{}/A \to \Z$ is a \we\ for $\cat S{}{}/A$ then
  Lemma~\ref{lemma:eucSA} shows that it is induced by the \we\ on the
  poset $\cat S{}{}$.
\end{proof}

\subsection{Equivariant \Euc s of  $A$-categories}
\label{sec:chir}
As suggested by Atiyah and Segal \cite{atiyah&segal89} there is a
hierarchy of equivariant \Euc s $\chi_r(\cat C{}{},A)$, and reduced
equivariant \Euc s $\rchi_r(\cat C{}{},A)$, $ r \geq 0$, given by
\begin{equation*}
  \chi_r(\cat C{}{},A) =
  \frac{1}{|A|} \sum_{x \in C_r(A)} \chi(C_{\cat
    C{}{}}(x)), \qquad
  \rchi_r(\cat C{}{},A) =
  \frac{1}{|A|} \sum_{x \in C_r(A)} \rchi(C_{\cat
    C{}{}}(x))
\end{equation*}
Consult Remark~\ref{rmk:GRT} for the definition of the set $C_r(A)$ of
commuting $r$-tuples in $A$ and for the function $\varphi_r(B)$ used
in Proposition~\ref{prop:altequivchi} below.  If $A$ acts trivially on
$\cat C{}{}$ then $\chi_r(\cat C{}{},A)$ equals the usual \Euc\
$\chi(\cat C{}{})$ multiplied by $|C_r(A)|/|A|$ for all $r \geq 0$.
The relation between the equivariant \Euc\ and the {\em reduced\/}
equivariant \Euc\ is that
\begin{equation*}
  \chi_r(\cat C{}{},A) = \rchi_r(\cat C{}{},A) + \frac{|C_r(A)|}{|A|},
  \qquad r \geq 0
\end{equation*}
We are here implicitly assuming that the centralizer subcategories do
have \Euc s in the sense of Leinster \cite{leinster08}.  We note that
the equivariant \Euc s are invariant under equivariant adjunctions: If
the $A$-categories $\cat C{}{}$ and $\cat D{}{}$ are
$\mathrm{EI}$-categories and if there is an $A$-adjunction between
them, then their equivariant \Euc s coincide. This follows from
Proposition~\ref{prop:adjunction} and the invariance of category \Euc\
under adjunction \cite[Proposition 2.4]{leinster08}.

\begin{prop}\label{prop:altequivchi}
  The $r$th, $r \geq 1$, equivariant \Euc\ of the $A$-category $\cat
  C{}{}$ is
  \begin{multline*}
    \chi_r(\cat C{}{},A) =
     \frac{1}{|A|} \sum_{x \in A}
  \chi_{r-1}(C_{\cat C{}{}}(x),C_A(x))|C_A(x)| 
    =
    \sum_{[x] \in [A]} \chi_{r-1}(C_{\cat C{}{}}(x),C_A(x)) =
    \sum_{[x] \in [C_{r-1}(A)]} \chi_{1}(C_{\cat C{}{}}(x),C_A(x)) =
     \\ = \frac{1}{|A|}
   \sum_{B \in \cat SA{\mathrm{abe}}}  \chi(C_{\cat C{}{}}(B))
    \varphi_r(B) = 
    \sum_{[B] \in [\cat SA{\mathrm{abe}}]}  
    \chi(C_{\cat C{}{}}(B)) \frac{\varphi_r(B)}{|N_A(B)|} 
  \end{multline*}
  Similar formulas hold in the reduced case.
\end{prop}
\begin{proof}
This is immediate from the definition. For instance,
\begin{multline*}
  \chi_r(\cat C{}{},A) =
   \frac{1}{|A|}
  \sum_{(x_1,\ldots,x_{r-1},x_r) \in C_{r}(C_A(x))} 
   \chi(C_{\cat C{}{}}(x_1,\ldots,x_{r-1},x_r)) \\=
  \frac{1}{|A|}\sum_{x_r \in A} \;
  \sum_{(x_1,\ldots,x_{r-1}) \in C_{r-1}(C_A(x_r))}
  \chi(C_{C_{\cat C{}{}}(x_r)}(x_1,\ldots,x_{r-1}))  =
  \frac{1}{|A|}\sum_{x \in A}
  \chi_{r-1}(C_{\cat C{}{}}(x),C_A(x))|C_A(x)|
  \\= \sum_{[x] \in [A]} \chi_{r-1}(C_{\cat C{}{}}(x),C_A(x))
\end{multline*}  
%
Declare two $r$-tuples of $C_r(A)$ to be equivalent if they generate
the same abelian subgroup of $A$. The number of $r$-tuples in the
equivalence class of the abelian subgroup $B \leq A$ is
$\varphi_r(B)$.  Therefore we may write
\begin{equation*}
  \chi_r(\cat C{}{},A) 
  = \frac 1{|A|}
    \sum_{B \in \cat SA{\mathrm{abe}}}  \chi(C_{\cat C{}{}}(B))
    \varphi_r(B)
\end{equation*}
where the sum ranges over the abelian subgroups $B$ of $A$. 
\end{proof}

For any subgroup $C$ of $A$
\begin{equation*}
  |C|\chi_r(\cat C{}{},C) = 
  \sum_B \chi(C_{\cat C{}{}}(B))\varphi_r(B) |\cat SA{}(B,C)|, \qquad
  \sum_B |B| \chi_r(\cat C{}{},B) \mu(B,C) = 
  \chi(C_{\cat C{}{}}(C)) \varphi_r(C)
\end{equation*}
by \Mb\ inversion in the poset $\cat SA{}$ of subgroups of $A$.

\begin{defn}\label{defn:classfct}
  The $r$th, $r \geq 1$, equivariant Euler class function $\alpha_r(\cat
  C{}{},A)$ of the $A$-category $\cat C{}{}$ is the rational class
  function on $A$ that takes $x \in A$ to 
  \begin{equation*}
  \alpha_r(\cat C{}{},A)(x) = \chi_{r-1}(C_{\cat C{}{}}(x),C_A(x))  
  \end{equation*}
  The reduced Euler class function $\ralpha_r(\cat C{}{},A)$ is
  defined similarly using reduced \Euc s.
\end{defn}


For any $x \in A$ and any $r \geq 1$ the value at $x$  of the $r$th
equivariant Euler class function is
\begin{equation*}
  \alpha_r(\cat C{}{},A)(x) =  \chi_{r-1}(C_{\cat C{}{}}(x),C_A(x)) 
  =  \sum_{[y] \in [C_A(x)]} \chi_{r-2}(C_{\cat C{}{}}(x,y),C_A(x,y))
\end{equation*}
and for $r=2$, in particular,
\begin{equation}\label{eq:firstalpha2}
   \alpha_2(\cat C{}{},A)(x) =
    \sum_{[y] \in [C_A(x)]} \frac{\chi(C_{\cat C{}{}}(x,y))}{|C_A(x,y)|}
\end{equation}
because $\chi_0(\cat C{}{},A) = \chi(\cat C{}{})/|A|$.

Let $\gen{\varphi,\psi}_A = |A|^{-1} \sum_{x \in A} \varphi(x)
\overline{\psi(x)} = \sum_{[x] \in [A]} |C_A(x)|^{-1} \varphi(x)
\overline{\psi(x)}$ be the character inner product (symmetric bilinear
form) in the complex class function space on $A$
\cite[Definition~2.16]{isaacs}.  The $1$-character $1_A$ is
characterized by the property that its inner product with any class
function $\varphi$ is the average $\gen{\varphi,1_A}_A =
|A|^{-1}\sum_{x \in A} \varphi(x) = \sum_{[x] \in [A]} \varphi(x) |A :
C_A(x)|$ of the values of $\varphi$.
Let $|C_A| \colon x \to |C_A(x)| = |C_A(x^{-1})|$ be the conjugation
character: The permutation character for the conjugation action of $A$
on itself.  The conjugation character is characterized by the property
that the inner product $\gen{\varphi,|C_A|}_A = \sum_{[x] \in [A]}
x^\varphi$ is the sum of the values of $\varphi$ on the conjugacy
classes of $A$. In particular, the inner product of the $r$th
equivariant Euler class function and $|C_A|$ is the $r$th equivariant \Euc :
\begin{equation}\label{eq:charinnerprod}
  \gen{\alpha_r(\cat C{}{},A),|C_A|}_A = 
  \sum_{[x] \in [A]} \alpha_r(\cat C{}{},A)(x) =
  \sum_{[x] \in [A]} \chi_{r-1}(C_{\cat C{}{}}(x),C_A(x)) 
   \stackrel{\tiny \text{Prop~\ref{prop:altequivchi}}}{=} 
  \chi_r(\cat C{}{},A)
\end{equation}
Similarly, the inner product of {\em
  reduced\/} class function $\ralpha_r(\cat C{}{},A)$ and $|C_A|$
coincides with the {\em reduced\/} equivariant \Euc\ $\rchi_r(\cat
C{}{},A)$.

\subsection{Equivariant \Euc s of $A$-posets}
\label{sec:chirS}
We now specialize from a finite $A$-category, $\cat C{}{}$, to a
finite $A$-poset, $\cat S{}{}$.  
The order $\Delta$-set of $\cat S{}{}$  is the
$A$-$\Delta$-set, $\Delta\cat S{}{}$, of all simplices in $\cat
S{}{}$ (\S\ref{sec:orbdelta}).  The equivariant \Euc s of a finite
$A$-poset $\cat S{}{}$ for $r=0,1,2$ are
\begin{equation*}
  \chi_0(\cat S{}{},A) = \chi(\cat S{}{})/|A|, \ 
  \chi_1(\cat S{}{},A) = \chi(\Delta{\cat S{}{}}/A), \ 
  \chi_2(\cat S{}{},A) =  \dim_{\Q} (K^0_A(B\cat S{}{}) \otimes \Q) -
   \dim_{\Q} (K^1_A(B\cat S{}{}) \otimes \Q)
\end{equation*}
according to Proposition~\ref{prop:quocomplex} below for $r=1$ and
\cite{atiyah&segal89} \cite[Corollary 2.2]{thevenaz93Alperin} for
$r=2$.  $K_A^*(B\cat S{}{})$ is the $A$-equivariant complex $K$-theory
of the finite $A$-simplicial complex $B\cat S{}{}$.  The corresponding
{\em reduced\/} \Euc s are $\rchi_0(\cat S{}{},A) = \chi_0(\cat
S{}{},A) - |A|^{-1}$, $\rchi_1(\cat S{}{},A) = \chi_1(\cat
S{}{},A)-1$, and $\rchi_2(\cat S{}{},A) = \chi_2(\cat S{}{},A)-k(A)$
(Remark~\ref{rmk:GRT}).
 
The next proposition shows that the first equivariant \Euc\ of the
$A$-poset $\cat S{}{}$ is the usual \Euc\ of the orbit $\Delta$-set of
the $A$-$\Delta$-set $\Delta\cat S{}{}$. This implies, as epressed in
the following corollary, that the $r$th equivariant \Euc\ for $r \geq
2$ can be expressed by means of the usual \Euc s of quotients
of $\Delta$-sets of  centralizer subposets of $\cat S{}{}$.

\begin{prop}\label{prop:quocomplex}
   For any subgroup $H$ of $A$
  \begin{equation*}
    \chi_1(C_{\cat S{}{}}(H),C_A(H)) =
    \chi(\Delta C_{\cat S{}{}}(H)/C_A(H)) 
  \end{equation*}
A similar formula holds in the reduced case.
\end{prop}
\begin{proof}
  For any simplicial complex $B$ let $B_d$ denote set the of
  $d$-dimensional simplices in $B$. Then
  \begin{multline*}
    \chi(\Delta C_{\cat S{}{}}(H)/C_A(H)) =
    \sum_{d \geq 0} (-1)^d \left|  \Delta C_{\cat S{}{}}(H)_d/C_A(H) \right| =
    \sum_{d \geq 0} (-1)^d \frac{1}{|C_A(H)|} \sum_{x \in C_A(H)}
    \left| C_{\Delta C_{\cat S{}{}}(H)_d}(x) \right| \\ =
    \frac{1}{|C_A(H)|} \sum_{x \in C_A(H)} \sum_{d \geq 0}  (-1)^d  
    |\Delta C_{\cat S{}{}}(\gen{H,x})_d| = 
    \frac{1}{|C_A(H)|} \sum_{x \in C_A(H)} 
    \chi ( C_{\cat S{}{}}(\gen{H,x})) =
    \chi_1(C_{\cat S{}{}}(H),C_A(H))
  \end{multline*}
  where the Cauchy--Frobenius Lemma~\ref{lemma:burnside} justifies the
  second equality above. 
\end{proof}

\begin{cor}\label{cor:quocomplex}
  The equivariant \Euc s of the finite $A$-poset $\cat S{}{}$ are
  $\chi_1(\cat S{}{},A) = \chi(\Delta{\cat S{}{}}/A)$ and
  \begin{equation*}
    \chi_r(\cat S{}{},A) =
  \sum_{[x] \in [C_{r-1}(A)]}\chi(\Delta{C_{\cat S{}{}}(x)}/C_A(x)) =
  \frac{1}{|A|}
  \sum_{B\in \cat SA{\mathrm{abe}}} \chi(\Delta{C_{\cat S{}{}}(B)}/C_A(B)) 
  |C_A(B)| \varphi_{r-1}(B)
  \end{equation*}
  for $r \geq 2$. The first sum is taken over the set $[C_{r-1}(A)]$
  of conjugacy classes of commuting $(r-1)$-tuples in $A$.
\end{cor}
\begin{proof}
  Proposition~\ref{prop:quocomplex} with $H$ trivial shows that
  $\chi_1(\cat S{}{},A) = \chi(\Delta \cat
  S{}{}/A)$. Using Proposition~\ref{prop:altequivchi} we get 
\begin{equation*}
  \chi_2(\cat S{}{},A) =
  \frac{1}{|A|}\sum_{x \in A}\chi_1(C_{\cat S{}{}}(x),C_A(x))|C_A(x)|=
  \frac{1}{|A|}\sum_{x \in A}\chi(\Delta C_{\cat S{}{}}(x)/C_A(x))|C_A(x)|  
\end{equation*}
so that, by induction,
\begin{equation*}
  \chi_r(\cat S{}{},A) = \frac{1}{|A|}
  \sum_{x \in C_{r-1}(A)}\chi(\Delta C_{\cat S{}{}}(x)/C_A(x))|C_A(x)| =
  \sum_{[x] \in [C_{r-1}(A)]}\chi(\Delta C_{\cat S{}{}(x)}/C_A(x))
\end{equation*}
where the last sum is taken over the set $[C_{r-1}(A)] = C_{r-1}(A)/A$
of conjugacy classes of commuting $r$-tuples in $A$ and $r \geq
2$. 
\end{proof}

\begin{cor}\label{cor:xXCG}
  The equivariant \Euc s $\chi_r(\cat S{}{},A)$ are integers for $r
  \geq 1$ and the equivariant Euler class functions $\alpha_r(\cat
  S{}{},A)$ take integer values for $r \geq 2$. The same holds in the
  reduced case.
\end{cor}
\begin{proof}
  The first \Euc\ $\chi_1(\cat S{}{},A) = \chi(\Delta \cat S{}{}/A)$
  is an integer by Proposition~\ref{prop:quocomplex} since it is the
  \Euc\ of a finite $\Delta$-set. By induction, using one of the
  formulas of Proposition~\ref{prop:altequivchi}, we get that
  $\chi_r(\cat S{}{},A) \in \Z$ for all $r \geq 1$. By applying this
  result to $(C_{\cat S{}{}}(x),C_A(x))$ we see that the equivariant
  Euler class function $\alpha_r(\cat S{}{},A)$ takes integer values
  on each $x \in A$ when $r \geq 2$.
\end{proof}

\begin{lemma} \cite[Lemma~7.24.5]{stanley99}\label{lemma:burnside} Let
  $S \times A \to A$ be an action of the group $A$ on the set $S$.
  Then
  \begin{equation*}
    \sum_{a \in A} |C_S(a)| = |S/A| |A| = \sum_{s \in S} |C_A(s)|
  \end{equation*}
  and the number of orbits is $|S/A| = \sum_{s \in S}
  |C_A(s)|/|A|$. 
\end{lemma}

We now specialize even further. When $G$ is a finite group, recall
that the Brown $G$-poset $\cat SG{p+*}$ is the $G$-poset of
nonidentity $p$-subgroups of $G$ with the conjugation action.

\begin{rmk}[Equivariant \Euc s of the Brown poset with conjugation
  action] \label{exmp:3conjs} Here are three statements about the
  equivariant \Euc s for $r=0,1,2$ of the poset $\cat SG{p+*}$ with
  conjugation action of $A=G$:
  \begin{enumerate}
  \item $\forall G \colon -\rchi_0(\cat SG{p+*},G) = 0 \iff O_p(G) \neq
    1$ \label{list:brown1} 
  \item $\forall G \colon -\rchi_1(\cat SG{p+*},G) =
    \begin{cases}
      0 & p \mid |G| \\ 1 & p \nmid |G|
    \end{cases}$ \label{list:brown2}
  \item $\forall G \colon -\rchi_2(\cat SG{p+*},G) = z_p(G)$
    \label{list:brown3} 
  \end{enumerate}
  Statement \eqref{list:brown1} is the strong form of Quillen's
  conjecture about the vanishing of the reduced \Euc\ of the Brown
  poset \cite[Conjecture 2.9]{quillen78} \cite[p.\@
  2667]{jmm_mwj:2010}, \eqref{list:brown2} was proved by Webb
  \cite[Proposition 8.2.(i)]{webb87} (and sharpened by Symonds
  \cite{symonds98} and later S{\l}omi{\'n}ska
  \cite{swominska01}), and \eqref{list:brown3} is equivalent
  (Section~\ref{sec:KRC}) to (the non-blockwise form of) Alperin's
  weight conjecture \cite[Theorem 3.1]{thevenaz93Alperin}.

  The significance of the equivariant \Euc s $\chi_r(\cat SG{p+*},G)$
  for $r \geq 3$ remains unclear.  In the case of the alternating and
  symmetric groups the $3$rd reduced \Euc s  $-\rchi_3(\cat SG{p+*},G)$ at
  $p=2$ are
  \begin{center} 
    \begin{tabular}[t]{>{$}c<{$} | *{7}{>{$}r<{$}} }
     n & 4 & 5 & 6 &  7 & 8 & 9 & 10 \\ \hline \noalign{\smallskip}
     -\rchi_3(\cat S{A_n}{2+*},A_n)
     & 0 & 8 & 24 & -2 & 32 & 20 & 42 \\ \noalign{\smallskip}
    -\rchi_3(\cat S{\Sigma_n}{2+*},\Sigma_n) 
    & 0 & 2 & 12 & -2 &  10 & 11 & 16  
   \end{tabular}
  \end{center}
  For $G=\GL 3{\F_2}$, the \Euc s   $-\rchi_3(\cat SG{p+*},G)$ are
  $12,24,7$ at $p=2,3,7$ and for $G=M_{11}$ they are $29,19,62,35$ at
  $p=2,3,5,11$. 

  Similarly, the significance of the more general equivariant \Euc s
  $\chi_r(\cat SG{p+*},A)$, $r \geq 1$, for an $A$-group $G$ remains
  unclear. (See Remark~\ref{rmk:rchiSGpK}.)
\end{rmk}

Using Proposition~\ref{prop:altequivchi}, Webb's computation of
$\chi_1(\cat SG{p+*},G)$
(Remark~\ref{exmp:3conjs}.\eqref{list:brown2})  may be
reformulated as either of three equivalent identities 
\begin{equation}\label{eq:webb}
  \sum_{x \in G} \chi(C_{\cat SG{p+*}}(x)) = |G|, \qquad
   \sum_{[x] \in [G]}
  \chi(C_{\cat S G{p+*}}(x)) |G : C_G(x)| = |G|, \qquad
   \sum_{[x] \in [G]} \rchi(C_{\cat S G{p+*}}(x)) |G : C_G(x)| = 0
\end{equation}
valid whenever $p \mid |G|$ (and reminiscent of the class equation).

\begin{exmp}
  The following two tables list the centralizer indices $|G : C_G(x)|$
  and the centralizer \Euc s $\chi(C_{\cat SG{p+*}}(x))$ as $x$ runs
  through the conjugacy classes of $G$ and $p$ runs through the prime
  divisors of $|G|$ for $G= \GL 3{\F_2}$ of order $168$ and the Mathieu group
  $G=M_{11}$ of order $7920$, respectively.
  \begin{center}  
    \begin{tabular}[t]{>{$}c<{$} | *{6}{>{$}r<{$}} | >{$}c<{$}  }
|x| & 1  & 2  & 3  & 4  & 7  & 7 & \cdot \\ \hline
|G:C_G(x)| & 1  &21  &56  &42  &24  &24 & {}\\ \hline
\rchi(C_{\cat S{G}{2+*}}(x)) & -8  & 0  & 1  & 0  &-1  &-1 & 0 \\
\rchi(C_{\cat S{G}{3+*}}(x)) &27  & 3  & 0  &-1  &-1  &-1 & 0\\ 
\rchi(C_{\cat S{G}{7+*}}(x)) &  7  &-1  & 1  &-1  & 0  & 0 & 0
    \end{tabular}
  \end{center}
  \begin{center}
    \begin{tabular}[t]{>{$}c<{$} | *{10}{>{$}r<{$}} | >{$}c<{$}  }
|x| & 1  &   2  &   3  &   4  &   5  &   6  &   8  &   8  &  11  &  11
& \cdot
\\ \hline
|G : C_G(x)| &
  1  & 165  & 440  & 990  &1584  &1320  & 990  & 990  & 720  & 720 & {}\\
  \hline
\rchi(C_{\cat SG{2+*}}(x)) &
-496  &   0  &   8  &   0  &  -1  &   0  &   0  &   0  &  -1  &  -1 &0\\
\rchi(C_{\cat SG{3+*}}(x)) &
  54  &   6  &   0  &   2  &  -1  &   0  &   0  &   0  &  -1  &  -1 &0\\
\rchi(C_{\cat SG{5+*}}(x)) &
 395  &  11  &  -1  &   3  &   0  &  -1  &  -1  &  -1  &  -1  &  -1 &0\\
\rchi(C_{\cat SG{11+*}}(x)) &
 143  &  -1  &  -1  &  -1  &   3  &  -1  &  -1  &  -1  &   0  &   0 &0
    \end{tabular}
  \end{center}
  Note that all rows satisfy Webb's relation \eqref{eq:webb}. It is no
  coincidence that $\rchi(C_{\cat SG{p+*}}(x))=0$ when $p$ divides the
  order of $x$ (Lemma~\ref{lemma:OpH}).
\end{exmp}

\begin{rmk}[Commuting $r$-tuples]\label{rmk:GRT}
  For $r=0,1,2,\ldots$ let
\begin{equation*}
  C_r(A) = \{ (x_1,x_2,\ldots,x_r) \mid \forall i,j \colon [x_i,x_j]=1
  \}
\end{equation*}
denote the set of commuting $r$-tuples in $A$ (with the understanding that
$C_0(A) = \{e \}$ is the set consisting of the unit element
of $A$).  

Then $|C_0(A)|=1$, $|C_1(A)|=|A|$, $|C_2(A)|=|A|k(A)$, and
\begin{equation*}
  |C_r(A)| = |A| \sum_{x \in C_{r-2}(A)} k(C_A(x)), \qquad r \geq 2
\end{equation*}
This follows from the recursive relation 
\begin{equation*}
  |C_r(A)| = \sum_{x \in A} |C_{r-1}(C_A(x))| =
  \sum_{[x] \in A} |A : C_A(x)||C_{r-1}(C_A(x))|, \qquad r \geq 1,
\end{equation*}
obtained by noting that there are $|C_{r-1}(C_A(x))|$ commuting
$r$-tuples with first coordinate $x$ for any $x \in A$.

Let also
\begin{equation*}
  \varphi_r(A) = | \{ (x_1,\ldots,x_r) \in C_r(A) \mid
  \gen{x_1,\ldots,x_r} = A \}|
\end{equation*}
be the number of generating commuting $r$-tuples. Of course,
$\varphi_r(A)=0$ for all $r$ unless $A$ is abelian. When $A$ is
abelian
\begin{equation*}
  \varphi_r(A) = \prod_p \varphi_r(O_p(A))
\end{equation*}
since any abelian group $A= \prod_p O_p(A)$ is the product of its \syl
ps. Thus $\varphi_r$ is completely determined by its values on cyclic
$p$-groups.

Similar to one of the expressions of
Proposition~\ref{prop:altequivchi} we get 
\begin{equation*}
    |C_r(A)| = \sum_{B \in \cat SA{\mathrm{abe}}} \varphi_r(B)
\end{equation*}
by using the partition of $C_r(A)$ into equivalence classes.

\end{rmk}

\subsection{Orbit posets and orbit $\Delta$-sets}
\label{sec:orbdelta}
This subsection contains remarks about orbit posets of $A$-posets.

\begin{defn}\label{defn:DeltaS}
Let $\cat S{}{}$  be an  $A$-poset.  
\begin{itemize}

\item A {\em simplex\/} in $\cat S{}{}$ is a totally ordered subset of
  $\cat S{}{}$

\item $\B(\cat S{}{})$ is the set of simplices in $\cat S{}{}$ viewed as an
$A$-simplicial complex

\item $\Delta(\cat S{}{})$ is the set of simplices in $\cat S{}{}$
  viewed as an $A$-$\Delta$-set\footnote{A $\Delta$-set is a functor
    $\Delta_< \to \mathbf{SET}$ where $\Delta_<$ is the category of
    standard simplices $n_+=\{0,\ldots,n\}$, $n=0,1,2,\ldots$, with
    strictly increasing maps.} \cite{hatcher}

\item $\sd {\cat S{}{}}$ is the set of simplices in $\cat S{}{}$ viewed
as an $A$-poset

\item $\cat S{}{}$ is graded if $\cat S{}{}$ admits an $A$-invariant
  {\em rank function\/} $\rho \colon \cat S{}{} \to \{0,1,2,\ldots\}$
  \cite[\S 3.1]{stanley97}

\item $\cat S{}{}$ is $A$-regular if $\cat
  S{}{}$ if $s \leq t, \; s^a \leq t \implies s=s^a$ holds for all
  $s,t \in \cat S{}{}$, $a \in A$  \cite[p $116$]{bredon72} 

\end{itemize}
\end{defn}

Any $A$-poset $\cat S{}{}$ admits an orbit poset $\cat S{}{}/A$ and
the $A$-$\Delta$-set $\Delta(\cat S{}{})$ admits an orbit $\Delta$-set
$\Delta(\cat S{}{})/A$.


\begin{prop}\label{prop:orbitdelta}
  Let $\cat S{}{}$ be any $A$-poset.
  \begin{enumerate}
  \item $\sd{\cat S{}{}}$ is an $A$-poset and $\chi(\Delta(\sd{\cat
      S{}{}})/A)=\chi(\Delta(\cat S{}{})/A)$ \label{prop:orbitdelta1}
  \item If $\cat S{}{}$ is $A$-regular then $\sd{\cat S{}{}}/A =
    \sd{\cat S{}{}/A}$ and $\Delta(\cat S{}{})/A =
    \Delta(\cat S{}{}/A)$ \label{prop:orbitdelta2}
  \item If $\cat S{}{}$ is graded then $\sd{\cat S{}{}}$ is
    $A$-regular, and $\chi_1(\cat S{}{},A) = \chi(\sd{\cat S{}{}}/A)$
    \label{prop:orbitdelta3}
  \item The first subsdivision $\sd{\cat S{}{}}$ is graded, the
    second subdivision $\mathrm{sd}^2(\cat S{}{})$ is $A$-regular, and
    $\chi_1(\cat S{}{},A) = \chi(\mathrm{sd}^2(\cat S{}{})/A)$
    \label{prop:orbitdelta4} 
  \end{enumerate}
\end{prop}
\begin{proof}
  \eqref{prop:orbitdelta1} Since the standard homoe\m\
  $|\Delta(\sd{\cat S{}{}})| \to |\Delta(\cat S{}{})|$ between the
  topological realizations of the subdivided and the original poset
  \cite[\S3, Theorem 9]{spanier} is $A$-equivariant by construction,
  it induces a homeo\m\ between the $A$-orbit spaces.

 \noindent \eqref{prop:orbitdelta2} The poset \m\ 
 $\cat S{}{} \to \cat S{}{}/A$ induces a poset \m\ 
 $\sd{\cat S{}{}} \to \sd{\cat S{}{}/A}$ and a $\Delta$-set \m\
 $\Delta(\cat S{}{}) \to \Delta(\cat S{}{}/A)$. When $\cat S {}{}$ is
 $A$-regular these induced \m s induce iso\m s 
 $\sd{\cat S{}{}}/A \to \sd{\cat S{}{}/A}$ and
 $\Delta(\cat S{}{})/A \to \Delta(\cat S{}{}/A)$ because
 $\{s^a \mid s \in \sigma \} = \{ t^a \mid t \in \tau \} \implies
 \sigma^A = \tau^A$ for all simplices $\sigma$ and $\tau$ in $\cat
 S{}{}$. 

 \noindent \eqref{prop:orbitdelta3} We must prove that $\sigma
 \subseteq \tau, \; \sigma^a \subseteq \tau \implies \sigma =
 \sigma^a$ for all simplices $\sigma$, $\tau$ in $\cat S{}{}$ and $a
 \in A$. The elements of the totally ordered subset $\tau$ are
 determined by their values under the rank function. Subsets of $\tau$
 are therefore determined by their images under the rank
 function. Since the rank function is $A$-invariant, $\sigma$ and
 $\sigma^a$ have the same images, so they must be equal. The first
 equivariant \Euc\ is
 \begin{equation*}
   \chi_1(\cat S{}{},A) 
   \stackrel{\text{\tiny Prop~\ref{prop:quocomplex}}}{=} 
   \chi(\Delta(\cat S{}{})/A)) \stackrel{\eqref{prop:orbitdelta1}}{=} 
   \chi(\Delta(\sd{\cat S{}{}})/A)
   \stackrel{\eqref{prop:orbitdelta2}}{=} 
   \chi(\Delta(\sd{\cat S{}{}/A})) = \chi(\sd{\cat S{}{}}/A)
 \end{equation*}
 because $\sd{\cat S{}{}}$ is $A$-regular.

 \noindent \eqref{prop:orbitdelta4} The dimension function is an
 $A$-invariant rank function on $\sd{\cat S{}{}}$. The first
 equivariant \Euc\ is
 \begin{equation*}
   \chi_1(\cat S{}{},A) 
   \stackrel{\text{\tiny Prop~\ref{prop:quocomplex}}}{=} 
   \chi(\Delta(\cat S{}{})/A)) \stackrel{\eqref{prop:orbitdelta1}}{=} 
   \chi(\Delta(\sd{\cat S{}{}})/A)
   \stackrel{\eqref{prop:orbitdelta1}}{=}
   \chi(\Delta(\mathrm{sd}^2(\cat S{}{}))/A)
   \stackrel{\eqref{prop:orbitdelta2}}{=}
   \chi(\Delta(\mathrm{sd}^2(\cat S{}{})/A)) =
   \chi(\mathrm{sd}^2(\cat S{}{}/A) 
 \end{equation*}
 since $\mathrm{sd}^2(\cat S{}{})$ is $A$-regular by
 \eqref{prop:orbitdelta3}.
\end{proof}

For any finite $A$-group $G$ and any subset $X$ of $A$, the
$X$-centralized Brown poset $C_{\cat SG{p+*}}(X)$ is a graded
$N_A(X)$-poset with rank function $\rho$ given by $|H| = p^{\rho(H)}$
for any $X$-normalized $p$-subgroup $H$ of $G$. Thus the first
equivariant \Euc\ of Proposition~\ref{prop:quocomplex}
\begin{equation*}
  \chi_1(C_{\cat SG{p+*}}(X),C_A(X)) = 
  \chi(\sd{C_{\cat SG{p+*}}(X)}/C_A(X))
\end{equation*}
can be expressed entirely within the category of posets.

\begin{exmp}\label{exmp:deltaset} 
  In case $G = \GL 3{\F_2}$ and $p=2$ the $2$-dimensional contractible
  orbit $\Delta$-set $\Delta(\cat SG{p+*})/G$
 \begin{center}
   \begin{tikzpicture}[>=stealth']
     \node (S0) at (-6.7,0) {$\{1,\ldots,5\}$};
     \node (S1) at (0,0) {$\{1,\ldots,9\}$};
     \node (S2) at (4.7,0) {$\{1,\ldots,5\}$};
     \draw[->] (S1.175) -- (S0.0 |- S1.175) node[pos=.5,above]   
    {$ 
      \begin{pmatrix}
        1 & 2 & 3 & 5 & 5 & 5 & 5 & 5 & 5 \\
        4 & 4 & 4 & 4 & 1 & 3 & 2 & 4 & 4
      \end{pmatrix}$};
    \draw[->] (S1.185) -- (S0.0 |- S1.185);
    \draw[->] (S2.170) -- (S1.0 |- S2.170) node[pos=.5,above]
    {$
      \begin{pmatrix}
        5 & 7 & 6 & 5 & 7 \\
        4 & 4 & 4 & 8 & 9 \\
        1 & 2 & 3 & 1 & 2
      \end{pmatrix}$};
     \draw[->] (S2.180) -- (S1.0 |- S2.180);
     \draw[->] (S2.190) -- (S1.0 |- S2.190);
   \end{tikzpicture}
 \end{center}
 is not a simplicial complex as three $1$-simplices connect the
 $0$-simplices $4$ and $5$. In the realization, $2$-simplices $1$, $4$
 and $2$,$5$ form two cones joined along the $1$-simplex $4$ to which
 also one edge of $2$-simplex $3$ is attached.
\end{exmp}

\section{The poset of $p$-subgroups}
\label{sec:centposet}

Write $\cat SG{}$ for the poset of all subgroups of $G$ ordered by
inclusion, $\cat SGp$ for the poset of all $p$-subgroups, and $\cat
SG{p+*}$ for the Brown poset of all nonidentity $p$-subgroups of $G$.

For any $p$-subgroup $K$ of $G$, the poset $\cat SK{}$ of subgroups of
$K$ has both an initial, $1$, and a terminal element, $K$. We write
$\cat SK{(1,K)}$ for the poset of nonidentity and proper subgroups of
$K$. The \Mb\ function on $\cat SK{}$ is the restriction of the \Mb\
function $\mu$ for $\cat SG{}$ \cite[Remark 2.6]{jmm_mwj:2010}. It is
well-known that $\rchi(\cat SK{(1,K)}) = \mu(1,K)$ when $K$ is
nontrivial. In fact, if $\cat S{}{}$ is any finite poset with \Mb\
function $\mu$, $a \in \cat S{}{}$, $b \in \cat S{}{}$ with $a < b$,
and $(a,b)$ and $[a,b]$ are the open and closed intervals from $a$ to
$b$, then
\begin{align*}
  1 &= \chi([a,b]) = \sum_{a \leq s,t \leq b} \mu(s,t) 
  = \sum_{a < s,t < b} \mu(s,t) + \sum_{a \leq t \leq b} \mu(a,t)
  + \sum_{a \leq s \leq b} \mu(s,b) - \mu(a,b) \\
  &= \chi(a,b) + \delta(a,b)+\delta(a,b)-\mu(a,b) 
  = \chi(a,b)-\mu(a,b)
\end{align*}
by the defining properties of \Mb\ functions \cite[3.3.7]{stanley97}.

For any $p$-subgroup $H$ of $G$ we write $\cat OG{}(H)$ for
$N_G(H)/H$ (the auto\m\ group of $H$ in the orbit category $\cat
OGp$).

\subsection{\Euc\ of the Brown poset $\cat SG{p+*}$}
\label{sec:chiSGp}
The arguably most basic fact about the Brown poset is its
contractibility when $O_p(G)$ is nontrivial.
\begin{lemma}\cite[Proposition 2.4]{quillen78}\label{lemma:quillenOpG}
  $O_p(G) \neq 1 \implies \cat SG{p+*} \simeq \ast \implies
  \chi(\cat SG{p+*}) = 1$
\end{lemma}
The \we\ and the co\we\ in the sense of Leinster
\cite[Definition~1.10]{leinster08} for the Brown poset are known.

\begin{prop}\cite[Theorems 1.1 and
  1.3]{jmm_mwj:2010}\label{prop:SGpwecowe} 
  The functions
  \begin{equation*}
    k^H = -\rchi(H// \cat SGp) = -\rchi(\cat S{\cat OG{}(H)}{p+*}), \qquad
    k_K = -\rchi(\cat SK{(1,K)})
  \end{equation*}
  are a \we\ for $\cat SGp$ and a co\we\ for $\cat SG{p+*}$. The \Euc\
  of $\cat SG{p+*}$ is
\begin{equation*}
  \sum_{H \in \cat SG{p+*+\mathrm{rad}}} -\rchi(\cat S{\cat
    OG{}(H)}{p+*}) =
  \chi(\cat SG{p+*}) =
  \sum_{K \in \cat SG{p+*+\mathrm{eab}}}  -\rchi(\cat SK{(1,K)})
\end{equation*}
\end{prop}

The \we\ is concentrated on the nonidentity $G$-radical $p$-subgroups
(Lemma~\ref{lemma:quillenOpG}) and the co\we\ on the elementary
abelian $p$-subgroups \cite[Lemma 3.2]{jmm_mwj:2010}.  We observed
above that the co\we\ can also be viewed as a \Mb\ function. 
Let $K$ be any $p$-subgroup of $G$.  Since the \we\ for $\cat SGp$
restricts to a \we\ for the
contractible left ideal $\cat SG{p \geq K}$ of $p$-subgroups of $G$
containing $K$ \cite[Remark 2.6]{jmm_mwj:2010} there is a 
relation
\begin{equation}
  \label{eq:affine}
       \sum_{K \in H / \cat SG{p+\mathrm{rad}}}
  -\rchi(\cat S{\cat OG{}(K)}{p+*}) = 1
\end{equation}
between the \Euc s of the Brown posets of $\cat OG{}(K) = N_G(K)/K$
for $K \in [H,G]$.

\begin{prop}\label{prop:weSGprad}
  The \we\ $k^H \colon \cat SGp \to \Z$ for $\cat SGp$ is given by
  \begin{equation*}
    k^H =
    \begin{cases}
      k^{[H]} & \text{$H$ is $G$-radical} \\
      0 & \text{otherwise}
    \end{cases}
  \end{equation*}
  where $k^{[H]} \colon \cat SG{p+\mathrm{rad}}/G \to \Z$ is the \we\
  for $\cat SG{p+\mathrm{rad}}/G$
  (Definition~\ref{defn:wecoweSA}). The co\we\ $k_K \colon \cat
  SG{p+*} \to \Z$ 
  for $\cat SG{p+*}$ is given by
  \begin{equation*}
    k_K =
    \begin{cases}
      k_{[K]} & \text{$K$ is elementary abelian} \\
      0 & \text{otherwise}
    \end{cases}
  \end{equation*}
  where $k_{[K]} \colon \cat SG{p+*+\mathrm{eab}} \to \Z$ is the
  co\we\ for $\cat SG{p+*+\mathrm{eab}}/G$.
\end{prop}
\begin{proof}
  We noted above that the \we\ $k^\bullet$ for $\cat SGp$ vanishes off
  the subposet $\cat SG{p+\mathrm{rad}}$ of $G$-radical $p$-subgroups.
  If $H$ is any $G$-radical $p$-subgroup of $G$, then the inclusion $H
  // \cat SG{p+\mathrm{rad}} \hookrightarrow H // \cat SGp$ is a
  homotopy equivalence by Bouc's theorem \cite{bouc84a} \cite[Theorem
  4.2]{gm:2012} because
  \begin{equation*}
    J // (H // \cat SGp) \not\simeq\ast \iff 
    J // \cat SGp  \not\simeq\ast \iff
    \cat S{N_G(J)/J}{p+*} \not\simeq\ast 
   \stackrel{\text{\tiny Lemma~\ref{lemma:quillenOpG}}}{\implies}
    J \in H// \cat SG{p+\mathrm{rad}}
  \end{equation*}
  for any $J \in H // \cat SGp$.  Thus the \we\ for $\cat SGp$ on $H$,
  $k^H = -\rchi( H // \cat SGp) = -\rchi( H // \cat
  SG{p+\mathrm{rad}})$, is the \we\ for $\cat SG{p+\mathrm{rad}}$ on
  $H$.   The \we\ on $\cat SG{p+\mathrm{rad}}$ is invariant under conjugation
  by $G$ so it induces the unique \we\ $k^{[H]} \colon \cat
  S{G}{p+\mathrm{rad}}/G \to \Z$ defined on the set of conjugacy
  classes of $G$-radical $p$-subgroups $[H]$ of $G$
  (Corollary~\ref{cor:uweSA}).

   The part about co\we s is clear as the nonidentity elementary
   abelian $p$-subgroups form a right ideal in the poset of
   nonidentity $p$-subgroups (\S\ref{sec:computechi}).
\end{proof}

\begin{defn}\label{defn:SGH[K]}
  When $H$ and $K$ are subgroups of $G$, $\cat SG{}(H,[K])$ is the
  number of conjugates of $K$ containing $H$, and $\cat SG{}([H],K)$
  is the number of conjugates of $H$ contained in $K$.
\end{defn}

In particular, $\cat SG{}(1,[K]) = |G : N_G(K)|$ is the length of $K$,
$\cat SG{}([H],G) = |G : N_G(H)|$ is the length of $H$, and $\cat
SG{}(H,[G]) = 1 = \cat SG{}([1],K)$.

\begin{lemma}\label{lemma:SGH[K]}
 For any two subgroups $H$, $K$ of $G$,
 \begin{equation*}
   \cat SG{}(H,[K]) = \frac{|\cat OG{}(H,K)|}{|\cat OG{}(K)|},
                    \quad
   \cat SG{}([H],K) 
   =\frac{|K|}{|H|} \frac{|\cat OG{}(H,K)|}{|\cat OG{}(H)|}
   = \frac{|\cat FG{}(H,K)|}{|\cat FG{}(H)|}, \quad
   \frac{\cat SG{}(H,[K])}{\cat SG{}([H],K)} = \frac{|N_G(H)|}{|N_G(K)|}
 \end{equation*}
\end{lemma}
\begin{proof}
 There is a
surjection
\begin{equation*}
  N_G(H,K) \to \{J \in [H] \mid J \leq K \} \colon g \to H^g
\end{equation*}
and the fibre over $J=H^x$ is $\{g \in G \mid H^g = H^x \} =
N_G(H)x$.
Since thus
\begin{equation*}
  N_G(H)  \backslash N_G(H,K) = \{ J \in [H] \mid J \leq K
  \}, \qquad
  N_G(H,K)/N_G(K) = \{ L \in [K] \mid H \leq L \}
\end{equation*}
we have
\begin{equation*}
  |N_G(H)| \cat SG{}([H],K) =  |N_G(H,K)| =
  \cat SG{}(H,[K]) |N_G(K)|
\end{equation*}
and also
\begin{align*}
  &|\cat OG{}(H,K)| = |N_G(H,K)| |K|^{-1} = 
   \cat SG{}(H,[K]) |\cat OG{}(K)| \\
  &|\cat FG{}(H,K)| = |C_G(H)|^{-1} |N_G(H,K)| = 
  |\cat FG{}(H)| \cat SG{}([H],K) 
\end{align*}
This finishes the proof.  
\end{proof}

The point of Proposition~\ref{prop:weSGprad} is that the \we\ $k^H$
for $\cat SGp$ restricts to the \we\ for the $G$-radical subposet
$\cat SG{p+\mathrm{rad}}$ and that the \we\ for $\cat
SG{p+\mathrm{rad}}/G$ can be computed by solving the linear equation
\begin{equation}
  \label{eq:lineqweS}
   \left( \cat SG{p+\mathrm{rad}}(H,[K]) 
   \right)_{H,[K] \in \cat SG{p+\mathrm{rad}}/G}
  \begin{pmatrix}
    \vdots \\ k^{[K]} \\ \vdots 
  \end{pmatrix}_{[K] \in \cat SG{\mathrm{rad}}/G} =
  \begin{pmatrix}
    \vdots \\ 1 \\ \vdots
  \end{pmatrix}
\end{equation}
Similarly, the co\we\ for $\cat SG{p+*}$ can be computed entirely
inside the subposet  $\cat SG{p+*+\mathrm{eab}}$ by solving the linear
equation
\begin{equation}
  \label{eq:lineqcowe}
  \begin{pmatrix}
    \hdots & k_{[H]} & \hdots
  \end{pmatrix}
   \left( \cat SG{p+\mathrm{eab}}([H],K) 
   \right)_{H,K \in \cat SG{p+\mathrm{eab}}/G} =
   \begin{pmatrix}
     \hdots & 1 & \hdots
   \end{pmatrix}
\end{equation}
If $k^{[K]}$ is the \we\ for $\left( \cat SG{p}(H,[K]) \right)_{H,K
  \in \cat SG{p}/G}$ and $k_{[H]}$ the co\we\ for $\left( \cat
  SG{p+*}(H,[K]) \right)_{H,K \in \cat SG{p+*}/G}$, then
$k^{[K]}/|\cat OG{p}(K)|$ is the \we\ for $[\cat OGp]$ and
$k_{[H]}/|\cat FG{p+*}(H)|$ the co\we\ for $[\cat FG{p+*}]$
\cite[\S2.4]{jmm_mwj:2010}.


\begin{exmp}\label{exmp:boreltits}  
  The radical $2$-subgroups of $\GL 3{\F_2}$ 
  \begin{equation*}
    U_\emptyset =
    \begin{pmatrix}
      1 & \ast & \ast \\ 0 & 1 & \ast \\ 0 & 0 & 1
    \end{pmatrix}, \quad
    U_{\{1\}} =
    \begin{pmatrix}
      1 & 0 & \ast \\ 0 & 1 & \ast \\ 0 & 0 & 1
    \end{pmatrix}, \quad
    U_{\{2\}} =
    \begin{pmatrix}
      1 & \ast & \ast \\ 0 & 1 & 0 \\ 0 & 0 & 1
    \end{pmatrix}, \quad
    U_{\Pi} = 1
  \end{equation*}
  are indexed up to conjugacy by subsets of the set
  $\Pi=\{\alpha_1,\alpha_2\}$ of fundamental roots for the root system
  $A_2$ according to the Borel--Tits theorem \cite[Theorem 3.1.3,
  Corollary 3.1.5]{GLSIII}. The \we\ $k^{[U_J]}= -\rchi(U_J//\cat
  S{\GL 3{\F_2}}{\mathrm{rad}}) = \rchi(\cat S{L_J}{2+*})$ for $\cat
  S{\GL 3{\F_2}}{\mathrm{rad}}/\GL 3{\F_2}$ lists the reduced \Euc s
  of the Levi complements $L_J = N_{\GL 3{\F_2}}(U_J)/U_J$
  \cite[Theorem 2.6.5]{GLSIII} and is the solution to the linear
  equation~\eqref{eq:lineqweS}
  \begin{equation*}
    \begin{pmatrix}
       1 &  0 & 0 & 0 \\  3 & 1 & 0 & 0 \\
       3 & 0 & 1 &  0 \\ 21 &  7 & 7 & 1
    \end{pmatrix}
    \begin{pmatrix}
       1 \\ -2 \\  -2 \\ 8
    \end{pmatrix} =
    \begin{pmatrix}
      1 \\ 1 \\ 1 \\ 1
    \end{pmatrix}
  \end{equation*}
  where the matrix is the table $(\cat S{\GL
    3{\F_2}}{2+\mathrm{rad}}(U_J,[U_K]))_{J,K \subset \Pi}$. In
  particular, $-\rchi(1 // \cat S{\GL 3{\F_2}}2) = -\rchi(\cat S{\GL
    3{\F_2}}{2+\ast}) = 8$. 
  
  The situation is completely different in the cross characteristic
  case.  With computer assistance one finds that $\SL 3{\F_3}$
  contains four $\SL 3{\F_3}$-radical subgroups of orders $ 16, 8, 4,
  1 $. The table $(\cat S{\SL 3{\F_3}}{2+\mathrm{rad}}(H,[K])$ and the
  \we\ $k^{[K]}$ for $\cat S{\SL 3{\F_3}}{2+\mathrm{rad}}/G$ satisfy
  the linear equation
  \begin{equation*}
    \begin{pmatrix}
      1 &  0 &  0 &   0 \\ 3 &  1 &  0 &  0 \\
      3 &  0 &  1 &  0 \\ 351 & 117 & 234 &  1
    \end{pmatrix}
    \begin{pmatrix}
       1 \\ -2 \\ -2 \\ 352
    \end{pmatrix} =
    \begin{pmatrix}
      1 \\ 1 \\ 1 \\ 1
    \end{pmatrix}
  \end{equation*}
  In particular, $-\rchi(\cat S{\SL 3{\F_3}}{2+*}) = 352$. 
\end{exmp}

\subsection{\Euc s of centralized subposets of  $\cat SG{p+*}$}
\label{sec:chiSGpA}
Suppose now that the group $A$ acts on $G$.  The centralizer of $A$ in
$\cat SG{p+*}$ is the $N_G(A)$-subposet
\begin{equation*}
C_{\catp SG*}(A) = \{ H \in \catp SG* \mid [H,A] \leq H \}  
\end{equation*}
of nontrivial $A$-normalized $p$-subgroups $H$ of $G$
(Definition~\ref{defn:centralizercat}).

As in the non-equivariant case (Lemma~\ref{lemma:quillenOpG}) this
centralizer poset is contractible as soon as $O_p(G)$ is nontrivial.


\begin{lemma}\label{lemma:OpHA}
  $O_p(G) \neq 1 \implies C_{\catp SG*}(A) \simeq \ast 
  \implies \chi(C_{\catp SG*}(A))=1$
\end{lemma}
\begin{proof}
  The characteristic subgroup $O_p(G)$ is a nonidentity $A$-normalized
  normal $p$-subgroup of $G$.  The natural transformations 
  $K \leq KO_p(G) \geq O_p(G)$ link the identity functor to a
  constant endofunctor of $C_{\cat SG{p+*}}(A)$.
\end{proof}


We can easily determine the unique \we\ and co\we\ in the sense of
Leinster \cite{leinster08} on $C_{\catp SG*}(A)$.

\begin{lemma}\label{lemma:wtA}
  The functions
 \begin{equation*}
   k^H = -\rchi(C_{H//\cat SGp}(A)) =  
    -\rchi(C_{\cat S{\cat OG{}(H)}{p+*}}(A)), \qquad
   k_K = -\rchi(C_{\cat SK{(1,K)}}(A))
 \end{equation*}
 are a \we\ for $C_{\cat SG{p}}(A)$ and a co\we\ for $C_{\cat
   SG{p+*}}(A)$ and the \Euc\ of $C_{\cat SG{p}}(A)$ is
\begin{equation*}
   \sum_{H \in C_{\catp SG{*+\mathrm{rad}}}(A)}  -\rchi(C_{\catp
    S{\cat OG{}(H)}*}(A)) =
   \chi(C_{\catp SG*}(A)) = 
  \sum_{K \in C_{\catp SG{*+\mathrm{eab}}}(A)} -\rchi(C_{\cat SK{(1,K)}}(A))
\end{equation*}
where the sum to the left runs over the nonidentity $G$-radical
$A$-normalized $p$-subgroups of $G$ and the sum to the right over the
nonidentity elementary abelian $A$-normalized $p$-subgroups of $G$.
\end{lemma}
\begin{proof}
  The adjunctions of \cite[\S6]{gm:2012}, $KR=N_K(H)/H$ and $\bar KL =
  K$ where $\bar K = K/H$,
\begin{center}
\begin{tikzpicture}[>=stealth']
  \matrix (dia) [matrix of math nodes, column sep=25pt, row sep=20pt]{
  H/\cat SG{p+*} & \cat S{\cat OG{}(H)}p \\};
 \draw[->] (dia-1-1.10) -- (dia-1-1.10 -| dia-1-2.west)
 node[pos=.5,above] {$R$};  
 \draw[->] (dia-1-2.180) -- (dia-1-2.180 -| dia-1-1.east)
 node[pos=.5,below] {$L$}; 
 \end{tikzpicture} \qquad
\begin{tikzpicture}[>=stealth']
  \matrix (dia) [matrix of math nodes, column sep=25pt, row sep=20pt]{
  H//\cat SG{p+*} & \cat S{\cat OG{}(H)}{p+*} \\};
 \draw[->] (dia-1-1.10) -- (dia-1-1.10 -| dia-1-2.west)
 node[pos=.5,above] {$R$};  
 \draw[->] (dia-1-2.180) -- (dia-1-2.180 -| dia-1-1.east)
 node[pos=.5,below] {$L$}; 
 \end{tikzpicture}
\end{center}
are $A$-adjunctions (Definition~\ref{defn:Gadj}) inducing
adjunctions
\begin{center}
\begin{tikzpicture}[>=stealth']
  \matrix (dia) [matrix of math nodes, column sep=25pt, row sep=20pt]{
  H/C_{\cat SG{p+*}}(A) & C_{\cat S{\cat OG{}(H)}{p}}(A) \\};
 \draw[->] (dia-1-1.10) -- (dia-1-1.10 -| dia-1-2.west)
 node[pos=.5,above] {$R$};  
 \draw[->] (dia-1-2.180) -- (dia-1-2.180 -| dia-1-1.east)
 node[pos=.5,below] {$L$}; 
 \end{tikzpicture} \quad
\begin{tikzpicture}[>=stealth']
  \matrix (dia) [matrix of math nodes, column sep=25pt, row sep=20pt]{
  H//C_{\cat SG{p+*}}(A) & C_{\cat S{\cat OG{}(H)}{p+*}}(A) \\};
 \draw[->] (dia-1-1.10) -- (dia-1-1.10 -| dia-1-2.west)
 node[pos=.5,above] {$R$};  
 \draw[->] (dia-1-2.180) -- (dia-1-2.180 -| dia-1-1.east)
 node[pos=.5,below] {$L$}; 
 \end{tikzpicture}
\end{center} 
between centralizer posets (Proposition~\ref{prop:adjunction}).  In
the notation of \cite[Theorem 3.7]{gm:2012}, the \we\ for $C_{\cat
  SG{p+*}}(A)$ is $k^H = -\rchi(H//C_{\cat SG{p+*}}(A)) =
-\rchi(C_{\cat S{\cat OG{}(H)}{p+*}}(A))$.

Similarly, the co\we\ for $C_{\cat SG{p+*}}(A)$ is $k_K =
-\rchi(C_{\cat SG{p+*}}(A))//K)$.  Since $\cat SG{p+*}//K = \cat
SK{(1,K)}$ it is clear that $C_{\cat SG{p+*}}(A))//K = C_{\cat
  SG{p+*}//K}(A) = C_{\cat SK{(1,K)}}(A)$.

The \Euc\ of $C_{\cat SG{p+*}}(A)$ is the sum of the values of the
\we\ (co\we). By Lemma~\ref{lemma:OpHA}, the \we\ of
Lemma~\ref{lemma:wtA} vanishes off the $G$-radical $p$-subgroups.

  The argument at the very beginning of this section shows that $k_K =
  \mu_A(1,K)$ where $\mu_A$ stands for the \Mb\ function of the poset
  $C_{\cat SG{p+*}}(A)$.  An equivariant version of the argument from
  \cite[Proposition 2.4]{kratzer_thevenaz84} shows that the \Mb\
  function $k_K=\mu_A(1,K)$ for $C_{\cat SG{p+*}}(A)$ is nonzero
  exactly when $K$ is a nonidentity elementary abelian $p$-subgroup.
\end{proof}


Let $\catp SG{*+\mathrm{rad}}$ ($\catp SG{*+\mathrm{eab}}$) be the
subposet of nontrivial $G$-radical $p$-subgroups (nontrivial
elementary abelian $p$-subgroups).  Lemma~\ref{lemma:wtA} shows that
the centralizer posets $C_{\catp SG{*+\mathrm{rad}}}(A)$ and $C_{\cat
  SG{*+\mathrm{eab}}}(A)$ have the same \Euc\ as $C_{\catp SG*}(A)$.
They are, in fact, homotopy equivalent:

\begin{prop}\label{prop:SGradeab}
  The inclusions
    \begin{tikzpicture}[>=stealth', baseline=(current bounding box.+0)]
  \matrix (dia) [matrix of math nodes, column sep=15pt, row sep=0pt]{
  C_{\cat SG{p+*+\mathrm{eab}}}(A) & C_{\cat SG{p+*}}(A) &
  C_{\cat SG{p+*+\mathrm{rad}}}(A)  \\};
  \draw[right hook->] 
  ($(dia-1-1.east)+(0,0.1)$) -- ($(dia-1-2.west)+(0,0.1)$); 
  \draw[left hook->] ($(dia-1-3.west)+(0,0.1)$) -- ($(dia-1-2.east)+(0,0.1)$);
 \end{tikzpicture}
  are homotopy equivalences.
\end{prop}
\begin{proof}
  This proof is similar to that of \cite[Theorem 6.1]{gm:2012} and
  goes back to Bouc \cite{bouc84a} and Quillen's Theorem A
  \cite[Theorem A]{quillen73}. It suffices to show that only the
  elementary abelian or the $G$-radical $p$-subgroups contribute to
  the homotopy type of $C_{\cat SG{p+*}}(A)$. More precisely, it is
  enough to show that
  \begin{align*}
    &\text{$K$ is not an elementary abelian $p$-subgroup of $G$} \implies 
    \text{$C_{\cat SG{p+*}}(A)//K$ is
      contractible} \\ 
    &\text{$H$ is not a $G$-radical $p$-subgroup of $G$} \implies 
    \text{$H//C_{\cat SG{p+*}}(A)$ is
      contractible} 
  \end{align*}
  when $H$ and $K$ are nonidentity $A$-normalized subgroups of $G$.
  In the first case, we use that $C_{\cat SG{p+*}}(A)//K= C_{\cat
    SK{(1,K)}}(A)$.  When $K$ is not elementary abelian, the Frattini
  subgroup $\Phi(K)$ is a nontrivial $A$-normalized subgroup of $K$
  and $H \leq H\Phi(K) \geq \Phi(K)$ defines a retraction of $C_{\cat
    SK{(1,K)}}(A)$. In the second case, we use the adjunction between
  $H//C_{\cat SG{p+*}}(A)$ and $C_{\cat S{\cat OG{}(H)}{p+*}}(A)$ and
  recall that the latter poset in contractible when $H$ is not
  $G$-radical by Lemma~\ref{lemma:OpHA}.
\end{proof}


\begin{prop}\label{prop:weCSGpradA}
  The \we\ $k^H \colon C_{\cat SGp}(A) \to \Z$ for $C_{\cat SGp}(A)$
  is given by
  \begin{equation*}
    k^H =
    \begin{cases}
      k^{[H]} & \text{$K$ is $G$-radical} \\ 0 & \text{otherwise}
    \end{cases}
  \end{equation*}
  where $k^{[H]} \colon  C_{\cat SG{p+\mathrm{rad}}}(A)/N_G(A) \to \Z$
  is the \we\ for $C_{\cat SG{p+\mathrm{rad}}}(A)/N_G(A)$.
  The co\we\ $k_K \colon C_{\cat SG{p+*}}(A) \to \Z$ for $C_{\cat
    SG{p+*}}(A)$ is given by
  \begin{equation*}
    k_K =
    \begin{cases}
      k_{[K]} & \text{$K$ is elementary abelian} \\
      0 & \text{otherwise}
    \end{cases}
  \end{equation*}
  where $k_{[K]} \colon C_{\cat SG{p+*+\mathrm{eab}}}(A)/N_G(A) \to
  \Z$ is the co\we\ for $C_{\cat SG{p+*+\mathrm{eab}}}(A)/N_G(A)$.
\end{prop}

We omit the proof of Proposition~\ref{prop:weCSGpradA} which is similar
to that of Proposition~\ref{prop:weSGprad} with Lemma~\ref{lemma:OpHA}
replacing Lemma~\ref{lemma:quillenOpG}.

\begin{proof}[Proof of Theorem~\ref{thm:global}.\eqref{thm:global0}]
  Let $K$ be an $A$-normalized $p$-subgroup of $G$.  The \we\ for
  $C_{\cat SG{p}}(A)$ (Lemma~\ref{lemma:wtA}) restricts to the \we\
  for the contractible left ideal $C_{K// \cat SG{p}}(A)$
  (\S\ref{sec:computechi}) of $A$-normalized $p$-subgroups containing
  $K$.  This means that
  \begin{equation*}
    1 = \chi(C_{K // \cat SG{p}}(A)) = \sum_{H \geq K}k^H =
    \sum_{H \geq K} -\rchi(C_{\cat S{\cat OG{}(H)}{p+*}}(A))
  \end{equation*}
  generalizing the nonequivariant Equation~\eqref{eq:affine}. 
\end{proof}

The next lemma applies in the special case where $A$ is a subgroup of
$G$ acting on $\cat SG{}$ by conjugation.

\begin{lemma}\cite[\S4]{webb87}\cite[Lemma 2.1.2]{webb91}\label{lemma:OpH}
  If $A$ is a subgroup of $G$,
  \begin{equation*}
    O_p(A) \neq 1 \implies C_{\catp SG*}(A) \simeq \ast
  \implies \chi(C_{\catp SG*}(A)) =1
  \end{equation*}
\end{lemma}
\begin{proof}
  The natural transformations $K \leq K O_p(A) \geq O_p(A)$ link the
  identity endofunctor to a constant endofunctor of $C_{\catp
    SG*}(A)$. The product $K O_p(A)$ is a subgroup as $A$ normalizes
  $K$. 
\end{proof}

\begin{cor}\label{cor:OpH}
  Consider the Brown poset $\cat SG{p+*}$ as a $G$-poset with the
  conjugation action.
  \begin{enumerate}
  \item \label{cor:OpH1} The $r$th, $r \geq 1$,
  reduced equivariant  \Euc\ is a weighted sum of reduced \Euc s
  \begin{equation*}
    \rchi_r(\cat SG{p+*},G) = 
    \frac{1}{|G|}\sum_{A \in \cat SG{p'+\mathrm{abe}}}
    \rchi((C_{\cat SG{p+*}}(A)))\varphi_r(A) =
    \sum_{[A] \in \cat SG{p'+\mathrm{abe}}/G}
    \rchi((C_{\cat SG{p+*}}(A)))\frac{\varphi_r(A)}{|N_G(A)|}
  \end{equation*}
  with contributions only from the $p$-regular abelian subgroups $A$ of
  $G$.
\item \label{cor:OpH2} The $r$th, $r \geq 2$, reduced equivariant
  Euler class function $\ralpha_r(\cat SG{p+*},G)$
  (Definition~\ref{defn:classfct}) vanishes off the $p$-regular
  classes.
\item \label{cor:OpH3} The $r$th, $r \geq 2$, reduced equivariant
  Euler class function has the form
  \begin{equation*}
    \ralpha_r(\cat SG{p+*},G) =
    \sum_{[C] \in \cat SG{p'+\mathrm{cyc}}/G}
    \frac{\wa_r(\cat SG{p+*},G)([C])}{|N_G(C):C|} 1_C^G  
  \end{equation*}
  where the sum runs over the set of conjugacy classes $[C]$ of
  $p$-regular cyclic subgroups and where the uniquely determined {\em
    Artin coefficients\/} $\wa_r(\cat SG{p+*},G)([C])$ are integers.
\item  \label{cor:OpH4} When $r=2$ the Artin coefficients satisfy the
  equation 
  \begin{equation*}
    0=\sum_{[C] \in [\cat SG{p'+\mathrm{cyc}}]}
 \wa_2(\cat SG{p+*},G)(C) \frac{1}{|N_G(C)|}
  \end{equation*}
   for all prime divisors $p$ of $|G|$.
  \end{enumerate}
\end{cor}
\begin{proof}
  \noindent \eqref{cor:OpH1}
  Use two of the reformulations of the equivariant \Euc\ from
  Proposition~\ref{prop:altequivchi}.
  Lemma~\ref{lemma:OpH} implies that only the $p$-regular abelian
  subgroups of $G$ contribute to the sum.

  \noindent \eqref{cor:OpH2}
  The $r$th equivariant reduced class function $\ralpha_r(\cat
  SG{p+*},G)$  takes $x \in G$ to
  \begin{equation*}
    \rchi_{r-1}(C_{\cat SG{p+*}}(x),C_G(x)) = 
    \frac{1}{|C_G(x)|} \sum_{y \in C_{r-1}(C_G(x))} \rchi(C_{\cat
      SG{p+*}}(\gen{x,y})) 
  \end{equation*}
  If $p$ divides the order of $x$, $O_p(\gen{x,y})$ is nontrivial and
  $\rchi(C_{\cat SG{p+*}}(\gen{x,y})) =0$ for all $y$ by
  Lemma~\ref{lemma:OpH}.

  \noindent \eqref{cor:OpH3} Since the value of $\ralpha_r(\cat
  SG{p+*},G)$ at $x \in G$ only depends on the subgroup $\gen x$
  generated by $x$ (Definition~\ref{defn:classfct}), the proof of
  Artin's induction theorem \cite[Theorem 5.21]{isaacs} shows that
  $\ralpha_r(\cat SG{p+*},G)$ can be decomposed as claimed.  Since
  $\ralpha_r(\cat SG{p+*},G)$ is supported on the $p$-regular classes,
  as we saw in item \eqref{cor:OpH2}, only $p$-regular cyclic
  subgroups are used in this decomposition. The coefficients are
  uniquely determined since the class functions $1_C^G$ as $C$ runs
  through the set of conjugacy classes of cyclic subgroups of $G$ are
  a basis for the vector space $\Q \otimes R_{\Q}(G)$ of $\Q$-linear
  combinations of rational characters \cite[\S13.1]{serre77}.

  \noindent \eqref{cor:OpH4} The virtual degree of the right hand side
  of the equation in item \eqref{cor:OpH3} is
  \begin{equation*}
    \sum_{[C] \in \cat SG{p'+\mathrm{cyc}}/G}
  \frac{\wa_2(\cat SG{p+*},G)(C)}{|N_G(C):C|} |G : C| =
 |G|\sum_{[C] \in \cat SG{p'+\mathrm{cyc}}/G}
 \wa_2(\cat SG{p+*},G)(C)\frac{1}{|N_G(C)|}
  \end{equation*}
  and on the left side it is $\ralpha_2(\cat SG{p+*},G)(1) =
  \rchi_1(\cat SG{p+*},G) = 0$ by Webb's theorem \cite[Proposition
  8.2.(i)]{webb87}.
\end{proof}

The {\em Artin coefficients\/} $\wa_r(\cat SG{p+*},G)(C)$ of
Corollary~\ref{cor:OpH}.\eqref{cor:OpH3} determine the $r$th
equivariant reduced \Euc\ as
\begin{multline}\label{eq:artincoeff}
  \rchi_r(\cat SG{p+*},G) = \gen{\ralpha_r(\cat SG{p+*},G),|C_G|}_G
   =
   \sum_{[C] \in \cat SG{p'+\mathrm{cyc}}/G} \wa_r(\cat SG{p+*},G)(C)
    \sum_{x \in C} \frac{|C_G(x)|}{|N_G(C)|} \\ =
    \sum_{[C] \in \cat SG{p'+\mathrm{cyc}}/G} \wa_r(\cat SG{p+*},G)(C)
    \sum_{\substack{x \in C \\ x \neq 1}} \frac{|C_G(x)|}{|N_G(C)|}
\end{multline}
because $\gen{1_C^G,|C_G|}_G = \gen{1_C,|C_G|}_C = \sum_{x \in C}
|C_G(x):C|$ by Frobenius reciprocity. Use
Corollary~\ref{cor:OpH}.\eqref{cor:OpH4} to get the final equality.
It is perhaps worth noting that the factor $\sum_{x \in
  C}|C_G(x)|/|N_G(C)|$ from Equation~\eqref{eq:artincoeff} is a
natural number.
\begin{prop}\label{prop:CGx/NGH}
  $|N_G(H)| \mid \sum_{x \in H} |C_G(x)|$ for any subgroup $H \leq G$.
\end{prop}
\begin{proof}
  Note that the function $x \to |C_G(x)|$ is constant over the orbits
  for the $N_G(H)$-action on $H$. For any element $x_0 \in H$, the
  contribution to the sum $\sum_{x \in H} |C_G(x)|$ from the orbit
  through $x_0$, $|N_G(H) : N_G(H) \cap C_G(x_0)| |C_G(x_0)| =
  |N_G(H)||C_G(x_0):C_G(x_0) \cap N_G(H)|$, is a multiplum of
  $|N_G(H)|$.  
\end{proof}

\begin{rmk}
  If $A \leq G$ with $O_p(A)$ and $O_p(C_G(A))$ both nontrivial then
  $C_{\catp SG*}(A) \simeq \ast \simeq \cat S{C_G(A)}{p+*}$ are both
  contractible.
In particular,
\begin{equation*}
  A \leq G, \; p \mid |A|, \; \text{$A$ abelian} \implies 
  \chi(C_{\catp SG*}(A)) = 1 =\chi(\cat S{C_G(A)}{p+*})
\end{equation*}
However, when $p \mid |A|$ and $A$ is nonabelian, the two \Euc s may
not be equal: Let $p=2$, $G=\Sigma_7$ the symmetric group, and
$A=\Sigma_3$ the subgroup $A=\gen{(1,2),(1,2,3)} \leq \Sigma_7$. Then
$C_G(A) = \Sigma_4$, $|C_G(A)|_2 = 8$, and $\rchi(\cat S{C_G(A)}{2+*})
= 0$ since $O_2(\Sigma_4) \neq 1$.  However, it turns out that
$\rchi(C_{\cat SG{2+*}}(A))=-8$. Thus $\chi(C_{\catp SG*}(A)) \neq
\chi(\cat S{C_G(A)}{p+*})$ in this case.
\end{rmk}

\begin{rmk}\label{rmk:rchiSGpK}  
  The $r$th, $r \geq 1$, reduced equivariant \Euc\ is an integer
  valued function $\rchi_r(\cat SG{p+*},-)$ in the second variable
  (Corollary~\ref{cor:xXCG}) on the set of conjugacy classes of
  subgroups of $G$. Proposition~\ref{prop:altequivchi} shows that for
  any subgroup $K$ of $G$
 \begin{equation*}
   \rchi_r(\cat SG{p+*},K) = \frac{1}{|K|}
   \sum_{[A] \in \cat SG{p'+\mathrm{eab}}/G}
   \rchi(C_{\cat SG{p+*}}(A)) \varphi_r(A) 
   \cat SG{}([A],K)
 \end{equation*}
 or, in matrix form,
 \begin{equation*}
   \begin{pmatrix}
   \hdots & \chi(C_{\cat SG{p+*}}(A)) \varphi_r(A) & \hdots   
   \end{pmatrix}_{A \in \cat SG{p'+\mathrm{abe}}/G}
   \begin{pmatrix}
     \cat SG{}([A],K)
   \end{pmatrix}_{\substack{A \in \cat SG{p'+\mathrm{abe}}/G \\ 
   K \in \cat SGp/G}} =
 \begin{pmatrix}
   \hdots & \rchi_r(\cat SG{p+*},K) & \hdots
 \end{pmatrix}_{K \in \cat SGp/G}
 \end{equation*}
 It suffices to let the sum range over the $p$-singular abelian
 subgroups of $G$ by Lemma~\ref{lemma:OpH} and we may replace $C_{\cat
   SG{p+*}}(A)$ by $C_{\cat SG{p+*+\mathrm{rad}}}(A)$ or $C_{\cat
   SG{p+*+\mathrm{eab}}}(A)$ by Proposition~\ref{prop:SGradeab}. When
 $K=1$ is the trivial subgroup we get the usual reduced \Euc\
 $\rchi(\cat SG{p+*})$ and when $K=G$ we get the $r$th reduced
 equvariant \Euc\ $\rchi_r(\cat SG{p+*},G)$ that we have focused on
 until now.  For the alternating group $G=A_5$ at $p=2$ the \Euc s
 $-\chi_r(\cat SG{2+*},K)$ for $r=1,2,3$ and for the subgroups $K$ of
 $G$ are
  \begin{center}  
    \begin{tabular}[t]{>{$}c<{$} | *{9}{>{$}r<{$}} }
      |K| &
      1 & 2  & 3 & 5 &   4 &   6  &  10 & 12 & 60 \\ \hline\noalign{\smallskip}
       -\rchi_1(\cat SG{2+*},K) &
        -4 & -2 & -2 & 0 & -1 & -1 & 0 & -1 & 0  \\
        -\rchi_2(\cat SG{2+*},K) &
         -4 & -2 & -4 & 4 & -1 & -2 & 2 & -3 & 1 \\
         -\rchi_3(\cat SG{2+*},K) &
         -4 & -2 & -10 & 24 & -1 & -5 & 12 & -9 & 8
    \end{tabular}
  \end{center}
  and the simple group  $G=\GL 3{\F_2}$, $p=2$ they are
  \begin{center} 
     \begin{tabular}[t]{>{$}c<{$} | *{15}{>{$}r<{$}} }
      |K| & 
       1 & 2 & 3 & 7 & 4 & 4 & 4 & 6 & 21 & 8 & 12 & 12 & 24 & 24 &
       168 \\ \hline\noalign{\smallskip}
     -\rchi_1(\cat SG{2+*},K) &
      8 & 4 & 2 & 2 & 2 & 2 & 2 & 1 & 0 & 1 & 0 & 0 & 0 & 0 & 0 \\
      -\rchi_2(\cat SG{2+*},K) &
       8 & 4 & 0 & 8 & 2 & 2 & 2 & 0 & 0 & 1 & -2 & -2 & -1 & -1 & 1  \\
      -\rchi_3(\cat SG{2+*},K) &
      8 & 4 & -6 & 50 & 2 & 2 & 2 & -3 & 8 & 1 & -8 & -8 & -4 & -4 & 12
     \end{tabular}
  \end{center} 
  No interpretations similar to those of Remark~\ref{exmp:3conjs} of
  the intermediate \Euc s $\chi_r(\cat SG{p+*},K)$, $r=1,2,3$, with
  nontrivial and proper subgroup $K$ are known.
\end{rmk}

\subsection{Ideals in the centralized poset $C_{\cat SG{p+*}}(A)$}
\label{sec:computechi}
In this subsection we discuss the relation between the centralizer
subposet $C_{\cat SG{p+*}}(A)$ and the poset $\cat S{C_G(A)}{p+*}$ of the
centralizer subgroup for an $A$-group $G$.

Let $\cat S-{}$ and $\cat S+{}$ be subposets of a (finite) poset $\cat
S{}{}$. By definition, $\cat S-{}$ is a left ideal and $\cat S+{}$ a
right ideal if
\begin{equation*}
  \forall a,b \in \cat S{}{} \colon
  \cat S-{} \ni a < b  \implies b  \in \cat S-{}, \qquad
  \forall a,b \in \cat S{}{} \colon
  a < b \in \cat S+{} \implies a  \in \cat S+{}
\end{equation*}
The \we\ on the left ideal $\cat S-{}$, $k^a = -\rchi(a//\cat S-{}) =
-\rchi(a//\cat S{}{})$, $a \in \cat S-{}$, is the restriction of the
\we\ on $\cat S{}{}$, and the co\we\ on the right ideal $\cat S+{}$,
$k_b = -\rchi(\cat S+{}//b) = -\rchi(\cat S{}{}//b)$, $b \in \cat
S+{}$, is the restriction of the co\we\ \cite[Theorem 3.7]{gm:2012}
\cite[Remark 2.6]{jmm_mwj:2010}.  $\cat S-{}$ is a left ideal if and
only if the complement $\cat S+{} = \cat S{}{} - \cat S-{}$ is a right
ideal.

\begin{prop}\label{prop:S+-}
  Let $\cat S{}{}$ be a finite poset.
  \begin{enumerate}
  \item \label{prop:S+-1} Suppose that $\cat S{}{} = \cat S-{} + \cat
    S+{}$ where $\cat S-{}$ is a left ideal and $\cat S+{}$ a right
    ideal. Then
    \begin{equation*}
  \chi(\cat S{}{}) = \chi(\cat S{-}{}) + \chi(\cat S{+}{}) -
  \sum_{
  \begin{subarray}c
    \cat S-{} \ni a > b \in \cat S+{}
  \end{subarray}}
   \rchi(a//\cat S{}{}) \rchi(\cat S{}{}//b)
  \end{equation*}
\item \label{prop:S+-2} Suppose that $\cat S{}{} = \cat S1{} \cup \cat
  S2{}$ is the union of two left or right ideals. Then
  \begin{equation*}
    \chi(\cat S{}{}) = \chi(\cat S1{}) + \chi(\cat S2{}) - \chi(\cat
    S1{} \cap \cat S2{})
  \end{equation*}
  \end{enumerate}
\end{prop}
\begin{proof}
  \noindent\eqref{prop:S+-1}  Simplices in $\cat S{}{}$ not entirely 
  contained in $\cat S-{}$ or $\cat S+{}$ contain $b<a$ for some $b
  \in \cat S+{}$ and some $a \in \cat S-{}$. They have the form $b<a$,
  $b<a<a_0<\cdots<a_d$, $b_0<\cdots<b_e<b<a$, or
  $b_0<\cdots<b_e<b<a<a_0<\cdots<a_d$. Their contribution to the \Euc\
  of $\cat S{}{}$ is $-1+\chi(a//\cat S{}{})+\chi(\cat
  S{}{}//b)-\chi(a//\cat S{}{})\chi(\cat S{}{}//b) =
  -(\chi(a//\cat S{}{})-1)(\chi(\cat S{}{}//b)-1) = -\rchi(a//\cat
  S{}{})\rchi(\cat S{}{}//b)$. 

  \noindent \eqref{prop:S+-2} Mayer--Vietoris.
\end{proof}

More generally, in case $\cat S{}{} = \bigcup_{j \in J} \cat S j{}$
is the union of finitely many left or right ideals $\cat Sj{}$ then
there is an inclusion-exclusion principle
\begin{equation*}
  \chi(\cat S{}{}) = \sum_{K \subset J} 
  (-1)^{|K|-1} \chi( \bigcap_{k \in K} \cat S k{})
\end{equation*}
as we see by induction from
Proposition~\ref{prop:S+-}.\eqref{prop:S+-2}. 




Suppose that $A$ is a $p$-regular group acting on $G$ and let $H$ be
an $A$-normalized $p$-subgroup of $G$.  Then $H = [H,A] C_H(A)$ is the
product of the normal commutator subgroup $[H,A]$ and centralizer
subgroup $C_H(A)$ \cite[Theorem~5.3.5]{gorenstein68}.  Because
\begin{itemize}
\item $H \leq K \iff H \leq [K,A]$ if $[H,A]=H$
\item $H \leq K \iff H \leq C_K(A)$ if $[H,A]=1$ 
\end{itemize}
there are adjunctions
\begin{center}
\begin{tikzpicture}[>=stealth']
  \matrix (dia) [matrix of math nodes, column sep=55pt, row sep=20pt]{
   \cat S1{} = \{K \mid [K,A] \neq 1 \}  &  
   \{ H \mid [H,A]=H \}   \\
   \cat S2{} = \{K \mid C_K(A) \neq 1 \}  & 
   \{ H \mid [H,A]= 1\} \\};
   
 \draw[->] (dia-1-1.5) -- (dia-1-1.5 -| dia-1-2.west)
 node[pos=.5,above] {$K \to [K,A]$};  
 \draw[->] (dia-1-2.182) -- (dia-1-2.182 -| dia-1-1.east)
 node[pos=.5,below] {$H \leftarrow H$};

 \draw[->] (dia-2-1.5) -- (dia-2-1.5 -| dia-2-2.west)
 node[pos=.5,above] {$K \to C_K(A)$};  
 \draw[->] (dia-2-2.182) -- (dia-2-2.182 -| dia-1-1.east)
 node[pos=.5,below] {$H \leftarrow H$};
 \end{tikzpicture}
\end{center}
between subposets of $C_{\cat SG{p+*}}(A)$. To set up the first
adjunction we use that $[K,A,A] = [K,A]$
\cite[Theorem~5.3.6]{gorenstein68}. We note that
\begin{equation*}
   \{ K \mid [K,A] \neq 1 \} = \{ K \mid C_K(A) \neq K \}, \qquad
   \{ H \mid [H,A]= 1\}  = \{H \mid C_H(A) =H \} = \cat S{C_G(A)}{p+*}
\end{equation*}
where we use that $[K,A] = 1 \iff C_K(A)=K$.

\begin{cor}\label{cor:S+-}
  Let $A$ be a $p$-regular group acting on $G$.
  \begin{enumerate}
  \item The left ideal $\cat S1{} = \{ K \in C_{\cat SG{p+*}}(A) \mid
    [K,A] \neq 1 \}$ and the subposet $\{ H \in C_{\cat SG{p+*}}(A)
    \mid [H,A]=H \}$ have the same \Euc s. \label{cor:S+-1}
  \item The left ideal $\cat S2{} = \{ K \in C_{\cat SG{p+*}}(A) \mid
    C_K(A) \neq 1 \}$ and the poset $\cat S{C_G(A)}{p+*}$ of the
    centralizer subgroup have the same \Euc s. \label{cor:S+-2}
\item \label{cor:S+-3} The difference between the \Euc s of the
  centralizer subposet, $C_{\cat SG{p+*}}(A)$, and the 
  centralizer subgroup poset, $\cat S{C_G(A)}{p+*}$, is
  \begin{equation*}
    \chi(C_{\cat SG{p+*}}(A))-\chi(\cat S{C_G(A)}{p+*})  =
    \chi(\{K \in C_{\cat SG{p+*}}(A) \mid C_K(A) \lneqq K \}) -
    \chi(\{ K \in C_{\cat SG{p+*}}(A) \mid 1 \lneqq C_K(A) \lneqq K \}) 
  \end{equation*}
  \end{enumerate}
\end{cor}
\begin{proof}
  \noindent \eqref{cor:S+-1} -- \eqref{cor:S+-2} Adjoint finite posets
  have identical \Euc s \cite[Proposition 2.4]{leinster08}.

  \noindent \eqref{cor:S+-3} Apply
  Proposition~\ref{prop:S+-}.\eqref{prop:S+-2} with $\cat S{}{} =
  C_{\cat SG{p+*}}(A) = \cat S1{} \cup \cat S2{}$ where $\cat S1{}$
  and $\cat S2{}$ are the two left ideals from \eqref{cor:S+-1} --
  \eqref{cor:S+-2}
 and use that $\chi(\cat
  S2{}) = \chi(\cat S{C_G(A)}{p+*})$ by \eqref{cor:S+-2}.
\end{proof}

The Mayer--Vietoris relation of
Corollary~\ref{cor:S+-}.\eqref{cor:S+-3} may also be written as 
\begin{equation*}
 \chi(C_{\cat SG{p+*}}(A))=
 \chi(\{ H \mid [H,A]=1 \}) +
 \chi(\{ H \mid [H,A] = H \}) -
 \chi(\{ H \mid  1 \lneqq C_H(A) \lneqq H \})
\end{equation*}
This identity is tautological when $A$ is trivial.

Since the \we\ (co\we) of $C_{\cat SG{p+*}}(A)$ is concentrated on the
$G$-radical $p$-subgroups (elementary abelian $p$-subgroups)
(Lemma~\ref{lemma:wtA}), the left (right) ideals of
Corollary~\ref{cor:S+-}.\eqref{cor:S+-1}--\eqref{cor:S+-3} have the
same \Euc s as the corresponding left ideals in $C_{\cat
  SG{p+*+\mathrm{rad}}}(A)$ ($C_{\cat SG{p+*+\mathrm{eab}}}(A)$).

\begin{cor}\label{cor:SC}
  Let $A$ be a $p$-regular group acting on $G$. Suppose that $C_K(A)
  \neq 1$ for all nonidentity $G$-radical $p$-subgroups $K$ of $G$. Then
  $\chi(\cat S{C_G(A)}{p+*}) = \chi(C_{\cat SG{p+*}}(A))$.
\end{cor}
\begin{proof}
  The subposets $\{K \in C_{\cat SG{p+*+\mathrm{rad}}}(A) \mid C_K(A)
  \lneqq K \}$ and $\{ K \in C_{\cat SG{p+*+\mathrm{rad}}}(A) \mid 1
  \lneqq C_K(A) \lneqq K \}$ are identical so $\chi(\cat
  S{C_G(A)}{p+*}) = \chi(C_{\cat SG{p+*}}(A))$ by
  Corollary~\ref{cor:S+-}.\eqref{cor:S+-3}.
\end{proof}

\begin{exmp}
When $G=\Sigma_7$, $p=2$, and $A$ is any of the 
six  abelian $2$-regular subgroups
of $G$, the reduced \Euc s of Corollary~\ref{cor:S+-}.\eqref{cor:S+-3}
are
\begin{center} 
   \begin{tabular}[t]{>{$}c<{$} | *{6}{>{$}r<{$}} }
     A & 1 & 3 & 3 & 5 & 7 & 3 \times 3 \\ \hline \noalign{\smallskip}
   \rchi(C_{\cat SG{p+*}}(A)) & 160 &-8 &4 &0 &-1 &1  \\   
   \rchi(\cat S{C_G(A)}{p+*}) & 160 & 0 & 2 & 0 & -1 &-1 \\
   \rchi(\{K \mid C_K(A) \lneqq K \})  &-1 &3 &4 &-1 &-1 &1 \\
   \rchi(\{K \mid 1 \lneqq C_K(A) \lneqq K\}) &-1 &11 &2 &-1 &-1 &-1 \\
   |C_G(A)|_p &16 &8 &2 &2 &1 &1
   \end{tabular}
   \end{center}
   We note that $ \rchi(C_{\cat SG{p+*}}(A))$ and $\rchi(\cat
   S{C_G(A)}{p+*})$ may not be equal, that both reduced \Euc s are
   divisible by the $p$-part $|C_G(A)|_p$, and that, in each column,
   the differences of Corollary~\ref{cor:S+-}.\eqref{cor:S+-3} between
   the first two numbers and the next two numbers are equal.
\end{exmp}

\section{The orbit category of $p$-subgroups}
\label{sec:orbit}

Write $\cat OG{}$ for the category of subgroups, $H$ and $K$, of $G$
with \m\ sets and auto\m\ groups
\begin{equation*}
  \cat OG{}(H,K) = N_G(H,K)/K, \qquad \cat OG{}(H) = N_G(H)/H
\end{equation*}
where $N_G(H,K) = \{x \in G | H^x \leq K \}$ is the transporter set.
In other words, $\cat OG{}$ is the finite $\mathrm{EI}$-category whose
objects are the transitive $G$-orbits $G/H$ and whose \m s are the
left $G$-maps $G/H \to G/K$ between the orbits: The effect of $x \in
N_G(H,K)/K$ is $G/H \ni gH \to gxK \in G/K$ for any $g \in G$.  $\cat
OGp$ is the full subcategory of $\cat OG{}$ generated by all the
$p$-subgroups of $G$.

\subsection{\Euc\ of the orbit category $\cat OGp$}
\label{sec:chiOG}
Frobenius proved in $1907$ that the number $|G_p|$ of $p$-singular
elements in $G$ is divisible by $|G|_p$
\cite{frobenius:1907,isaacsrobinson} \cite[Corollary~41.11]{cr}
\cite[11.2, Corollary~2]{serre77}.

\begin{thm}[Frobenius $1907$]
\label{thm:frobenius}
  $|G|_p \mid |G_p|$
\end{thm}

K.\@S.\@ Brown proved in $1975$ that the reduced \Euc\ of the Brown
poset $\cat SG{p+*}$ is divisible by $|G|_p$ \cite[Corollary
2]{brown75}. This was later reproved with new arguments by Quillen
\cite[Corollary 4.2]{quillen78} or Webb \cite[Theorem 8.1]{webb87}.

\begin{thm}[Brown $1975$]  \label{thm:brown}
  $|G|_p \mid \rchi(\cat SG{p+*})$
\end{thm}

In this section we give yet another argument for Brown's theorem using
the \Euc\ of the orbit category $\cat OGp$. This argument will show
that the two theorems of Frobenius and Brown are, in some sense,
equivalent (and thus equivalent to the Sylow theorem). We finish the
section by proving Theorem~\ref{thm:global}.

\begin{lemma}\label{lemma:Gp}
  The number of $p$-singular elements in $G$ is
  \begin{equation*}
    |G_p| = 1+ \sum_{1<C\leq G}(1-p^{-1})|C| =
    p^{-1} +  \sum_{1 \leq C\leq G}(1-p^{-1})|C|
  \end{equation*}
  where the sum is over all cyclic $p$-subgroups $C$ of $G$.
\end{lemma}
\begin{proof}
  Declare two $p$-singular elements to be equivalent of they generate
  the same cyclic subgroup. The set of equivalence classes is the set
  of cyclic $p$-subgroups $C$ of $G$.  The number of elements in the
  equivalence class $C$ is the number of generators of $C$: $1$ if
  $|C|=1$ and $|C|-p^{-1}|C|$ if $|C| > 1$.
\end{proof}

Let $V$ be an elementary abelian $p$-group (or a finite dimensional
vector space over the finite field $\F_p$ with $p$ elements).

\begin{prop}\label{prop:cowtOV}
  The function
  \begin{equation*}
    k_U =
    \begin{cases}
      |V|^{-1} & \dim U = 0 \\
      (p-1)|V|^{-1}  & \dim U = 1 \\
      0 & \text{otherwise}
    \end{cases}
  \end{equation*}
is a co\we\ for the orbit category $\cat OV{}$.
\end{prop}
\begin{proof}
  The assertion is that
  \begin{equation*}
    \sum_{U_1 \leq U_2}k_{U_1} \frac{|V|}{|U_2|} = 1
  \end{equation*}
  for any subspace $U_2$ of $V$. This is easily verified.
\end{proof}

Let $\cat OV{[1,V)}$ be the full subcategory of $\cat OV{}$ generated
by all objects but the final object $V$. 

\begin{cor}\label{cor:cowtOV}
  The \Euc\ of $\cat OV{[1,V)}$ is
  \begin{equation*}
    \chi(\cat OV{[1,V)}) =
    \begin{cases}
      0 & \dim V =0 \\
      p^{-1} & \dim V = 1 \\
      1 & \dim V > 1
    \end{cases}
  \end{equation*}
\end{cor}
\begin{proof}
  $\cat OV{[1,V)}$ is the $0$-object category when $\dim V=0$ and the
  $1$-object category given by the group $V$ when $\dim V=1$. When
  $\dim V>1$, $\chi(\cat OV{[1,V)}) = \chi(\cat OV{})$, because the
  co\we\ for $\cat OV{}$ restricts to a co\we\ for the right ideal
  $\cat OV{[1,V)}$ \cite[Remark 2.6]{jmm_mwj:2010} and it has value
  $0$ at the deleted object $V$. Now note that as the category
  $\chi(\cat OV{})$ has a final object, its \Euc\ is $1$
  \cite[Examples 2.3.(d)]{leinster08}.
\end{proof}

\begin{prop}\cite[Theorem 4.1]{jmm_mwj:2010}\label{prop:coweOGp}
  The function
  \begin{equation*}
    k_K =
    \begin{cases}
      |G|^{-1} & K=1 \\
      |G|^{-1}(1-p^{-1})|K| & \text{$K>1$ cyclic} \\
      0 & \text{otherwise}
    \end{cases}
  \end{equation*}
  is a co\we\ for $\cat OGp$ and the \Euc\
  \begin{equation*}
    \chi(\cat OGp) = \frac{|G_p|}{|G|}
  \end{equation*}
  is the density of the $p$-singular elements in $G$. 
\end{prop}
\begin{proof}
  The co\we\ for $\cat OGp$ is the function
  \begin{equation*}
    k_K = \frac{-\rchi(\cat OGp//K)}{|[K]| |\cat OG{}(K)|} =
    \frac{-\rchi(\cat OGp//K)}{|G:K|}
  \end{equation*}
  defined for any $p$-subgroup $K$ of $G$ \cite[Theorem 3.7]{gm:2012}.
  There are equivalences of categories
    \begin{equation*}
      i_K \colon \cat OK{} \to \cat O{G}{}/K, \qquad
      i_K \colon \cat OK{[1,K)} \to \cat O{G}{}//K
    \end{equation*}
    On objects, $Hi_K = 1K \in N_G(H,K)/K = \cat OG{}(H,K)$, for any
    subgroup $H$ of $K$. We observe that there is an obvious
    identification of \m\ sets,
    \begin{equation*}
      \cat OK{}(H_1,H_2) = (\cat OG{}/K)(H_1i_K,H_2i_K)
    \end{equation*}
    and we use this identification to define $i_K$ on \m\ sets. By
    construction, $i_K$ is full and faithful, and as it is also
    essentially surjective on objects, $i_K$ is an equivalence of
    categories. Finally, we have that
    \begin{equation*}
       \chi(\cat OK{[1,K)}) = \chi(\cat OV{[1,V)})  
    \end{equation*}
    where $V=K/\Phi(K)$ is the Frattini quotient of $K$: There is, by
    the the proof of \cite[Lemma 5.1.(c)]{gm:2012}, an adjunction
    between these two categories so that they must have the same \Euc
    s \cite[Proposition 2.4]{leinster08}. To arrive at the description
    of the co\we\ we recall that the $p$-group $K$ is cyclic if and
    only if its Frattini quotient is $1$-dimensional \cite[Corollary
    5.1.2]{gorenstein68} and use Corollary~\ref{cor:cowtOV}. Now the
    \Euc, the sum of the values of the co\we, can be computed thanks
    to the counting formula of Lemma~\ref{lemma:Gp}.
\end{proof}

\begin{prop}\cite[Theorem 1.3.(4)]{jmm_mwj:2010}\label{prop:weOGp}
  The function
  \begin{equation*}
    k^H 
       = \frac{-\rchi(\cat S{\cat OG{}(H)}{p+*})}{|G:H|}
  \end{equation*}
  is a \we\ for $\cat OGp$ and the \Euc\ is
  \begin{equation*}
    \chi(\cat OGp) 
    = \sum_{H} 
    \frac{-\rchi(\cat S{\cat OG{}(H)}{p+*})}{|G:H|}
    = \sum_{[H]} 
    \frac{-\rchi(\cat S{\cat OG{}(H)}{p+*})}{|\cat OG{}(H)|}
  \end{equation*}
  where the first sum is over the set of $p$-subgroups $H$ of $G$ and
  the second one over the set of conjugacy classes of such subgroups.
\end{prop}
\begin{proof}
Lemma~\ref{lemma:SGH[K]} implies that the 
\we s for $[\cat OGp]$ and $\cat OGp$ are
\begin{equation*}
   \frac{k^{[K]}}{|\cat OG{}(K)|}, \qquad
  \frac{k^{[K]}}{|\cat OG{}(K)|}\frac{1}{|G:N_G(K)|} = \frac{k^{[K]}}{|G:K|}
\end{equation*}
where $k^{[K]} = -\rchi(\cat S{\cat OGp(K)}{p+*})$ is the \we\ for
$\cat SGp/G$.
\end{proof}

By comparing the two expressions from Propositions~\ref{prop:coweOGp}
and \ref{prop:weOGp} for the \Euc\ of $\cat OGp$ we obtain the global
formula
\begin{equation}
  \label{eq:globalnoneq}
   \sum_{H \in \cat SG{p+\mathrm{rad}}}
  -\rchi(\cat S{\cat OG{}(H)}{p+*})|H| = |G_p|
\end{equation}
For our purposes it will be convenient to isolate the contribution
from the trivial subgroup and rewrite Equation~\eqref{eq:globalnoneq} on
the form
\begin{equation}
  \label{eq:chiOG}
    |G_p| + 
    \rchi(\cat SG{p+*}) + \sum_{[H] \neq 1} 
   \frac{\rchi(\cat S{\cat OG{}(H)}{p+*})}{|\cat OG{}(H)|_p}
   \frac{|G|}{|\cat OG{}(H)|_{p'}} = 0
\end{equation}
that will allow us to verify that the theorems of Frobenius and Brown
are equivalent.

\begin{thm}\label{thm:frobeniusbrown}
  Given relation \eqref{eq:chiOG}, Frobenius'
  Theorem~\ref{thm:frobenius} and Brown's Theorem~\ref{thm:brown} are
  equivalent.
\end{thm}
\begin{proof}
  Assume first that Frobenius' Theorem~\ref{thm:frobenius} holds.  In
  Equation~\eqref{eq:chiOG}, we may assume that
  \begin{itemize}
  \item $\rchi(\cat S{\cat OG{}(H)}{p+*})/ |\cat OG{}(H)|_p$ is an integer
  when $H$ is nontrivial (as part of an inductional argument)
\item $|G|/|\cat OG{}(H)|_{p'}$ is an integer divisible by $|G|_p$ (as
  $|\cat OG{}(H)|$ divides $|G|$)
  \end{itemize} 
  Thus every term in the sum is divisible by $|G|_p$ and so is $|G_p|$
  by assumption.  We conclude that $\rchi(\cat SG{p+*})$ is divisible
  by $|G|_p$. This is Brown's Theorem~\ref{thm:brown}.

  Next assume that Brown's Theorem~\ref{thm:brown} holds.  In
  Equation~\eqref{eq:chiOG}
  \begin{itemize}
  \item $|G|/|\cat OG{}(H)|_{p'}$ is an integer divisible by $|G|_p$
\item $\rchi(\cat S{\cat OG{}(H)}{p+*})/|\cat OG{}(H)|_p$ is an integer
\item $\rchi(\cat SG{p+*})$ is divisible by $|G|_p$
\end{itemize}  
and thus $|G|_p$ divides for $|G_p|$. This is Frobenius'
Theorem~\ref{thm:frobenius}.
\end{proof}

See \cite[Theorem~6.3]{HIO89} for one direction of
Theorem~\ref{thm:frobeniusbrown} in an even more general context.

\begin{exmp}\label{exmp:chevalley}
  Let $K \in \mathit{Lie}(p)$ be a finite group of Lie type in
  defining characteristic $p$. Then $K = O^{p'}C_{\bar K}(\sigma)$ is
  the subgroup of $C_{\bar K}(\sigma)$ generated by its $p$-singular
  elements where $(\bar K,\sigma)$ is a $\sigma$-setup for $K$
  \cite[Definitions 2.1.1--2.2.2]{GLSIII}.  We shall assume that $K =
  \Sigma(q)$ is an untwisted group of Lie type \cite[Definition
  2.2.4]{GLSIII} and investigate Equations~\eqref{eq:globalnoneq} and
  \eqref{eq:affine} in this case.

  Let $(\Sigma,\Pi)$ be the root system for the simple algebraic group
  $\bar K$ over $\bar \F_p$. Write $r(\Sigma)$ for the rank of
  $\Sigma$ and $\Sigma^+$ for the set of positive roots.  
  For any subset $J$ of the set $\Pi$ of
  fundamental roots, there is an associated parabolic subgroup $P_J$
  \cite[Definition 2.6.4]{GLSIII} and $P_J = U_J \rtimes L_J$, $U_J =
  O_p(P_J)$, $P_J = N_K(U_J)$ where $L_J$ is the Levi complement
  \cite[Theorem 2.6.5]{GLSIII}.

  
  We first note that
  \begin{equation}\label{eq:eulerKp}
     -\rchi(\cat S{\Sigma(q)}{p+*}) = 
     (-1)^{r(\Sigma)} q^{|\Sigma^+|}  
  \end{equation}
  because $\cat S{\Sigma(q)}{p+*}$ is homotopy equivalent to the
  building of $\Sigma(q)$ \cite[Theorem 3.1]{quillen78} which, by the
  Solomon--Tits theorem \cite[Theorem 4.73]{abramenko_brown2008}, has
  the homotopy type of a wedge of $q^{|\Sigma^+|}$ spheres of
  dimension $r(\Sigma)-1$.  Since the Borel--Tits theorem tells us
  that any $\Sigma(q)$-radical $p$-subgroup of $\Sigma(q)$ is
  conjugate to precisely one of the subgroups $U_J$
  \cite[Theorem~3.1.3, Corollary~3.1.5]{GLSIII}, the identity
  \begin{equation*}
    |\Sigma(q)_p| = 
    \sum_{J \subset \Pi} -\rchi(\cat S{L_J}{p+*}) |U_J|  |\Sigma(q) : P_J| 
  \end{equation*}
  is the manifestation of Equation~\eqref{eq:globalnoneq} for
  $G=\Sigma(q)$. Here,
  \begin{equation*}
    -\rchi(\cat S{L_J}{p+*})|U_J| = (-1)^{|J|} |L_J|_p |U_J| =
    (-1)^{|J|} |P_J|_p = (-1)^{|J|}q^{|\Sigma^+|}  
  \end{equation*}
  where $\Sigma^+$ is the set of positive roots of $\Sigma$.  This
  follows from Equation~\eqref{eq:eulerKp} because $M_J=O^{r'}(L_J)$
  is a central product of groups from $\mathit{Lie}(p)$ \cite[Theorem
  2.6.5.(f)]{GLSIII} and $-\rchi(\cat S-{p+*})$ is a multiplicative
  function \cite[Theorem 6.1]{jmm_mwj:2010}. Since also,
  \begin{equation*}
     |\Sigma(q) : P_J | = \frac{|P_{\Pi} : P_\emptyset|}{|P_J : P_\emptyset|}
     = \frac{\KB(\Sigma,\Pi)}{\KB(\Sigma,J)}
  \end{equation*}
  we conclude that Equation~\eqref{eq:globalnoneq} for $\Sigma(q)$ has
  the form
 \begin{equation}
   \label{eq:K2}
     |\Sigma(q)_p| = q^{|\Sigma^+|}     
    \sum_{J \subset \Pi} (-1)^{|J|} \frac{\KB(\Sigma,\Pi)}{\KB(\Sigma,J)}
 \end{equation}
 Lemma~\ref{lemma:KBJ} now implies that $|\Sigma(q)_p| = q^{|\Sigma|}
 = |\Sigma(q)|_p^2$.  (See below for the notation used here.)
Similar arguments show that 
Equation~\eqref{eq:affine} with $G=\Sigma(q)$ specializes to
\begin{equation} \label{eq:K1}
  \sum_{J \subset \Pi}  (-1)^{|J|}
    \frac{\KB(\Sigma,\Pi)}{\KB(\Sigma,J)} q^{|\Sigma^+_J|} = 1
\end{equation}
where we used that $-\rchi(\cat S{L_J}{p+*}) =
(-1)^{|J|}q^{|\Sigma^+_J|}$.
\end{exmp}

Let $(\Sigma,\Pi)$ be a reduced crystallographic root system with
fundamental roots $\Pi$. Put $\KB(\Sigma) = \prod [d]$ where $d$ runs
through the set of degrees for the Weyl group $W(\Sigma)$ \cite[\S9.3,
Proposition~10.2.5]{carter:lie} and $[m] \in \Z[X]$, $m \geq 1$, is
the polynomial
\begin{equation*}
  [m] = \frac{X^m-1}{X-1} = 1+X+\cdots+X^{m-1}
\end{equation*}
Thus $\KB(\Sigma) = |K : B|$ where $B$ is the Borel subgroup of $K =
\Sigma(q)$ and $\KB(A_m) = [2] \cdots [m+1]$, for instance.  More
generally, for any set $J \subset \Pi(\Sigma)$ of fundamental roots,
set $\KB(\Sigma,J) = \prod \KB(\Sigma_i)$ when the orthogonal
decomposition of the root system $\Sigma_J = \Sigma \cap \R J$ spanned
by the fundamental roots in $J$ is $\bigcup \Sigma_i$ for irreducible
root systems $\Sigma_i$ \cite[Definition~1.8.4]{GLSIII}.  Thus
$\KB(\Sigma,\Pi) = \KB(\Sigma)$. We let $\KB(\Sigma,\emptyset)=1$ by
convention.

\begin{lemma}\label{lemma:KBJ}\cite[Theorem~9.4.5]{carter:lie}
  For any  reduced crystallographic root system $\Sigma$
  \begin{equation*}
    \sum_{J \subset \Pi}(-1)^{|J|} 
    \frac{\KB(\Sigma,\Pi)}{\KB(\Sigma,J)} = 
    X^{|\Sigma^+|}
  \end{equation*}
\end{lemma}

In the concrete case where the root system $\Sigma = A_{m-1}$,
Lemma~\ref{lemma:KBJ} and Equation~\eqref{eq:K1} give two prototypical
combinatorial identities involving Gaussian multinomial coefficients
\cite[\S1.7]{stanley97}
\begin{equation*}
  \sum_{(m_1,\ldots,m_k) \in \OP(m)}
  (-1)^k \binom{[m]}{[m_1],\cdots,[m_k]} = X^{\binom m2}, \qquad
   \sum_{(m_1,\ldots,m_k) \in \OP(m)}
  (-1)^k  \binom{[m]}{[m_1],\cdots,[m_k]}X^{\sum \binom{m_i}2} = 1 
\end{equation*}
where the sums run over the set $\OP(m)$ of the $2^{m-1}$ ordered
partitions of $m$. 

Experiments indicate that Equation~\eqref{eq:eulerKp} holds also
  for Steinberg and Suzuki--Ree groups when we replace $r(\Sigma)$
  with the twisted rank $r(\widetilde \Sigma)$ \cite[Proposition
  2.3.2, Remark 2.3.3]{GLSIII}.



  \begin{exmp}\label{exmp:eucSigma}
    The exponential generating function for the integer sequence $n
    \to |(\Sigma_n)_p|$ counting the number of $p$-singular
    permutations of $n$ things is \cite[Example 5.2.10]{stanley99}
    is the Artin--Hasse exponential
    \begin{equation*}
      \sum_{n=1}^\infty  |(\Sigma_n)_p| \frac{x^n}{n!} =
      \sum_{n=1}^\infty  \chi(\cat O{\Sigma_n}p)x^n =
      \exp ( x + \frac{x^p}{p} + \frac{x^{p^2}}{p^2}+ \cdots + 
      \frac{x^{p^m}}{p^m} + \cdots )
    \end{equation*}
    The sequence $|(\Sigma_n)_2|$  begins with
    $1, 2, 4, 16, 56, 256, 1072, 11264, 78976, 672256$ for $1 \leq n
    \leq 10$ (OEIS \href{https://oeis.org/A005388}{A005388}).

    The $p$-radical subgroups of the symmetric group are described in
    \cite[\S2]{AlpFong} and \cite[Lemma~2.2]{michler&olsson91} as
    follows: Let $r=(c_1,c_2,\ldots)$ be a finite sequence of natural
    numbers $c_i \geq 0$. Define the degree at $p$ of the sequence $r$
    to be
  \begin{equation*}
    \deg_p r = p^{\sum r}
  \end{equation*}
  The set $\cat S{\Sigma_n}{p+*+\mathrm{rad}}/\Sigma_n$ of conjugacy
  classes of nonidentity $\Sigma_n$-radical $p$-subgroups is in
  bijective correspondence with the set of nonempty multisets
  $R=\{r_1,r_2,\ldots\}$ of finite sequences $r_i$ such that $\sum_i
  \deg_pr_i \leq n$. There are extra restrictions, somewhat ignored in
  \cite{AlpFong, michler&olsson91}, in case $p=2$ or $p=3$ since
  $O_2(\Sigma_n)$ is nontrivial (only) for $n=2,4$ and $O_3(\Sigma_n)$
  is nontrivial (only) for $n=3$.  In case $p=2$ we must also demand
  that
  \begin{itemize}
  \item $n - \sum_i \deg_pr_i$ does not equal $2$ or $4$
  \item none of the sequences $(1),(1,1),\ldots$ has multiplicity $2$
    or $4$ in the multiset $R$
  \end{itemize}
  and when $p=3$ we must also demand that
   \begin{itemize} 
  \item $n - \sum_i \deg_pr_i$ does not equal $3$
  \item none of the sequences $(1),(1,1),\ldots$ has multiplicity $3$
   in the multiset $R$
  \end{itemize}
  The correspondence of \cite[\S2]{AlpFong} sends the multiset
  $R=\{r_1,r_2,\ldots\}$ to the $p$-radical subgroup $A_R = 1 \times A_{r_1}
  \times A_{r_2} \times \cdots$ where $1$ is the trivial subgroup of
  the symmetric group on $n - \sum_i \deg_pr_i$ elements and the
  $A_{r_i}$s are certain (iterated) wreath products of elementary
  abelian subgroups. We can now fill out the table
%
\begin{center}  
\begin{tabular}[t]{>{$}c<{$} | *{3}{>{$}c<{$}} >{$}r<{$} >{$}c<{$}}
R & A_R  & |A_R|  & N_{\Sigma_n}(A_R)/A_R  & 
k^{[A_R]} 
& |\Sigma_n : N_{\Sigma_n}(A_R)|    \\ \hline \noalign{\smallskip}
\{(1,1,1)\} & A_1 \wr A_1 \wr A_1 & 128 & \GL 1{\F_2} \times \GL
1{\F_2} \times \GL 1{\F_2} &
1 & 315 \\
\{(1,2)\} & A_1 \wr A_2 & 64 & \GL 1{\F_2} \times \GL 2{\F_2} &
-2 & 105 \\
\{(2,1)\} & A_2 \wr A_1 & 32 & \GL 2{\F_2} \times \GL 1{\F_2} &
-2 & 210 \\
\{(2),(1,1)\} & A_2 \times A_1 \wr A_1 & 32 & 
\GL 2{\F_2} \times \GL 1{\F_2} \times \GL 1{\F_2}  &
-2 & 210 \\
\{(2),(2)\} & A_2 \times A_2 & 16 & \GL 2{\F_2} \wr \Sigma_2 &
16 & 35 \\
\{(3)\} & A_3 & 8 & \GL 3{\F_2} & 8 & 30 \\
\{(1)\} & 1_6 \times A_{1}  & 2  & \Sigma_6 \times \GL 1{\F_2} &
16 & 28 \\
\emptyset & 1_8 & 1 & \Sigma_8 & -512 & 1
\end{tabular}
\end{center}
of $2$-radical subgroups of $\Sigma_8$.  With computer assistance it
is possible to determine the table $(\cat S{\Sigma_8}{}(A_R,[A_S]))$
and then read off the \we\ $k^{[K]}$ for $\cat
S{\Sigma_8}{2+\mathrm{rad}}/\Sigma_8$ (Definition~\ref{defn:wecoweSA})
as the solution to the linear equation~\eqref{eq:lineqweS}
\begin{equation*}
    \begin{pmatrix}
      1 &   0  &   0  &  0  &  0  &  0  &  0  &  0 \\
  3  &  1  &  0  &  0  &  0  &  0  &  0  &  0 \\
  3  &  0  &  1  &  0  &  0  &  0  &  0  &  0 \\
  3  &  0  &  0  &  1  &  0  &  0  &  0  &  0 \\
  9  &  0  &  6  &  6  &  1  &  0  &  0  &  0 \\
 21  &  7  &  0  &  7  &  0  &  1  &  0  &  0 \\
 45  & 15  & 15  &  0  &  0  &  0  &  1  &  0 \\
 315  &105  &210  &210  & 35  & 30  & 28  &  1 
    \end{pmatrix}
    \begin{pmatrix}
         1  \\ -2 \\  -2 \\   -2 \\  16 \\   8 \\  16 \\ -512
    \end{pmatrix} =
    \begin{pmatrix}
      1 \\ 1 \\ 1 \\ 1 \\ 1 \\ 1 \\ 1 \\ 1
    \end{pmatrix}
  \end{equation*}
  In particular, $-\rchi(\cat S{\Sigma_8}{2+*}) = -512$ and
  $-\rchi(\cat S {\GL 3{\F_2}}{2+*})=8$ (agreeing with
  Equation~\eqref{eq:eulerKp}) and
\begin{equation*}
   |\Sigma_8| |\chi(\cat O{\Sigma_8}2)| =
    \sum_{[A_R] \geq 1} k^{[A_R]} |\Sigma_8 : N_{\Sigma_8}(A_R)|
    |A_R| = 11264
\end{equation*}
correctly counts the number $|(\Sigma_8)_2|$ of $2$-singular elements
in $\Sigma_8$ (Propositions~\ref{prop:SGpwecowe}, \ref{prop:weOGp}).
 
  It is, however, still unclear how to continue the computer generated
  sequence
    \begin{center}
       \begin{tabular}[t]{>{$}c<{$} | *{13}{>{$}c<{$}} }
       n  & 3 & 4 & 5 & 6 & 7 & 8 & 9 & 10 & 11 & 12 & 13 & 14 & 15
       \\ \hline \noalign{\smallskip}
       \chi(\cat S{\Sigma_n}{2+*})  & 3 & 1 & -15 & -15 & 161
       & 513 & -639 & -7935 & -20735 & 235521 & 3244033 & 2232321
       & -190068735
       \end{tabular}
    \end{center}
    of \Euc s of the Brown posets at $p=2$ for the symmetric groups.
    It is possible that the methods of \cite{pfeiffer97} can be used
    to extend the sequence a little longer.
    The generating functions $\sum_{n=1}^\infty \chi(\cat
    S{\Sigma_n}{p+*})x^n$ are unknown.
\end{exmp}

\subsection{\Euc s of centralized subcategories of $\cat OGp$}
\label{sec:chiCOGA}

Let $A$ be a group acting on $G$ from the right.  The centralized
subcategory $C_{\cat OG{}}(A)$ of $\cat OG{}$ has objects, \m\ sets,
and auto\m\ groups
\begin{align*}
  &\Ob{C_{\cat OG{}}(A)} = C_{\Ob{\cat OG{}}}(A) =
\{ K \leq G \mid [K,A] \leq K \} \\
 &C_{\cat OG{}}(A)(H,K) = C_{\cat OG{}(H,K)}(A) 
  = \{g \in N_G(H,K) \mid [g,A] \leq K \}/K \\
  &C_{\cat OG{}(A)}(K) =
 C_{\cat OG{}(K)}(A) =
\{ g \in N_G(K) \mid [g,A] \leq K \}/K
\end{align*}
The category $\cat O{(G,A)}{}$ is defined to be the subcategory of the
centralized subcategory with the same objects but with
\m\ sets, and auto\m\ groups
\begin{equation*}
 \cat O{(G,A)}{}(H,K) = C_{N_G(H,K)}(A)/C_K(A), \qquad
 \cat O{(G,A)}{}(H) = C_{N_G(H)}(A)/C_H(A)
\end{equation*}
and with composition $C_{N_G(H,K)}(A)/C_K(A) \times
C_{N_G(K,L)}(A)/C_L(A) \to C_{N_G(H,L)}(A)/C_L(A)$ induced by
composition in $G$. This is well-defined for if $k \in C_K(A)$, $y \in
C_{N_G(K,L)}(A)$ then $k^{ya} = k^{ay} = k^y$ so $k^y \in C_L(A)$.

Let $[K]$ be the set of objects in $C_{\cat OG{}}(A)$ isomorphic to
the object $K$. There is a bijection 
\begin{center}
  \begin{tikzpicture}
   \matrix (dia) [matrix of math nodes, column sep=45pt, row
   sep=30pt]{
   \{ g \in N_G(K) \mid [g,A] \leq K \} \backslash
   \{ g \in G \mid [g,A] \leq K \} & \left[K\right] \\};
 \draw [->] (dia-1-1) -- (dia-1-2) node[pos=.5, above] {$K \to K^g$};
  \end{tikzpicture}
\end{center}
between $[K]$ and the orbit set for the free left action of the group
$\{ g \in N_G(K) \mid [g,A] \leq K \}$ on the set $\{ g \in G \mid
[g,A] \leq K \}$. Indeed, if $g \in G$ and conjugation by $g$ is an
iso\m\ $K \to K^g$ in $C_{\cat OG{}}(A)$, then $K^g$ is normalized by
$A$ and $[g,A] \leq K$ as we have the commutative square
\begin{center}
  \begin{tikzpicture}
   \matrix (dia) [matrix of math nodes, column sep=25pt, row
   sep=20pt]{
     K & K^g \\
     K & K^g \\};
   \draw [->] (dia-1-1) -- (dia-1-2) node[pos=.5, above] {$g$};
   \draw [->] (dia-2-1) -- (dia-2-2) node[pos=.5, below] {$g$};
   \draw [->] (dia-1-1) -- (dia-2-1) node[pos=.5, left] {$a$};
   \draw [->] (dia-1-2) -- (dia-2-2) node[pos=.5, right] {$a$};
  \end{tikzpicture}
\end{center}
in $\cat OG{}$. Conversely, if $g \in G$ and $[g,A] \leq K$, then $K^g$
is normalized by any $a \in A$ as
$K^{ga} = K^{ag} = K^g$ and the above diagram commutes.

The free right action of $N_G(K)/K = \cat OG{}(K)$ on $G/K$, $gK \cdot
xK = gxK$ for $g \in G$, $x \in N_G(K)$, restricts to a free right
action of $C_{N_G(K)/K}(A) = C_{\cat OG{}(K)}(A)$ on $C_{G/K}(A)$. The
number of orbits for this action equals the number $|[K]|$ of objects
of $C_{\cat OGp}(A)$ isomorphic to $K$.
To see this we note that
\begin{equation*}
  |C_{G/K}(A)| = \frac{|\{g \in G \mid [g,A] \leq K\}|}{|K|}, \qquad
  |C_{N_G(K)/K}(A)| = \frac{|\{g \in N_G(K) \mid [g,A] \leq K\}|}{|K|}
\end{equation*}
and thus
\begin{equation}\label{eq:|[K]|}
  |[K]| = 
  \frac{|\{g \in G \mid [g,A] \leq K\}|}{|\{g \in N_G(K) \mid [g,A] \leq K\}|} 
  =\frac{|C_{G/K}(A)|}{|C_{N_G(K)/K}(A)|}
  =\frac{|C_{G/K}(A)|}{|C_{\cat OG{}(K)}(A)|}
\end{equation}
For later use we record that
\begin{equation}
  \label{eq:cowtnum}
  |[K]| | C_{\cat OG{}}(A)(K) | =
  |[K]| | C_{\cat OG{}(K)}(A) | =
  |C_{G/K}(A)|
\end{equation}
for any object $K$ of $C_{\cat OG{}}(A)$. 
Under the additional
 assumptions that $A$ {\em centralizes\/} $K$ and
 the order of $A$ is prime to the order of $K$, $[K,A] = 1$ and
 $(|K|,|A|)=1$, we have that $[g,a] \leq K \iff [g,a]=1$.
To see this, note that if $g \in G$ and $[g,a] \leq
 K$ then, in fact, $[g,a]=1$ is the neutral element of $G$ as the
 order of $[g,a]$ divides $|K|$ since it is an element of $K$ and also
 divides the order of $A$ as $[g,a^j]=[g,a]^j$ for any $j \geq 0$
 \cite[5.1.5.(ii)]{robinson:groups}. We conclude that
 \begin{equation}
   \label{eq:cowtnumspecial}
    |[K]| | C_{\cat OG{}}(A)(K) | 
    = |C_{G/K}(A)| = \frac{|\{g \in G \mid [g,A] \leq K\}|}{|K|}
    = \frac{|C_G(A)|}{|K|}
    = |C_G(A) : K |
 \end{equation}
 when $[K,A]=1$ and $(|K|,|A|)=1$.

 By using nonabelian cohomology groups we can get a similar result in
 a slightly different situation.



 \begin{cor}\label{cor:CGHA2}
   $\cat O{(G,A)}p = C_{\cat OGp}(A)$ when $A$ is $p$-regular.
 \end{cor}
 \begin{proof}
   The map $C_{N_G(H,K)}(A) \to C_{N_G(H,K)/K}(A)$ is surjective: Let
   $x \in N_G(H,K)$ and suppose that the image of $x$ in $N_G(H,K)/K$
   is centralized by $A$. Define a function $A \to K \colon a \to k_a$
   by $x^a =x k_a$. As $k_{a_1a_2} =k_{a_2}k_{a_1}^{a_2}$ this a
   nonabelian $1$-cocycle \cite[\S5]{serre:galois}. But
   $H^1(A;K)=\{\ast\}$ \cite[Exercises p $59$]{serre:galois} so there
   is $\ell \in K$ such that $k_a = \ell \ell^{-a}$. Then $x \ell$ is
   centralized by $A$. This implies that $\cat O{(G,A)}{}(H,K) =
   C_{N_G(H,K)}(A)/C_K(A) = C_{N_G(H,K)/K}(A) = C_{\cat OG{}}(A)(H,K)$
   for any two $A$-normalized $p$-subgroups $H,K \leq G$.
 \end{proof}


 As a special case, suppose that $V$ is a finite dimensional
 $\F_p$-vector space $V$ and $A \leq \GL {}V{}$.  The objects of the
 category $\cat O{(V,A)}{}$ are the $\F_p[A]$-submodules of the
 $\F_p[\GL {}V{}]$-module $V$. The set of \m s between two
 $\F_p[A]$-submodules, $U_1$ and $U_2$, is
  \begin{equation}\label{eq:COVA}
    C_{\cat OV{}}(A)(U_1,U_2) =
  \begin{cases}
    C_{V}(A)/C_{U_2}(A) & U_1 \leq U_2 \\ 
    \emptyset & U_1 \nleq U_2
  \end{cases}
  \end{equation}
  $V$ is the terminal object in $\cat O{(V,A)}{}$.

  \begin{prop}\label{prop:cowtCOVA}
    Let $A$ be a finite group acting on the finite dimensional
    $\F_p$-vector space $V$.  The function
    \begin{equation*}
      k_U =
      \begin{cases}
        |C_V(A)|^{-1}  &  \dim U=0 \\
        (p-1)|C_V(A)|^{-1}  &    \dim U= 1 =\dim C_U(A)  \\
        0 & \text{otherwise}
      \end{cases}
    \end{equation*}
    is a co\we\ for $\cat O{(V,A)}{}$.
  \end{prop}
  \begin{proof}
   Let $U_2$ be an $\F_p[A]$-submodule of $V$. 
   Then
   \begin{equation*}
     \sum_{
     \begin{subarray}{c}
       U_1 \leq U_2 \\ [A,U_1] \leq U_1
     \end{subarray}} k_{U_1} |\cat O{(V,A)}{}(U_1,U_2)| =
     \sum_{U_1 \leq C_{U_2}(A)} k_{U_1} |\cat O{(V,A)}{}(U_1,U_2)| =
     \sum_{U_1 \leq C_{U_2}(A)} k_{U_1} \frac{|C_V(A)|}{|C_{U_2}(A)|}
     =
     \sum_{U_1 \leq C_{U_2}(A)} k_{U_1}^{C_{U_2}(A)} = 1
   \end{equation*}
   since
   \begin{itemize}
   \item $k_{U_1} \neq 0 \implies U_1=C_{U_1}(A)$
   \item if $U_1 \leq C_{U_2}(A) \leq U_2$ then $C_{\cat
       OV{}}(A)(U_1,U_2) = C_V(A)/C_{U_2}(A) = \cat
     O{C_V(A)}{}(U_1,C_{U_2}(A))$ \eqref{eq:COVA}
   \item $k_{U_1} \frac{|C_V(A)|}{|C_{U_2}(A)|}$ is the co\we\
     $k_{U_1}^{C_{U_2}(A)}$ on $\cat O{C_{U_2}(A)}{}$ from
     Proposition~\ref{prop:cowtOV}
   \end{itemize}
   This shows that the function $k_\bullet$ is a co\we\ on $\cat
   O{(V,A)}{}$.
  \end{proof}

  Let $\cat O{(V,A)}{[1,V)}$ be the full subcategory of $\cat
O{(V,A)}{}$ generated by all its object but the final object $V$.

  \begin{lemma}\label{lemma:COVA}
    Let $A$ be a finite group acting on the finite dimensional
    $\F_p$-vector space $V$. Then
    \begin{equation*}
      \chi(\cat O{(V,A)}{[1,V)})=
      \begin{cases}
        0 & \dim V=0 \\
        p^{-1}  & \dim V =1 = \dim C_V(A) \\
        1 & \text{otherwise}
      \end{cases}
    \end{equation*}
   \end{lemma}
 \begin{proof}
   If $\dim V=0$, $\cat O{(V,A)}{[1,V)}$ is the $0$-object category
 which has \Euc s $\chi=0$. If $\dim V=1$, $\cat O{(V,A)}{[1,V)}$ is
 the $1$-object category of the group $C_V(A)$ so its \Euc\ is
 $p^{-1}$ if $A$ acts trivially and $1$ if not. If $\dim V>1$,
 $\chi(\cat O{(V,A)}{[1,V)})=\chi(\cat O{(V,A)}{})$ because the
 co\we\ for $\cat O{(V,A)}{}$ (Proposition~\ref{prop:cowtCOVA})
 restricts to a co\we\ for the right ideal $\cat O{(V,A)}{[1,V)}$
 \cite[Remark 2.6]{jmm_mwj:2010} and it has value $0$ at the deleted
 object $V$.  But clearly, $\chi(\cat O{(V,A)}{})=1$ as this category
 has a final object.
\end{proof}

\begin{exmp}
  When $A$ is $p$-regular, $\cat O{(V,A)}{} = C_{\cat OV{}}(A)$ by
  Corollary~\ref{cor:CGHA2}. However, when
$p=2$, $V=\F_2^2$ has dimension $2$, and $A=\Sigma_2 \leq
\mathrm{GL}(V)$ has order $2$, then $\chi(C_{\cat
  OV{[1,V)}}(A))=\frac{1}{2}$ and $\chi(\cat O{(V,A)}{})=1$. Thus
these two categories are not equivalent.

\end{exmp}


\begin{prop}\label{prop:coweCOGpA}  
Consider the centralized subcategory $C_{\cat OGp}(A)$ and its
subcategory $\cat O{(G,A)}p$.
\begin{enumerate}
\item \label{prop:coweCOGpA1} The function
  \begin{equation*}
    k_K = \frac{-\rchi(C^{[1,V)}_{\cat OV{}}(A))}{|C_{G/K}(A)|}, \qquad
   V=K/\Phi(K)
  \end{equation*}
  is a co\we\ for the centralized orbit category $C_{\cat OGp}(A)$.
\item \label{prop:coweCOGpA2}  The function 
  \begin{equation*}
    k_K =
    \begin{cases}
      |C_G(A)|^{-1} & K=1 \\
      |C_G(A)|^{-1}(1-p^{-1})|K| & \text{$K>1$ cyclic} \\
      0 & \text{otherwise}
    \end{cases}
  \end{equation*}
  is a co\we\ for $\cat O{(G,A)}p$
  and the
  \Euc\ 
  \begin{equation*}
    \chi(\cat O{(G,A)}p) = \frac{|C_G(A)_p|}{|C_G(A)|}
  \end{equation*}
  is the density of $p$-singular elements in $C_G(A)$.
\end{enumerate}
\end{prop}
\begin{proof}
  Let $K$ be an $A$-normalized $p$-subgroup of $G$. 
  Any \m\  $xK \in C_{N_G(H,K)/K}(A) = C_{\cat OGp}(A)(H,K)$ is
  isomorphic in $C_{\cat OG{p}}(A)//K$ to $K \in  C_{N_G(H^x,K)/K}(A)
  = C_{\cat OGp}(A)(H^x,K)$. This observation shows that
  that there
  are category equivalences
  \begin{equation*}
    C_{\cat OK{}}(A) \to  C_{\cat OG{p}}(A)/K, \qquad
    C_{\cat OK{[1,K)}}(A) \to C_{\cat OG{p}}(A)//K
  \end{equation*}
  and, using \eqref{eq:cowtnum}, it follows that
  \begin{equation*}
    k_K = 
    \frac{-\rchi(C_{\catp [{}]OK{[1,K)}}(A))}
    {|[K]| |C_{\cat OG{}}(A)(K)|} =
    \frac{-\rchi(C_{\catp [{}]O{V}{[1,V)}}(A))}
    {|[K]| |C_{\cat OG{}}(A)(K)|} =
    \frac{-\rchi(C_{\catp [{}]O{V}{[1,V)}}(A))}
    {|C_{G/K}(A)|}, \qquad V=V(K), 
  \end{equation*}
  is a co\we\ for $C_{\cat OGp}(A)$ according to (equivariant versions
  of) \cite[Theorem~3.7, Lemma~5.1.(c)]{gm:2012}. Here, $V = V(K) =
  K/\Phi(K)$ is the Frattini quotient with its inherited $A$-action.

  Similarly, there are equivalences of categories
  \begin{equation*}
    \cat O{(K,A)}{} \to \cat O{(G,A)}{}/K, \qquad
    \cat O{(K,A)}{[1,K)} \to \cat O{(G,A)}{}//K
  \end{equation*}
  and we get the co\we\ for $\cat O{(G,A)}p$
  \begin{equation*}
    k_K = \frac{-\rchi(\cat O{(K,A)}{[1,K)})}
    {|[K]| |\cat O{(G,A)}{}(K)|} =
    \frac{-\rchi(\cat O{(V,A)}{[1,V)})}{|C_G(A) : C_K(A)|}
  \end{equation*}
  as there are $|[K]| = |C_G(A) : C_{N_G(K)}(A)|$ objects of $\cat
  O{(G,A)}p$ isomorphic to $K$.  Since $K$ is cyclic if and and
  only if its Frattini quotient $V = V(K)$ is cyclic
  \cite[Corollary~5.1.2]{gorenstein68} we see from
  Lemma~\ref{lemma:COVA} that $k_K$ is concentrated on the cyclic
  $p$-subgroups normalized by $A$, in fact on the cyclic $p$-subgroups
  centralized by $A$. Here we use that if $A$ centralizes $V(K)$ then
  $A$ centralizes $K$ \cite[Theorem~5.3.5]{gorenstein68}. The co\we\
  takes value $k_1= |C_G(A)|^{-1}$ at the trivial subgroup and value
  \begin{equation*}
    k_C = \frac{1-p^{-1}}{|C_G(A):C|} = 
    \frac{1}{|C_G(A)|}(1-p^{-1})|C|
  \end{equation*}
  at a nontrivial cyclic subgroup $C$ centralized by $A$ by
  Equation~\eqref{eq:cowtnumspecial}.  The counting formula of
  Lemma~\ref{lemma:Gp} now finishes the proof.
\end{proof}

\begin{exmp}\label{exmp:coweAnotpreg}
  Let $p=2$, $G=\Sigma_4$, and $A=\gen{(1,2)(3,4)} = C_2$. Then $A$ is
  {\em not\/} $p$-regular, the co\we\ for $C_{\cat OGp}(A)$ is {\em
    not\/} concentrated on the cyclic subgroups, and $\chi(C_{\cat
    OGp}(A))=2/3$ does {\em not\/} equal the density, $1$, of the
  $p$-singular elements in $C_G(A)=D_8$. Thus $C_{\cat OGp}(A)$ and
  $\cat O{(G,A)}p$ are not equivalent.
\end{exmp}

It is a curious fact, in contrast to the situation for Brown posets
(Corollary~\ref{cor:S+-}.\eqref{cor:S+-3}), that the centralizer
subcategory of the orbit category, $C_{\cat OGp}(A)$, and the orbit
category of centralizer subgroup, $\cat O{C_G(A)}p$, have the same
\Euc s when $A$ is $p$-regular (Propositions \ref{prop:coweOGp} and
\ref{prop:coweCOGpA}).

\begin{prop}\label{prop:weCOGpA}
  Consider the centralized subcategory $C_{\cat OGp}(A)$ and its
  subcategory $\cat O{(G,A)}p$.
\begin{enumerate}
\item  \label{prop:weCOGpA1}  The function
  \begin{equation*}
    k^H =
    \frac{-\rchi(C_{\catp S{{\cat OG{}(H)}}*}(A))}
    {|C_{G/H}(A)|}
  \end{equation*}
  is a \we\ for $C_{\cat OG{p}}(A)$ and the \Euc\ is
  \begin{equation*}
    \chi(C_{\cat OG{p}}(A)) =
    \sum_{H} 
     \frac{-\rchi(C_{\catp
        S{{\cat OG{}(H)}}*}(A))}{|C_{G/H}(A)|} =
    \sum_{[H]} 
     \frac{-\rchi(C_{\catp
        S{{\cat OG{}(H)}}*}(A))}{|C_{{\cat OG{}(H)}}(A)|}
  \end{equation*}
  where the first sum runs over the set of objects $H$ of $C_{\cat
    OG{p+\mathrm{rad}}}(A)$ and the second one over the set isomorphism
    classes of such objects.
  \item \label{prop:weCOGpA2} The function
  \begin{equation*}
    k^H =
    \frac{-\rchi(C_{\catp S{{\cat OG{}(H)}}*}(A))}
    {|C_G(A):C_H(A)|}
  \end{equation*}
  is a \we\ for $\cat O{(G,A)}p$ and the \Euc\ is
  \begin{equation*}
    \chi(\cat O{(G,A)}p) =
    \sum_{H} 
     \frac{-\rchi(C_{\catp
        S{{\cat OG{}(H)}}*}(A))}{|C_G(A) : C_H(A)|} =
    \sum_{[H]} 
     \frac{-\rchi(C_{\catp
        S{{\cat OG{}(H)}}*}(A))}{|\cat O{(G,A)}{}(H)|}
  \end{equation*}
  where the first sum runs over the set of objects $H$ of $C_{\cat
    OG{p+\mathrm{rad}}}(A)$ and the second one over the set isomorphism
    classes of such objects. 
\end{enumerate}
\end{prop}
\begin{proof}
  As in the nonequivariant case \cite[\S8]{gm:2012} there are homotopy
  equivalence of categories
  \begin{equation*}
    H/C_{\cat OGp}(A) \to C_{\cat S{\cat OG{}(H)}p}(A), \qquad
    H//C_{\cat OGp}(A) \to C_{\cat S{\cat OG{}(H)}{p+*}}(A)
  \end{equation*}
  which takes $gK \in C_{N_G(H,K)/K}(A)$ to the nonidentity
  \cite[5.2.4]{robinson:groups} $A$-normalized subgroup $N_{^gK}(H)/H$
  of $N_G(H)/H$. This subgroup is indeed $A$-normalized because if
  $[g,A] \subseteq K$ then $g^{-1}a=kag^{-1}$ for some $k \in K$. Then
  $({}^gK)^a = K^{g^{-1}a} = K^{kag^{-1}} = K^{ag^{-1}} = K^{g^{-1}} =
  {}^gK$ for any $a \in A$. It now follows from \cite[Theorem
  C]{gm:2012} and Equation~\ref{eq:cowtnum} that the function
  \begin{equation*}
    k^H =  
    \frac{-\rchi(C_{\catp S{{\cat OG{}(H)}}*}(A))}
    {|[H]||C_{{\cat OG{}(H)}}(A)|} =
    \frac{-\rchi(C_{\catp S{{\cat OG{}(H)}}*}(A))}
    {|C_{G/H}(A)|}
  \end{equation*}
  is a \we\ for $C_{\cat OGp}(A)$.  Lemma~\ref{lemma:OpHA} shows that
  the $k^H=0$ unless $H$ is $G$-radical $p$-subgroup of $G$.

  Similarly, there are homotopy equivalence of categories
  \begin{equation*}
    H/\cat O{(G,A)}p \to C_{\cat S{\cat OG{}(H)}p}(A), \qquad
    H//\cat O{(G,A)}p \to C_{\cat S{\cat OG{}(H)}{p+*}}(A)
  \end{equation*}
  and the function
  \begin{equation*}
    k^H =  
    \frac{-\rchi(C_{\catp S{{\cat OG{}(H)}}*}(A))}
    {|[H]||\cat O{(G,A)}{}(H)|} =
    \frac{-\rchi(C_{\catp S{{\cat OG{}(H)}}*}(A))}
    {|C_G(A) : C_H(A)|}
  \end{equation*}
  is a co\we\ for $\cat O{(G,A)}p$.
\end{proof}


We can now the generalize the global formula below
Proposition~\ref{prop:weOGp} to an equivariant situation.


\begin{proof}[Proof of Theorem~\ref{thm:global}.\eqref{thm:global1}]
  The equation
  \begin{equation*}
    |C_G(A)_p| +
    \sum_{H} 
     \rchi(C_{\catp S{{\cat OG{}(H)}}*}(A))|C_H(A)| = 0
  \end{equation*}
  simply states that the \Euc\ of $\cat O{(G,A)}p$, the density of the
  $p$-singular elements in $C_G(A)$
  (Proposition~\ref{prop:coweCOGpA}.\eqref{prop:coweCOGpA2}), is the
  sum of the values of the \we\ from
  Proposition~\ref{prop:weCOGpA}.\eqref{prop:weCOGpA2}.
\end{proof}

The equation of Theorem~\ref{thm:global}.\eqref{thm:global1} can be
reformulated as
\begin{equation*}
  \sum_{K \in C_{\cat SG{p+\mathrm{rad}}}(A)/N_G(A)}
  k^{[K]} C_{\cat SG{p+\mathrm{rad}}}(A)(1,[K]) |C_K(A)| = |C_G(A)_p|
\end{equation*}
where by Proposition~\ref{prop:weCSGpradA}
\begin{equation*}
  \left(
    C_{\cat SG{p+\mathrm{rad}}}(A)(H,[K]) 
    \right)_{H,K \in  C_{\cat SG{p+\mathrm{rad}}}(A)/N_G(A)}
    \begin{pmatrix}
      \vdots \\ k^{[K]} \\ \vdots
    \end{pmatrix} =
    \begin{pmatrix}
      \vdots \\ 1 \\ \vdots
    \end{pmatrix}
\end{equation*}
and $C_{\cat SG{p+\mathrm{rad}}}(A)(1,[K])$ is the length of
the $N_G(A)$-orbit through $K \in C_{\cat SG{p+\mathrm{rad}}}(A)$.
The \we\ values are $k^{[K]} = -\rchi(C_{\cat S{N_G(K)/K}{p+*}}(A))$
by Lemma~\ref{lemma:wtA}; in particular, $k^{[1]} = -\rchi(C_{\cat
  S{G}{p+*}}(A))$.

\begin{exmp}   
  Let $p=2$ and $G=\Sigma_5$ the symmetric group on $5$ things
  considered as an $A=\gen{(1,2,3)}$-group with conjugation action.
  The centralizer $C_G(A) = A \times \gen{(4,5)}$ is cyclic of order
  $6$ and contains $2=|C_G(A)_2|$ $2$-singular elements. The poset of
  $A$-normalized $2$-subgroups
  \begin{equation*}
  C_{\cat SG{2+\mathrm{rad}}}(A) = \{H_4,H_3,H_2,H_1 \} =
  \{ \gen{(1,4)(2,3),(1,3)(2,4)}, \gen{(1,2)(3,5),(1,3)(2,5)}, \gen{(4,5)}, 
  \gen{1}\}
  \end{equation*}
  consists of four subgroups: $H_3$ and $H_4$ are four-groups, $H_2$
  is $O_2C_G(A)$, and $H_1$ is the trivial subgroup. The orbits for
  the $N_G(A)$-action on $ C_{\cat SG{2+\mathrm{rad}}}(A)$ are $
  \{H_3,H_4\}, \{H_2\}, \{H_1\} $. The \we\ $k^{[H]} = -\rchi(C_{\cat
    S{N_G(H)/H}{2+*}}(A))$  for $ C_{\cat
    SG{2+\mathrm{rad}}}(A)/N_G(A)$ is the solution
   to the linear equation
  \begin{equation*}
    \begin{pmatrix}
      1 & 0 & 0 \\ 0 & 1 & 0 \\ 2 & 1 & 1
    \end{pmatrix}
    \begin{pmatrix}
      1 \\ 1 \\ -2
    \end{pmatrix} =
    \begin{pmatrix}
      1 \\ 1 \\ 1
    \end{pmatrix}
  \end{equation*}
   and the vector $(|C_{H_i}(A)|)_{i=3,2,1}$ is $(1,2,1)$.
  The sum on the left hand side of the equation of
  Theorem~\ref{thm:global}.\eqref{thm:global1} is
  \begin{equation*}
   k^{[H_3]} \cdot |[H_3]| \cdot |C_{H_3}(A)| + 
   k^{[H_2]} \cdot |[H_2]| \cdot |C_{H_2}(A)| +   
   k^{[H_1]} \cdot |[H_1]| \cdot |C_{H_1}(A)| =
    1 \cdot 2 \cdot 1 +
    1 \cdot 1 \cdot 2 +
    (-2) \cdot 1 \cdot 1 = 2  
  \end{equation*}
  which indeed equals the number $|C_G(A)_2|$ of $2$-singular elements
  of $C_G(A)$.  Alternatively, Equation~\eqref{eq:global1} applied
  directly to $C_G(A) = C_6$ has the form
  \begin{equation*}
    \sum_{H \in \cat S{C_G(A)}{2+\mathrm{rad}}}
    -\chi(\cat S{C_G(A(/H}{2+*})|H| =  -\chi(\cat S{C_3}{2+*})|C_2| = 1
    \cdot 2 = |C_G(A)_2| 
  \end{equation*}
  as $\cat S{C_6}{2+\mathrm{rad}} = \{O_2C_6 \}$ has but one element.
  This shows that the these two formulas are quite different from one
  another.
\end{exmp}

\begin{exmp}
  Let $K = \SL m{\F_q}$, $m \geq 2$, and $\bar K = \SL m{\bar \F_q}$
  where $q$ is a power of $p$. The involutory graph auto\m\ $\gamma$
  of $K$ or $\bar K$ maps $g$ to $A^{-1}(g^t)^{-1}A$ where $A$ is the
  $(m \times m)$-matrix with $(+1,-1,\ldots, \pm 1)$ in the diagonal
  from upper right to lower left corner \cite[\S2.7]{GLSIII}.  Let
  $\Pi$ denote the set of fundamental roots of the root system
  $\Sigma$ for the algebraic group $\SL m{\bar\F_p}$. Let $\tau$ be
  the linear auto\m\ of $\R^m$ given by $a_i^\tau = -a_{m+1-i}$ where
  $a_i$ is the standard basis vector. Then $\alpha_i^\tau =
  \alpha_{m-i}$ and $x_{\alpha}(t)^\gamma = x_{\alpha^\tau}(t)$ where
  $\alpha \in \Pi$ is any fundamental root and $x_\alpha$ the root
  group parameterization \cite[Definition 1.9.4]{GLSIII}.

  The centralizer $C_{\SL m{\F_q}}(\gamma)$ of $\gamma$ in $A_{m-1}(q) = \SL
  m{\F_q}$ is $B_{(m-1)/2}(q) = \SOodd m{\F_q} \leq \SL m{\F_q}$ if
  $m$ is odd and $C_{m/2}(q) = \Sp m{\F_q} \leq \SL m{\F_q}$ if $m$ is
  even \cite[\S2.7]{GLSIII}. It seems to be well-known, although I
  have not been able to find an explicit reference in the literature,
  that in fact there is a an iso\m\ of posets
  \begin{equation*}
    C_{\cat S{\SL m{\F_q}}{p+\mathrm{rad}}}(\gamma) \to
    \cat S{C_{\SL m{\F_q}}(\gamma)}{p+\mathrm{rad}}
  \end{equation*}
  induced by the map $P_J \to C_{P_J}(\gamma)$ defined on the set of
  standard $\gamma$-invariant parabolic subgroups $P_J$ where $J
  \subset \Pi$ is any $\tau$-invariant set of fundamental
  roots. Consequently,
  \begin{equation*}
    \chi(C_{\cat S{\SL m{\F_q}}{p+*+\mathrm{rad}}}(\gamma)) =
    \begin{cases}
      \chi(\cat S{C_{\SOodd m{\F_q}}}{p+*+\mathrm{rad}}) & 
      \text{$m$ odd} \\
      \chi(\cat S{C_{\Sp m{\F_q}}}{p+*+\mathrm{rad}}) &
      \text{$m$ even}
    \end{cases}
  \end{equation*}
  and the equivariant relations of
  Theorem~\ref{thm:global}.\eqref{thm:global1}--\eqref{thm:global0}
  for the centralized subposet $ C_{\cat S{\SL
      m{\F_q}}{p+\mathrm{rad}}}(\gamma)$ are simply the corresponding
  nonequivariant relations for $\cat S{\SOodd m{\F_q}}{p+*}$ and 
  $\cat S{\Sp m{\F_q}}{p+*}$ studied in Example~\ref{exmp:chevalley}.
\end{exmp}


This section ends with two proofs of
Theorem~\ref{thm:global}.\eqref{thm:global2}. The first proof is
similar to that of Theorem~\ref{thm:frobeniusbrown}, and the second
one is simply a rewriting of Quillen's proof of Brown's theorem
\cite[Corollary~4.2]{quillen78}.

\begin{proof}[Proof of Theorem~\ref{thm:global}.\eqref{thm:global2} using
 \ref{thm:global}.\eqref{thm:global1}] 
 According to  Theorem~\ref{thm:global}.\eqref{thm:global1},
 \begin{equation*}
    |C_G(A)_p| + \rchi(C_{\cat SG{p+*}}(A)) +
    \sum_{H \neq 1} 
     \rchi(C_{\catp S{{\cat OG{}(H)}}*}(A))|C_H(A)| = 0
 \end{equation*}
 with summation over all nonidentity $A$-normalized $p$-subgroups of
 $G$, or, equivalently,
\begin{equation*}
  |C_G(A)_p| + \rchi(C_{\cat SG{p+*}}(A)) +
    \sum_{[H] \neq [1]} 
     \frac{\rchi(C_{\catp S{{\cat OG{}(H)}}*}(A))}
     {|\cat O{(G,A)}{}(H)|}|C_G(A)| = 0
\end{equation*}
with summation over all iso\m\ classes of objects of $\cat
O{(G,A)}{p+*}$. Write each term of the sum as a product
\begin{equation*}
   \frac{\rchi(C_{\catp S{{\cat OG{}(H)}}*}(A))}
     {|\cat O{(G,A)}{}(H)|}|C_G(A)| =
     \frac{\rchi(C_{\catp S{{\cat OG{}(H)}}*}(A))}
     {|\cat O{(G,A)}{}(H)|_p}
     \frac{|C_G(A)|}{|\cat O{(G,A)}{}(H)|_{p'}}
\end{equation*}
of two factors. Both these factors are integers. The first factor is
an integer because, by induction, we may assume that $|C_{\cat
  OG{}(H)}(A)|_p = |C_{\cat OG{}}(A)(H)|_p$ divides the numerator.
Then also $|\cat O{(G,A)}{}(H)|_p$ divides the numerator as the
auto\m\ group of $H$ in $\cat O{(G,A)}{}$ is a subgroup of the auto\m\
group of $H$ in $C_{\cat OG{}}(A)$. The second factor is an integer
because the auto\m\ group of $H$ in $\cat O{(G,A)}{}$ is a
subquotient of $C_G(A)$. This second factor is clearly also divisible
by $|C_G(A)|_p$.  

We now know that $|C_G(A)|_p$ divides the sum. As it also divides the
number $|C_G(A)_p|$ of $p$-singular elements in $C_G(A)$ by Frobenius'
Theorem~\ref{thm:frobenius}, we conclude that it divides the reduced
\Euc\ of the $A$-centralized Brown poset $C_{\cat SG{p+*}}(A)$.
\end{proof}

\begin{proof}[Proof of Theorem~\ref{thm:global}.\eqref{thm:global2}
  following Quillen \cite{quillen78}]
  Let $B$ be any nonidentity $p$-subgroup of $C_G(A)$. Consider the
  $B$-poset $Y = \mathrm{sd}\; C_{\cat SG{p+*}}(A)$ that is the
  subdivision of the $A$-centralized Brown poset. Then $\rchi(Y)
  =\rchi(C_{\cat SG{p+*}}(A))$.  Define
  \begin{equation*}
    Y' = \bigcup_{X \in \cat SB*} C_Y(X)
  \end{equation*}
  be the subposet of $Y$ consisting of all elements  $y \in Y$ with
  a nontrivial isotropy subgroup $B_y$. Define $Z$ to be the subposet
  of $\cat SB* \times Y'$ given by
  \begin{equation*}
    Z = \{(X,y) \in \cat SB* \times Y' \mid C_Y(X) \ni y \} =
    \{ (X,y) \in \cat SB* \times Y' \mid X \leq B_y \}
  \end{equation*}
  of pairs $(X,y)$, $1 \lneq X \leq B$, $y \in Y'$, such that $X$
  fixes $y$, or, equivalently, $X$ is contained in the isotropy
  subgroup at $y$. When $1 \lneq X \leq B$,
  \begin{equation*}
    \{ y \in Y' \mid C_Y(X) \ni y \} = C_Y(X) = 
    \mathrm{sd}\; C_{C_{\cat SG{p+*}}(A)}(X)
  \end{equation*}
  is conically contractible: $H \leq XH \geq X$. When $y \in Y'$,
  \begin{equation*}
    \{ X \in \cat SB* \mid (X,y) \in Z \} =
    \{ X \in \cat SB* \mid X \leq B_y \} = \cat S{B_y}*
  \end{equation*}
  is contractible because $B_y$ is the largest element. Now
  \cite[Corollary 1.8]{quillen78} (a variant of Quillen's Theorem A)
  shows that $Y'$, $Z$, and $\cat SB*$ are homotopy equivalent; in
  fact, they are contractible as $\cat SB*$, with largest element $B$,
  is so. It follows that the reduced \Euc\ $\rchi(Y) = \chi(Y,*) =
  \chi(Y,Y')$. This relative \Euc\ is divisible by $|B|$ because $B$
  acts freely on $Y - Y'$.
\end{proof}

\begin{cor}\label{cor:equibrown}
  Let $A$ be a $p$-regular group acting on $G$.  Then
  \begin{equation*}
    \chi(\{K \in C_{\cat SG{p+*}}(A) \mid C_K(A) \neq K \}) \equiv
    \chi(\{ K \in C_{\cat SG{p+*}}(A) \mid 1 \lneqq C_K(A) \lneqq K \})
    \bmod |C_G(A)|_p
  \end{equation*}
\end{cor}
\begin{proof}
  By Corollary~\ref{cor:S+-}.\eqref{cor:S+-3}, the difference between
  these two integers is the difference between the reduced \Euc s of the
  centralizer subposet and the poset of the centralizer subgroup, both
  of which are divisible by $|C_G(A)|_p$ (Theorems \ref{thm:brown} and
  \ref{thm:global}.\eqref{thm:global2}).
\end{proof}

\section{Conjectures of Alperin and Kn\"orr--Robinson}
\label{sec:KRC}
This section contains a brief comment on the Alperin Weight Conjecture.

\begin{defn}[$\mathrm{KRC}_p$ and $\mathrm{AWC}_p$]
\label{defn:KR}
The finite group $G$ satisfies the
\begin{itemize}
\item Kn\"orr--Robinson conjecture at $p$ if \cite{knorr_robinson:89,
    thevenaz93Alperin}
\begin{equation*}
  \chi_2(\cat SG{p+*},G) = k(G) - z_p(G), \qquad
  -\rchi_2(\cat SG{p+*},G) = z_p(G)
\end{equation*}
\item  Alperin weight conjecture at $p$ if \cite{alperin87}
  \begin{equation*}
    k_{p'}(G)  = \sum_{[P] \in \cat SG{p+\mathrm{rad}}/G} z_p(N_G(P)/P)
  \end{equation*}
  where the sum runs over the set $\cat SG{p+\mathrm{rad}}/G$ of
  conjugacy classes $[P]$ of $G$-radical $p$-subgroups $P$ of $G$.
\end{itemize}
\end{defn}

The two conjectures are equivalent in the sense that
\cite{knorr_robinson:89} \cite[Thm 3.1]{thevenaz93Alperin}
\begin{equation*}
  \forall H \in \cat SGp \colon \mathrm{KRC}_p(N_G(H)/H) \iff
  \forall H \in \cat SGp \colon \mathrm{AWC}_p(N_G(H)/H)
\end{equation*}
for the given finite group $G$.

It is sometimes possible to verify the Kn\"orr--Robinson conjecture by
machine computations as  we have seen that $\mathrm{KRC}_p(G)$ is
equivalent to any of the following three equivalent conditions
\begin{align}\label{eq:KRCp01}
    \sum_{[x] \in [G]} \ralpha_2(\cat SG{p+*},G)(x) &= -z_p(G)  \\
 \sum_{[C] \in \cat SG{p'+\mathrm{cyc}}/G} \wa_2(\cat SG{p+*},G)(C)
   \frac{\sum_{x \in C-\{1\}} |C_G(x)|)}{|N_G(C)| } 
   &= -z_p(G) \label{eq:KRCp02}
    \\
   \sum_{[A] \in \cat SG{p'+\mathrm{abe}}/G}
   \rchi(C_{\catp SG{*}}(A)) \frac{\varphi_2(A)}{ |N_G(A)|} &= -z_p(G) 
   \label{eq:KRCp03} 
\end{align}
where
\begin{itemize}
\item In Equation~\eqref{eq:KRCp01}, quoting 
Equation~\eqref{eq:charinnerprod},
the value at $x
\in G$ of the reduced class function $\ralpha_2(\cat SG{p+*},G)$ is
\begin{equation*}
  \ralpha_2(\cat SG{p+*},G)(x) = 
  \sum_{[y] \in [C_G(x)]} 
  \frac{\rchi(C_{\cat SG{p+*}}(\gen{x,y}))}{|C_G(\gen{x,y})|} =
  \rchi_1(C_{\cat SG{p+*}}(x),C_G(x)) = 
  \rchi(\sd{C_{\cat SG{p+*}}(x)}/C_G(x))
\end{equation*}
and $\cat SG{p+*}$ can be replaced by the smaller posets $\cat
SG{p+*+\mathrm{rad}}$ of $G$-radical or $\cat SG{p+*+\mathrm{eab}}$ of
elementary abelian $p$-subgroups (Proposition~\ref{prop:SGradeab},
Definition~\ref{defn:classfct}, Equation~\eqref{eq:firstalpha2},
Corollary~\ref{cor:quocomplex})
\item Equation~\eqref{eq:KRCp02}, repeating
  Equation~\eqref{eq:artincoeff}, uses the Artin coefficients,
  $\wa_2(\cat SG{p+*},G)(C)$, introduced in
  Corollary~\ref{cor:OpH}.\eqref{cor:OpH3}
\item In Equation~\eqref{eq:KRCp03} the sum runs over all conjugacy classes
of $p$-regular abelian subgroups $A$ of $G$
(Corollary~\ref{cor:OpH}.\eqref{cor:OpH1}) and $C_{\catp SG*}(A)$ may
be replaced by $C_{\cat SG{p+*+\mathrm{eab}}}(A)$ or $C_{\cat
  SG{p+*+\mathrm{rad}}}(A)$ (Proposition~\ref{prop:SGradeab}). 
\end{itemize}

\begin{exmp}\label{exmp:cyclicsyl}
  $\mathrm{KRC}_p(G)$ is true when $G$ has a cyclic \syl p $S$
\cite[Example 1.4]{thevenaz93Alperin}.  The discrete $G$-poset $\cat
SG{p+*+\mathrm{rad}}$ is the right $G$-set $N_G(S) \backslash G$ of
$G$-conjugates of $S$, the discrete poset $C_{\cat
  SG{p+*+\mathrm{rad}}}(x)$ is the set $C_{N_G(S) \backslash G}(x)$
for any $x \in G$, and (the unreduced form of)
Equation~\eqref{eq:KRCp01} takes the form
\begin{equation*}
    \sum_{[x] \in [G]} |C_{N_G(S)\backslash G}(x)/C_G(x)| = k(G)-z_p(G)
  \end{equation*}
  in this special case.
\end{exmp}

$\mathrm{KRC}_p(G)$ is true also if $O_p(G) \neq 1$ \cite[Example
1.4]{thevenaz93Alperin}. This is because $z_p(G)=0$ by Ito's theorem
\cite[Corollary~53.18]{cr} and $\rchi(C_{\cat SG{p+*}}(A))=0$ for all
$A$ in Equation~\eqref{eq:KRCp03} by Lemma~\ref{lemma:OpHA}.  In
particular the Kn\"orr--Robinson conjecture at $p$ is true for $\GL
n{\F_q}$ when $p$ divides $q-1$ as these groups have a central element
of order $p$. The conjecture also holds for $\GL n{\F_q}$ when $q$ is
a power of $p$ \cite[Proposition~4.1]{thevenaz92poly} and for all
solvable groups \cite[p 195]{thevenaz93Alperin}.



We now use a computer to check
Equations~\eqref{eq:KRCp01}--\eqref{eq:KRCp03} in case of the Mathieu
group $G=M_{11}$. The computations of Example~\ref{exmp:M11} were
carried out using Magma \cite{magma}.

\begin{exmp} \label{exmp:M11}
  The Mathieu group $G=M_{11}$ of order $|G|=7,920 = 2^4 \cdot 3^2
  \cdot 5^1 \cdot 11^1$ contains $k(G)=10$ conjugacy classes of orders
  $(1, 2, 3, 4, 5, 6, 8, 8, 11, 11)$ and $8$ classes of cyclic groups
  of orders $(1, 2, 3, 4, 5, 6, 8, 11)$.  There are $z_p(G)=
  (2,1,5,3)$ irreducible complex representations of $p$-defect $0$ for
  each of the prime divisors $p(G)=(2,3,5,11)$ of the group order.  We
  now verify Equations~\eqref{eq:KRCp01}--\eqref{eq:KRCp03} at all
  primes $p \in p(G)$.

  Table~\ref{tab:alpha2M11} displays the values of the reduced class
  functions $\ralpha_2(\cat SG{p+*},G)$ at each of the $10$ conjugacy
  classes $x$.  The column to the right contains the row sum.  By
  Equation~\eqref{eq:KRCp01}, $\mathrm{KRC}_p(G)$ is true (as it is
  for $p=5, 11$ by Example~\ref{exmp:cyclicsyl}) if and only if this
  column contains $(-z_p(G))_{p \in p(G)}$.
\begin{table}[t]
  \centering 
     \begin{tabular}[t]{>{$}c<{$} | *{10}{>{$}r<{$}} | >{$}c<{$}  }
  |x| &  1 & 2 &  3 &  4 &  5 &  6 &  8 &  8 &  11 &  11  & \Sigma \\ \hline
  \ralpha_2(\cat SG{2+*},G)
  & 0  &   0  &   1  &   0  &  -1  &   0  &   0  &   0  &  -1  &  -1 &
  -2 \\ \noalign{\smallskip}
  \ralpha_2(\cat SG{3+*},G)
  & 0  &   1  &   0  &   1  &  -1  &   0  &   0  &   0  &  -1  &  -1 &
  -1  \\ \noalign{\smallskip}
  \ralpha_2(\cat SG{5+*},G)
  &0  &   0  &  -1  &   1  &   0  &  -1  &  -1  &  -1  &  -1  &  -1 &
  -5\\ \noalign{\smallskip}
  \ralpha_2(\cat SG{11+*},G)
  &0  &  -1  &  -1  &  -1  &   3  &  -1  &  -1  &  -1  &   0  &   0 & -3 \\
 \end{tabular}
  \caption{The reduced class functions 
   $\ralpha_2(\cat SG{p+*},G)$ for $G=M_{11}$ }
  \label{tab:alpha2M11}
\end{table}

Table~\ref{tab:artinM11} shows the reduced Artin coefficients
$\wa_2(\cat SG{p+*},G)(C)$ of Corollary~\ref{cor:OpH}.\eqref{cor:OpH3}
at each of the $8$ classes of cyclic subgroups $C$.  The top row
contains the vector with coordinates $\sum_{x \in C-\{1\}}
|C_G(x)|/|N_G(C)|$ (all nonnegative integers by
Proposition~\ref{prop:CGx/NGH}).  The column to the right shows the
standard inner product of the row with the top row.  For instance, the
second row means that
\begin{equation*}
  \ralpha_2(\cat SG{2+*},G) = 
  \frac{4}{99} \cdot 1_1^G + 
  \frac{1}{12}1_3^G-
  \frac{1}{4}1_5^G-
  \frac{1}{5}1_{11}^G,
  \quad 
  0 \cdot 320 + 1 \cdot 0 + 1 \cdot 1 + 4 \cdot 0 + 1 \cdot (-1)
  + 8 \cdot 0 + 6 \cdot 0 + 2 \cdot (-1) 
 = -2
\end{equation*}
at the prime $p=2$.  (Recall from
Corollary~\ref{cor:OpH}.\eqref{cor:OpH3} that the reduced Euler class
function $\wa_2(\cat SG{p+*},G)$ is the linear combination of the
Artin coefficients at $[C]$ divided by $|N_G(C):C|$ which here equals
$(7920, 24, 12, 4, 4, 2, 2, 5)$.)  By Equation~\eqref{eq:KRCp02},
$\mathrm{KRC}_p(G)$ is true (as it is for $p=5, 11$ by
Example~\ref{exmp:cyclicsyl}) if and only if the rightmost column
contains $(-z_p(G))_{p \in p(G)}$.
\begin{table}[t]
  \centering
     \begin{tabular}[t]{ >{$}c<{$} | *{8}{>{$}r<{$}} |  >{$}c<{$}  }
|C| &
   1  &    2  &    3  &    4  &    5  &    6  &    8  &   11 & \cdot \\
\hline\noalign{\smallskip}
\tiny{\sum_{x \neq 1} |C_G(x)|/|N_G(C)|}
{}& 
0 &  1 &   1 &  4 &   1 &  8 &  6 &  2 & {}\\
\hline\noalign{\smallskip}
\wa_2(\cat SG{2+*},G) &
  320  &    0  &    1  &    0  &   -1  &    0  &    0  &   -1 & -2 \\
  \noalign{\smallskip} 
\wa_2(\cat SG{3+*},G) &
  375  &   -2  &    0  &    1  &   -1  &    0  &    0  &   -1 & -1\\
  \noalign{\smallskip} 
\wa_2(\cat SG{5+*},G) &
 -296  &    1  &    2  &    2  &    0  &   -1  &   -1  &   -1 & -5 \\
 \noalign{\smallskip} 
\wa_2(\cat SG{11+*},G) &
-1463  &    6  &    2  &    0  &    3  &   -1  &   -1  &    0 & -3
\end{tabular}
  \caption{The reduced Artin coefficients 
    $\wa_2(\cat SG{p+*},G)$ for $G=M_{11}$}
  \label{tab:artinM11}
\end{table}

Thirdly, the four tables of Figure~\ref{fig:tabkrcp03} list, for each
prime divisor $p \in p(G)$, the $p$-part of the order of $C_G(A)$, the
\Euc\ $-\rchi(C_{\cat S{G}{p+*}}(A))$ of the $A$-centralized subposet,
the number $\varphi_2(A)$ of generating pairs of $A$, the length $|G :
N_{G}(A)|$ of $A$, and, in the bottom row, the product of these last
three numbers for each conjugacy class of abelian $p$-regular
subgroups $A$ of $G$; $\mathrm{KRC}_p(G)$ is true if and only if the
sum of the integers of the bottom row equals $z_p(G)|G|$.  (The row
containing $|C_G(A)|_p$ plays no role here but serves to illustrate
the divisibility statement of
Theorem~\ref{thm:global}.\eqref{thm:global2}.)

\begin{figure}[t]
  \centering
   \begin{tabular}[t]{>{$}c<{$} | *{5}{>{$}r<{$}} }
     A & 1 & 3 & 5 & 11 & 3 \times 3 \\ \hline\noalign{\smallskip}
     |C_G(A)|_2 & 16 & 2 & 1 & 1 & 1 \\ \hline\noalign{\smallskip}
     -\rchi(C_{\cat S{G}{2+*}}(A)) &496 &-8 &1 &1 &1 \\
     \grp A &1 &8 &24 &120 &48 \\
     |G : N_G(A)| &1 &220 &396 &144 &55 \\ \hline\noalign{\smallskip}
     &496 &-14080 &9504 &17280 &2640
   \end{tabular}

   \begin{tabular}[t]{>{$}c<{$} | *{7}{>{$}r<{$}} }
     A & 1 & 2 & 5 & 11  & 4 & 8  & 2 \times 2 \\ \hline\noalign{\smallskip}
    |C_G(A)|_3 & 9 & 3 & 1 & 1 & 1 & 1 & 1 \\ \hline\noalign{\smallskip}
     -\rchi(C_{\cat S{G}{3+*}}(A)) &-54 &-6 &1 &1 &-2 &0 &-2 \\
     \grp A &1 &3 &24 &120 &12 &48 &6 \\
     |G : N_G(A)| &1 &165 &396 &144 &495 &495 &330 \\ 
     \hline\noalign{\smallskip}
     &-54 &-2970 &9504 &17280 &-11880 &0 &-3960
   \end{tabular}

   \begin{tabular}[t]{>{$}c<{$} | *{9}{>{$}r<{$}} }
     A & 1 & 2 &3   &4 &6  &8 &11 &2 \times 2 &3 \times 3 \\ \hline
     \noalign{\smallskip}
     |C_G(A)|_5 & 5 & 1 & 1 & 1 & 1 & 1 & 1 & 1 & 1 \\
     \hline\noalign{\smallskip} 
     -\rchi(C_{\cat S{G}{5+*}}(A)) &-395 &-11 &1  &-3 &1 &1 &1 &1 &1 \\
     \grp A &1 &3 &8  &12 &24 &48 &120 &6 &48 \\
     |G : N_G(A)| &1 &165 &220  &495 &660 &495 &144 &330 &55 \\ 
     \hline\noalign{\smallskip}
     &-395 &-5445 &1760  &-17820 &15840 &23760 &17280 &1980 &2640
   \end{tabular}

   \begin{tabular}[t]{>{$}c<{$} | *{9}{>{$}r<{$}} }
     A &1 &2 &3 &4 &5 &6 &8 &2 \times 2 &3 \times 3 \\ \hline
     \noalign{\smallskip}
     |C_G(A)|_{11} & 11 & 1 & 1 & 1 & 1 & 1 & 1 & 1 & 1 \\
     \hline\noalign{\smallskip} 
     -\rchi(C_{\cat S{G}{11+*}}(A)) &-143 &1 &1 &1 &-3 &1 &1 &1 &1 \\
     \grp A &1 &3 &8  &12 &24 &24 &48 &6 &48 \\
     |G : N_G(A)| &1 &165 &220  &495 &396 &660 &495 &330 &55 \\ 
     \hline\noalign{\smallskip}
     &-143 &495 &1760  &5940 &-28512 &15840 &23760 &1980 &-640
   \end{tabular}
  \caption{Tables for Equation~\eqref{eq:KRCp03} for $G=M_{11}$ 
  at $p=2,3,5,11$}
  \label{fig:tabkrcp03}
\end{figure}

   In each of these four tables the sum of the numbers of the bottom
   row is $z_{p}(M_{11}) |M_{11}|$ and thus
   Equation~\eqref{eq:KRCp02} holds for $G=M_{11}$ for all primes $p
   \in p(G)$.
   \end{exmp}

   Similar machine computations demonstrate that $\mathrm{KRC}_p(G)$
   is true when $G=M_{11}, M_{12}, M_{22}, M_{23}, M_{24}$ is one of
   the five Mathieu groups or $G=J_1,J_2,J_3$ one of the three first
   Janko groups and $p$ is any prime divisor of $|G|$.

\section{The orbit category $\pi$-subgroups}
\label{sec:OGpi}

Let $\pi$ is a nonempty {\em set\/} of primes.  In this short section
we consider the orbit category $\cat OG\pi$ of $\pi$-singular
subgroups of $G$. The arguments here will be minor variations of the
ones used in the case considered above where $\pi$ consisted of a
single prime.

A left ideal in an $\mathrm{EI}$-category $\cat C{}{}$ is a full
subcategory $\cat I{}{}$ of $\cat C{}{}$ if $a \in \cat I{}{},\; \cat
C{}{}(a,b) \neq \emptyset \implies b \in \cat I{}{}$ for all objects
$a$, $b$ of $\cat C{}{}$. Similarly, a right ideal is a full
subcategory $\cat J{}{}$ of $\cat C{}{}$ if $b \in \cat J{}{},\; \cat
C{}{}(a,b) \neq \emptyset \implies a \in \cat J{}{}$ for all objects
$a$, $b$ of $\cat C{}{}$.
If $\cat I{}{}$ be a left and $\cat J{}{}$ a
right ideal in $\cat C{}{}$. Then
\begin{equation}\label{eq:lrideal}
  \chi(\cat C{}{}) - \chi(\cat I{}{}) = 
  \sum_{a \in \cat C{}{} - \cat I{}{}} k^a, \qquad
  \chi(\cat C{}{}) - \chi(\cat J{}{}) = 
  \sum_{a \in \cat C{}{} - \cat J{}{}} k_b
\end{equation}
where $k^\bullet$ and $k_\bullet$ are the unique iso\m\ invariant \we\
and co\we\ on $\cat C{}{}$ \cite[Theorem 3.7]{gm:2012}.
If the \we\ is concentrated on $\cat I{}{}$ then 
 $\chi(\cat C{}{}) = \chi(\cat I{}{})$ and if the co\we\ 
 is concentrated on $\cat J{}{}$ then 
 $\chi(\cat C{}{}) = \chi(\cat J{}{})$.

With more sophisticated methods it is possible to generalize
Corollary~\ref{cor:cowtOV}.

\begin{lemma}\label{lemma:chiOK}
  For any finite group $K$
  \begin{equation*}
    -\rchi(\cat OK{[1,K)}) =
    \begin{cases}
      \phi(|K|)/|K| & \text{$K$ cyclic} \\
      0 & \text{$K$ noncyclic}
    \end{cases}
  \end{equation*}
  where $\phi$ is Eulers $\phi$-function.
\end{lemma}
\begin{proof}
  The co\we\ on the contractible category $\cat OK{}$ is $k_B =
  |K|^{-1} \sum_{A \leq B}|A|\mu(A,B)$
  \cite[Theorem~2.18.(3)]{jmm_mwj:2010}. The subcategory $\cat
  OK{[1,K)}$ of proper subgroups is a right ideal so by
  Equations~\eqref{eq:lrideal}
  \begin{equation*}
    1 = \chi(\cat OK{}) = \chi(\cat OK{[1,K)}) + k_K
  \end{equation*}
  When $K$ is not cyclic then $k_K=0$ by
  \cite[Proposition~2.8]{kratzer_thevenaz85} and the formula follows
  in this case.

  Next, assume that $K$ is cyclic. For any subgroup $H$ of $K$,
  $H//\cat OK{[1,K)}$ is the interval $(1,K/H)$ in the poset $\cat
  SK{}$,  $-\rchi(H//\cat
  OK{[1,K)}) = -\rchi(1,K/H) = -\mu(|K|/|H|)$   and
  \begin{equation*}
    -|K|\rchi(\cat OK{[1,K)}) = |K|-|K|\chi(\cat OK{[1,K)}) =
    \sum_{1 \leq H \leq K} \mu(|K|/|H|)|H| = \phi(|K|) 
  \end{equation*}
  by classical \Mb\ inversion.
\end{proof}

\begin{lemma}\label{lemma:coweOGpi}
  The function
  \begin{equation*}
    k_K = \frac{-\rchi(\cat OK{[1,K)})}{|G:K|}
  \end{equation*}
  is a co\we\ for $\cat OG\pi$ and the \Euc\ 
  \begin{equation*}
    \chi(\cat OG\pi) = \frac{|G_\pi|}{|G|}
  \end{equation*}
  is the density of the $\pi$-singular elements in $G$.
\end{lemma}
\begin{proof}
  As we saw in the proof of Proposition~\ref{prop:coweOGp}, the
  category $\cat OG\pi//K = \cat OG{}//K$ is equivalent to $\cat
  OK{[1,K)}$ for any object $K$ of $\cat OG\pi$.  This gives the
  co\we\ by \cite[Theorem 3.7]{gm:2012} as $|\cat OG\pi(K)||[K]| =
  |N_G(K):K| |G : N_G(K)| = |G : K|$.  Now apply
  Lemma~\ref{lemma:chiOK} to compute the \Euc\ as the sum of the
  values of the co\we .
\end{proof}

\begin{lemma}\label{lemma:weOGpi}
   The function
  \begin{equation*}
    k^H = \frac{-\rchi(H// \cat SG{\pi})}{|G:H|}
  \end{equation*}
  is a \we\ for $\cat OG\pi$ and the \Euc\ is
  \begin{equation*}
    \chi(\cat OG\pi) = 
    \sum_{H \in \cat SG\pi} 
    \frac{-\rchi(H// \cat SG{\pi})}{|G:H|} =
    \sum_{[H] \in \cat SG\pi/G} 
    \frac{-\rchi(H// \cat SG{\pi})}{|N_G(H):H|}
  \end{equation*}
\end{lemma}
\begin{proof}
  The objects of the category $H//\cat OG\pi$ are left cosets $gK$
  where $g$ is an element of $G$, $K$ is a $\pi$-singular subgroup of
  $G$, and $H \leq {}^gK$. The functor $H//\cat OG\pi \to H// \cat
  SG{\pi}$ taking $gK$ to ${}^gK$ is an equivalence of categories
  since it is surjective on objects and bijective on \m\ sets as $(H
  // \cat OG\pi)(g_1K_1,g_2K_2) = \cat
  SG\pi({}^{g_1}K_1,{}^{g_2}K_2)$. This gives the \we\ on $\cat OG\pi$
  by \cite[Theorem 3.7]{gm:2012}.
\end{proof}

\begin{lemma}\label{lemma:weSGpi}
  The function $k^H = -\rchi(H// \cat SG{\pi}) = \sum_{H \leq K \in
    \cat SG\pi} \mu(H,K)$ is the \we\ on $\cat SG{\pi}$.
\end{lemma}
\begin{proof}
  The \Mb\ function $\mu$ for $\cat SG{}$ restricts to the \Mb\
  function for the right ideal $\cat SG{\pi}$. Thus the \we\ is $k^H =
  \sum_{H \leq K \in \cat SG\pi} \mu(H,K)$. The \we\ may also be
  expressed as $-\chi(H// \cat SG{\pi})$ \cite[Theorem 3.7]{gm:2012}.
\end{proof}

We can now generalize Equations~\eqref{eq:affine} and
\eqref{eq:globalnoneq} to sets $\pi$ of maybe more than one prime.

\begin{cor}\label{cor:chiOGpi}
  For any finite group $G$, any set $\pi$ of primes, and any
  $\pi$-subgroup $H \in \cat SG\pi$ of $G$,
  \begin{equation*}
  \sum_{K \in H / \cat SG{\pi}} -\rchi( K // \cat SG{\pi}) = 1, \qquad
  \sum_{H \in \cat SG{\pi}} -\rchi(H// \cat SG{\pi})|H| = |G_\pi|
\end{equation*}
\end{cor}
\begin{proof}
  The \we\ for $\cat SG\pi$ restricts to a \we\ for the contractible
  left ideal $H / \cat SG{\pi}$. Thus $1= \sum_{K \in H / \cat SG\pi }
  k^K = \sum_{K \in H / \cat SG\pi } -\rchi(K // \cat SG{\pi}) =
  \sum_{[K] \in \cat SG\pi/K } \cat SG\pi(H,[K])(-\rchi(K // \cat
  SG{\pi}))$.  This proves the first equation.  Combine
  Lemmas~\ref{lemma:coweOGpi} and \ref{lemma:weOGpi} to prove the
  second equation by computing the \Euc\ of $\cat OG\pi$ in two
  different ways.
\end{proof}

Corollary~\ref{cor:chiOGpi} shows that the linear equation
\begin{equation}\label{eq:SGH[K]pi}
  \begin{pmatrix}
    \cat SG\pi(H,[K])
  \end{pmatrix}_{H,[K] \in \cat SG\pi/G}
\begin{pmatrix}
  \vdots \\ -\rchi(K // \cat SG\pi) \\ \vdots
\end{pmatrix}_{K \in \cat SG\pi/G} =
\begin{pmatrix}
  \vdots \\ 1 \\ \vdots
\end{pmatrix}
\end{equation}
determines the \we\ for $k^{[K]} = -\rchi(K//\cat SG\pi) \colon \cat
SG\pi/G \to \Z$ for $\cat SG\pi/G$ (Definition~\ref{defn:wecoweSA}).
The next proposition offers a sufficient condition for the vanishing
of $k^{[K]}$.

\begin{prop}\label{prop:zerowe}
  If $H$ is a \syl\pi\ in $G$ then $H// \cat SG{\pi}$ is empty and
  $-\rchi(H// \cat SG{\pi}) = 1$.  If $H$ is not a \syl\pi\ of $O(G,H)
  = \bigcap_{K \in H// \cat SG{\pi}} N_G(K)$ then $H // \cat SG{\pi}$
  is contractible and $-\rchi(H// \cat SG{\pi})=0$.
\end{prop}
\begin{proof}
  The first assertion is clear. In the second case, the assumption is
  that $G$ possesses a $\pi$-subgroup $L>H$ normalizing all
  $\pi$-subgroups $K \geq H$. But then $K \leq KL \geq L$ is a
  contraction of $H// \cat SG{\pi}$.
\end{proof}

\begin{exmp}\label{exmp:chiSGL32pi} 
  The simple group $G=\GL 3{\F_2}$ of order $168$ contains $12$
  conjugacy classes of $\pi$-subgroups where $\pi=\{2,3\}$.  The table
  \begin{center}
  \begin{tabular}[t]{>{$}c<{$} | *{12}{>{$}c<{$}} }    
     |H|  &
      24 & 24 & 12 & 12 & 8 & 6 & 4 & 4 & 4 & 3 & 2 & 1 \\ 
        \hline \noalign{\smallskip}
     -\rchi(H // \cat S{G}{\pi}) &
   1  & 1 &   0 &  0 &  -1 &  -1 &  0 &  0 &  0 &  0 &  4 & -48 \\
   |G:N_G(H)| &
   7  & 7  & 7  & 7  &21  &28  & 7  & 7  &21  &28  &21  & 1 \\
   |N_G(H) : H|_\pi & 
   1 & 1 & 2 & 2 & 1 & 1 & 6 & 6 & 2 & 2 & 4 &24 
 \end{tabular} 
 \end{center}
 shows the orders and lengths of the conjugacy classes of
 $\pi$-subgroups $H$ of $G$ together with the \we\ values
 $k^H=-\rchi(H // \cat S{G}{\pi})$ and the $\pi$-part of $|N_G(H):H|$.
 The criterion of Proposition~\ref{prop:zerowe} accounts for four of
 the six $0$s in the first row of the table and for the two $1$s at
 the \syl\pi s (abstractly isomorphic to $\Sigma_4$).  These \Euc s
 were determined as the solution to the linear
 equation~\eqref{eq:SGH[K]pi}
\begin{equation*}
\left(\begin{array}{*{12}{c}}  
 1  & 0  & 0  & 0  & 0  & 0  & 0  & 0  & 0  & 0  & 0  & 0 \\
 0  & 1  & 0  & 0  & 0  & 0  & 0  & 0  & 0  & 0  & 0  & 0 \\
 0  & 1  & 1  & 0  & 0  & 0  & 0  & 0  & 0  & 0  & 0  & 0 \\
 1  & 0  & 0  & 1  & 0  & 0  & 0  & 0  & 0  & 0  & 0  & 0 \\
 1  & 1  & 0  & 0  & 1  & 0  & 0  & 0  & 0  & 0  & 0  & 0 \\
 1  & 1  & 0  & 0  & 0  & 1  & 0  & 0  & 0  & 0  & 0  & 0 \\
 1  & 3  & 0  & 1  & 3  & 0  & 1  & 0  & 0  & 0  & 0  & 0 \\
 3  & 1  & 1  & 0  & 3  & 0  & 0  & 1  & 0  & 0  & 0  & 0 \\
 1  & 1  & 0  & 0  & 1  & 0  & 0  & 0  & 1  & 0  & 0  & 0 \\
 1  & 1  & 1  & 1  & 0  & 1  & 0  & 0  & 0  & 1  & 0  & 0 \\
 3  & 3  & 1  & 1  & 5  & 4  & 1  & 1  & 1  & 0  & 1  & 0 \\
 7  & 7  & 7  & 7  &21  &28  & 7  & 7  &21  &28  &21  & 1
\end{array}\right)
\left(\begin{array}{r}
   1 \\  1 \\   0  \\ 0 \\ -1 \\  -1 \\  0 \\  0 \\  0 \\  0 \\  4 \\ -48
\end{array}\right) =
\begin{pmatrix}
  1 \\ 1 \\ 1 \\ 1 \\ 1 \\ 1 \\ 1 \\ 1 \\ 1 \\ 1 \\ 1 \\ 1 
\end{pmatrix}
\end{equation*}
where the $(12 \times 12)$-matrix is the table $\left( \cat
  SG{\pi}(H,[K]) \right)_{H,K \in \cat SG{\pi+\mathrm{rad}}/G}$.  In
particular, $-\rchi(1//\cat SG\pi) = -\rchi(\cat SG{\pi+*}) = -48$ and
$\chi(\cat SG{\pi+*}) = 1-(-48) = 49$.  Since
\begin{equation*}
  \sum_{[H] \in \cat SG{\pi}/G}
  -\rchi(H // \cat SG{\pi})|G:N_G(H)||H| = 1 \cdot 7 \cdot 24 + \cdots
  + (-48) \cdot 1 \cdot 1 = 120 
\end{equation*}
and $G$ indeed contains $120$ $\pi$-singular elements the numbers
of this table agrees with Corollary~\ref{cor:chiOGpi}. 
 \end{exmp}

An even more general form of
the second formula of Corollary~\ref{cor:chiOGpi} was first proved in 
\cite[Theorem 6.3]{HIO89}. 
In the table of Example~\ref{exmp:chiSGL32pi} it happens that
$-\rchi(H//\cat SG\pi)$ is always divisible by the $\pi$-part of
$|N_G(H):H|$. This is no concidence as shown by the following
generalization of Brown's theorem. 

\begin{thm}\cite[Corollary~3.9]{HIO89}
  $|N_G(H) : H |_{\pi} \mid -\rchi(H// \cat SG{\pi})$ for any
  $\pi$-subgroup $H$ of $G$.
\end{thm}

 When $\pi$ consists of just one prime the \Euc\ of $H// \cat SG{\pi}$
 is a function of the quotient group $N_G(H)/H$
 (Proposition~\ref{prop:weOGp}) but this is not true when $\pi$ has
 more than one element.


\section*{Acknowledgements}

I would like to thank Radha Kessar and J{\o}rn B{\o}rling Olsson for
valuable help and Sune Precht Reeh for a crucial remark.  Comments
from contributors to a post on the internet site {\em mathoverflow}
were also helpful.


\def\cprime{$'$} \def\cprime{$'$} \def\cprime{$'$} \def\cprime{$'$}
  \def\cprime{$'$}
\providecommand{\bysame}{\leavevmode\hbox to3em{\hrulefill}\thinspace}
\providecommand{\MR}{\relax\ifhmode\unskip\space\fi MR }
\providecommand{\MRhref}[2]{%
  \href{http://www.ams.org/mathscinet-getitem?mr=#1}{#2}
}
\providecommand{\href}[2]{#2}

\end{document}